\documentclass[10pt,reqno]{amsart}
\usepackage{fullpage, amssymb,mathrsfs,color}
\usepackage{pinlabel}

\usepackage{amsmath, amsfonts, amsthm, verbatim}
\usepackage{epstopdf, mathtools}
\mathtoolsset{showonlyrefs}
\usepackage{color}
\usepackage{esint}

\usepackage{empheq}
\usepackage[shortlabels]{enumitem}
\usepackage[T1]{fontenc}

\newcommand{\rar}{\rightarrow}

\newcommand{\cl}{\mathcal}

\def\cZ {\mathcal{Z}}

\newcommand{\tx}[1]{\mathrm{#1}}

\newcommand{\uln}[1]{{\underline{ #1 }}}

\definecolor{deepgreen}{cmyk}{1,0,1,0.5}

\newcommand{\A}{\mathcal{A}}

\newcommand{\E}{\mathcal{E}}

\newcommand{\HH}{\mathcal{H}}

\newcommand{\LL}{\mathcal{L}}

\newcommand{\KK}{\mathcal{K}}

\newcommand{\ZZ}{\mathcal{Z}}

\newcommand{\N}{\mathbb{N}}
\newcommand{\R}{\mathbb{R}}

\newcommand{\Z}{\mathbb{Z}}

\newcommand{\al}{\alpha}
\newcommand{\be}{\beta}
\newcommand{\ga}{\gamma}
\newcommand{\de}{\delta}
\newcommand{\e}{\epsilon}
\newcommand{\fy}{\varphi}
\newcommand{\om}{\omega}
\newcommand{\la}{\lambda}
\newcommand{\te}{\theta}

\newcommand{\s}{\sigma}

\newcommand{\De}{\Delta}
\newcommand{\Om}{\Omega}

\newcommand{\La}{\Lambda}

\newcommand{\p}{\partial}

\makeatletter

\newcommand{\Rmnum}[1]{\expandafter\@slowromancap\romannumeral #1@}
\makeatother

\newcommand{\I}{\infty}

\newcommand{\ti}{\widetilde}

\newcommand{\ang}[1]{\left\langle{#1}\right\rangle}
\newcommand{\abs}[1]{\left\lvert{#1}\right\rvert}

\newcommand{\EQ}[1]{\begin{equation}\begin{split} #1 \end{split}\end{equation}}
\newcommand{\pmat}[1]{\begin{pmatrix} #1 \end{pmatrix}}

\newcommand{\Del}[1]{}

\numberwithin{equation}{section}

\newtheorem{thm}{Theorem}[section]
\newtheorem{cor}[thm]{Corollary}
\newtheorem{lem}[thm]{Lemma}
\newtheorem{prop}[thm]{Proposition}

\newtheorem{claim}[thm]{Claim}
\theoremstyle{remark}
\newtheorem{rem}[thm]{Remark}

\newcommand{\mand}{{\ \ \text{and} \ \  }}

\newcommand{\mif}{{\ \ \text{if} \ \ }}

\newcommand{\mas}{{\ \ \text{as} \ \ }}

\newcommand{\ula}{\underline{\lambda}}
\newcommand{\umu}{\underline{\mu}}

\usepackage{tensor}

\usepackage{graphicx}
\usepackage{xcolor} % A package to add color.
\usepackage{tensor}
\usepackage{slashed}
\usepackage{cite}
\definecolor{green}{rgb}{0,0.8,0} % Redefines the color green.
\newcommand{\eps}{\epsilon}

\newcommand{\bfd}{{\bf d}}

\newcommand{\bbS}{\mathbb S}

\newcommand{\calK}{\mathcal K}

\begin{document}

\title[Threshold corotational wave maps]{Threshold dynamics for corotational wave maps}

\author{Casey Rodriguez}

\begin{abstract}
We study the dynamics of corotational wave maps from $\R^{1+2} \rar  \mathbb S^2$ at threshold energy.  It is known that topologically trivial wave maps with energy $< 8\pi$ are global and scatter to a constant map.  In this work, we prove that a corotational wave map with energy equal to $8\pi$ is globally defined and scatters in one time direction, and in the other time direction, either the map is globally defined and scatters, or the map breaks down in finite time and converges to a superposition of two harmonic maps.  The latter behavior stands in stark contrast to higher equivariant wave maps with threshold energy which have been proven to be globally defined for all time. Using techniques developed in this paper, we also construct a corotational wave map with energy $= 8\pi$ which blows up in finite time. The blow-up solution we construct provides the first example of a minimal topologically trivial non-dispersing solution to the full wave map evolution.  
%Comments on proof
\end{abstract}

\maketitle

\section{Introduction}

\subsection{Wave maps}

In this paper we study the dynamics of energy critical wave maps which are defined as follows.  Let $\eta$ be the Minkowski metric on $\R^{1+2}_{t,x}$, and let $\cl N$ be a Riemannian manifold with metric $h$.  A map $u:\R^{1+2} \rar \cl N$ is a \emph{wave map} if it is a critical point of the action 
\begin{align*}
\cl A(u) = \frac{1}{2} \int_{\R^{1+2}} \langle \p^\mu u , 
\p_\mu u \rangle_{h} \, dx dt,
\end{align*}
where we raise and lower indices using the Minkowski metric $\eta$.  
The associated Euler-Lagrange equations are the \emph{wave maps equations} given in local coordinates by 
\begin{align}\label{wm}
\p^\mu \p_\mu u^a + \Gamma^a_{bc}(u) \p^\mu u^b \p_\mu u^c = 0.
\end{align}
Here the $\Gamma^a_{bc}$ are the Christoffel symbols associated to the metric $h$ on $\cl N$. The time translational symmetry of Minkowski space and Noether's theorem provide a conserved energy for the evolution
\begin{align}\label{energy}
\cl E(u(t),\p_t u(t)) := \frac{1}{2} \int_{\R^2} |\p_t u(t,x)|^2_h + |\nabla u(t,x)|_h^2 \, dx
= \mbox{const.}   
\end{align}
We study wave maps as solutions to the Cauchy problem \eqref{wm} with prescribed finite energy initial data $\vec u(0) = (u_0, u_1)$ where 
\begin{align*}
u_0(x) \in \cl N, \quad u_1(x) \in T_{u_0(x)} \cl N, \quad x \in \R^2.  
\end{align*}
Here and throughout the paper we use the notation $\vec u(t)$ to denote the pair of functions 
\begin{align*}
\vec u(t) := (u(t,\cdot), \p_t u(t,\cdot)).  
\end{align*}
We also assume that there exists $u_\infty \in \cl N$ such that 
\begin{align}\label{eq:behav_at_infinity}
u_0(x) \rar u_\infty \mbox{ as } |x| \rar \infty. 
\end{align}
Due to the conformal symmetry of Minkowski space, we also have the following scaling symmetry: if $\vec u(t)$ is a wave map and $\la > 0$, then 
\begin{align}\label{scale}
\vec u_\la(t,x) = (u_\la(t,x), \p_t u_\la(t,x)) := \left (
u \Bigl ( \frac{t}{\la},  \frac{x}{\la} \Bigr ),
\frac{1}{\la} \p_t u \Bigl ( \frac{t}{\la},  \frac{x}{\la} \Bigr )
\right )
\end{align}
is also a wave map.  The energy is scale invariant, 
\begin{align*}
\cl E (\vec u_\la) = \cl E(\vec u),
\end{align*}
and for this reason, the wave maps equations in (1+2)-dimensions are said to be \emph{energy critical}.  Wave maps have been extensively studied over the past several decades, and we refer the reader to \cite{SSbook} and \cite{GG} for reviews of the work that has been done. 

In this work we  specialize to the case $\cl N = \bbS^2$ (with the usual round metric) and wave maps which respect the rotational symmetry of the background and target.  More precisely, we fix an origin in $\R^2$ and north pole $N \in \bbS^2$. We say a map $u : \R^{1+2} \rar \bbS^2$ is \emph{corotational} or $1$-\emph{equivariant} if $u \circ \rho = \rho \circ u$ for all $\rho \in SO(2)$.  Here $\rho$ acts on $\bbS^2$ by rotation about the axis determined by $N$. Choosing $N = (0,0,1)$ without loss generality, we can write a corotational map as   
\begin{align}\label{equiv}
u(t,r,\theta) = (\sin \psi(t,r) \cos \theta, \sin \psi(t,r) \sin \theta, \cos \psi(t,r)) \in \bbS^2 \subset \R^3, 
\end{align}  
where $(t,r,\theta)$ are polar coordinates on $\R^{1+2}$, and $(\psi,\theta)$ are spherical coordinates on $\bbS^2$.  For corotational maps, the Cauchy problem \eqref{wm} reduces to a single equation for the azimuth angle $\psi = \psi(t,r)$:   
\EQ{ \label{eq:wmk}
\begin{aligned}
&\p_t^2 \psi -  \p_r^2 \psi  - \frac{1}{r} \p_r \psi +  \frac{\sin 2 \psi}{2r^2} = 0, \\ 
&\vec \psi(0) = (\psi_0, \psi_1),
\end{aligned}
}
The conserved energy \eqref{energy} is given by 
\EQ{ \label{eq:en} 
 \E( \vec \psi(t) ) = \pi \int_0^\I  \left( (\p_t \psi (t, r))^2  +  ( \p_r \psi(t, r))^2 + \frac{\sin^2\psi(t, r)}{r^2}  \right) rdr,
 } 
and the scaling symmetry of the equation \eqref{scale} is given by
\begin{align}
\vec \psi_\la(t,r) := \left (
\psi \Bigl ( \frac{t}{\la},  \frac{r}{\la} \Bigr ),
\frac{1}{\la} \p_t \psi \Bigl ( \frac{t}{\la},  \frac{r}{\la} \Bigr )
\right ).
\end{align}
The expression for the energy implies that there exists $m,n \in \Z$ such that 
$
\lim_{r \rar 0} \psi_0(r) = m\pi$ and $\lim_{r \rar \infty} \psi_0(r) = n\pi. 
$
By continuity of the flow $\vec \psi(t)$, 
\begin{align*}
\lim_{r \rar 0} \psi(t,r) = m\pi, \quad \lim_{r \rar \infty} \psi(t,r) 
= n\pi, \quad \forall t.  
\end{align*}
Without loss of generality, we may assume that $m = 0$ and $n \in \N \cup \{0\}$.
Thus, finite energy solutions to \eqref{eq:wmk} are split into disjoint classes given by  
\EQ{\label{eq:Hn}
	\HH_n := \{ (\psi_0, \psi_1) \mid \E(\psi_0, \psi_1) < \infty \mand \lim_{r \to 0}\psi_0(r) =0, \, \lim_{r \to \infty}\psi_0(r) = n\pi\}. 
}
The parameter $n \in \N \cup \{0\}$ we refer to as the \emph{degree} of the map, and it can be thought of as parameterizing the minimal number of times the map $\psi(t)$ (more precisely, $u(t)$ given by \eqref{equiv}) wraps $\R^2$ around the sphere. We study those corotational initial data $(\psi_0, \psi_1) \in \cl H_0$, i.e. which satisfy
\begin{align*}
\lim_{r \rar 0} \psi_0(r) = \lim_{r \rar \infty} \psi_0(r) = 0. 
\end{align*}  

A corotational ansatz reduces the complexity of the wave maps equations greatly and is possible in the more general case when $\cl N$ is a surface of revolution.  Choosing $\cl N = \bbS^2$ is motivated by what is known about stationary wave maps, or \emph{harmonic maps}, in this setting. By an ODE argument, the unique (up to scaling) nontrivial corotational harmonic map is given explicitly by 
\begin{align*}
Q(r) = 2 \arctan r,
\end{align*} 
with energy
\begin{align*}
\cl E(\vec Q) = 4\pi. 
\end{align*}
We note that 
\begin{align*}
\lim_{r \rar 0} Q(r) = 0, \quad \lim_{r \rar \infty} Q(r) = \pi, 
\end{align*}
so that $\vec Q \in \cl H_1$.  In fact, it can be shown that $Q$ minimizes the energy in $\cl H_1$ (see Section 2). As we will soon discuss, these harmonic maps play a fundamental role in the long time dynamics of wave maps with large initial data. 

We conclude this subsection by discussing $k$-equivariant maps, a generalization of our corotational reduction.  For $k \in \N$, we say a map $u : \R^{1+2} \rar \bbS^2$ is \emph{k-equivariant} if $u \circ \rho = \rho^k \circ u$ for all $\rho \in SO(2)$ where $SO(2)$ acts on the $\R^{1+2}$ and $\bbS^2$ as before.  Then we may write $$u(t,r,\theta) = (\sin \psi(t,r) \cos k\theta, \sin \psi(t,r) \sin k\theta, \cos \psi(t,r))$$ and the wave maps equations reduce to the single equation 
\EQ{ \label{eq:equiv}
	\begin{aligned}
		&\p_t^2 \psi -  \p_r^2 \psi  - \frac{1}{r} \p_r \psi +  k^2 \frac{\sin 2 \psi}{2r^2} = 0, \\ 
		&\vec \psi(0) = (\psi_0, \psi_1),
	\end{aligned}
}
The conserved energy \eqref{energy} is given by $\E^k( \vec \psi(t) ) = \pi \int_0^\I  \left( (\p_t \psi (t, r))^2  +  ( \p_r \psi(t, r))^2 + k^2 \frac{\sin^2\psi(t, r)}{r^2}  \right) rdr$.
As in the corotational setting, the unique (up to scaling) nontrivial $k$-equivariant harmonic map is given by 
\begin{align*}
Q^k(r) = 2 \arctan (r^k). 
\end{align*}
The harmonic map $\vec Q^k \in \cl H_1$, $\E^k(\vec Q^k) = 4 \pi k$ and $\vec Q^k$ minimizes the energy $\E^k(\cdot)$ in the class $\cl H_1$.  In particular, the corotational harmonic map $Q = Q^1$ has the least energy of all nontrivial equivariant harmonic maps.   

We now turn to motivating our main results. 

\subsection{History and motivation}
Strichartz estimates suffice to prove global existence for equivariant wave maps evolving from small degree-0 data (see Section 2), so recent work has been dedicated to understanding the long-time dynamics of wave maps evolving from large initial data. It is here that the family of harmonic maps play a fundamental role.  Indeed, a classical result of Struwe \cite{Struwe} states that if a smooth $k$-equivariant wave map $\vec \psi(t)$ breaks down at time $t = 1$, say, then $\vec \psi(t)$ converges to the harmonic map $\vec Q^k$ in a local spacetime norm.  Moreover, $\vec \psi(t,r)$ must concentrate energy in excess of $\E^k(\vec Q^k)$ at the tip of the inverted light cone centered at $(T_+,r) = (1,0)$.  Thus, a $k$-equivariant wave map $\vec \psi(t)$ with energy less that $\E^k(\vec Q_k)$ is globally defined and smooth.   
 The works by Krieger, Schlag, Tataru~\cite{KST}, Rodnianski, Sterbenz~\cite{RS}, and Rapha\"el, Rodnianski~\cite{RR} constructed examples of degree-1 wave maps that blow-up by bubbling off a harmonic map, i.e. 
 \begin{align*}
 \vec \psi(t) = \vec Q^k_{\la(t)} + \vec \varphi(t),
 \end{align*}
 with $\la(t) \rar 0$ as $t \rar T_+ < \infty$ and $\varphi(t)$ regular up to $t = T_+$.  
 
As we've discussed, harmonic maps play a key role in singularity formation for wave maps, but in fact they should be fundamental in describing the dynamics of \emph{arbitrary} wave maps.  Indeed, according to the \emph{soliton resolution conjecture}, one expects the following beautiful simplification of the dynamics: smooth wave maps asymptotically break up into a sum of dynamically rescaled harmonic maps and a free radiation term (a solution to the linearized equations).  
The problem of describing the dynamics of corotational wave maps with energy $= 2\cl E(\vec Q)$ we address in this paper is motivated by several recent advances made in establishing this conjecture for equivariant wave maps. We first state the following refined threshold theorem proved in~\cite{CKLS1}. 

\begin{thm}\emph{\cite{CKLS1}\label{t:2EQ}}  For smooth initial data $(\psi_0,\psi_1) \in \HH_0$ with 
\EQ{
\E^k(\psi_0,\psi_1) < 2\E^k(\vec Q^k), 
}
there exists a unique global smooth $k$-equivariant wave map $\vec \psi \in C(\R; \HH_0)$ with $\vec \psi(0) = (\psi_0,\psi_1)$. Moreover, $\vec\psi(t)$ scatters both forward and backward in time, i.e. there exist solutions $\vec \fy_L^\pm$ to the linearized equation
\begin{align}\label{free}
\p_t^2 \varphi - \p_r^2 \varphi - \frac{1}{r} \p_r \varphi + \frac{k^2}{r^2} \p_r \varphi = 0, 
\end{align}
 such that
\EQ{ \label{scat}
 \vec{\psi}(t) = \vec \fy_L^\pm(t) + o_{\HH_0}(1) \mas t \to  \pm\infty. 
}
\end{thm}

The intuition for the threshold energy being $2 \E^k(\vec Q^k)$ rather than $\E^k(\vec Q^k)$ is the following.  If a $k$-equivariant map $\vec \psi(t) \in \cl H_0$ wraps the plane around the sphere once, then it must also unwrap the sphere once more in order to have degree 0.  Since the minimum amount of energy needed for a $k$-equivariant map to wrap the plane around the sphere once is equal to $\E^k(\vec Q^k)$, it follows that if $\E^k(\vec \psi) < 2 \E^k(\vec Q^k)$ then $\psi(t)$ is bounded away from the south pole (i.e. $\psi(t,r) < \pi - \e, \quad \forall t,r$).  Thus, $\vec \psi(t)$ cannot converge locally to a harmonic map $\vec Q^k$ which by Struwe's bubbling result implies $\vec \psi(t)$ is globally regular.

A result analogous to Theorem \ref{t:2EQ} for the full wave map system, with no symmetry assumptions, was established by Lawrie and Oh in \cite{LO1}.  More precisely, we say initial data $(u_0,u_1)$ (with target $\bbS^2$) is \emph{topologically trivial} if 
\begin{align*}
\frac{1}{4\pi} \int_{\R^2} u_0^* \, \om_{\bbS^2} = 0,  
\end{align*}
where $\om_{\bbS^2}$ is the volume form on $\bbS^2$. It can be checked that the above condition is propagated by the wave map evolution, and an equivariant map $\vec u$ with associated azimuth angle $\vec \psi \in \HH_0$ is topological trivial.  The authors obtain the following result as a consequence of the analysis from \cite{ST2}. 
  \begin{thm}\emph{ \cite{LO1}} \label{t:2EQfull}
  Suppose that $(u_0,u_1)$ is smooth topologically trivial finite energy initial data with 
  \begin{align*}
  \cl E(u_0,u_1) < 8\pi = 2 \cl E(\vec Q^1).
  \end{align*} 
Then there exists a unique global solution $u:\R^{1+2} \rar \bbS^2$ to the wave maps equations \eqref{wm} with $\vec u(0) = (u_0,u_1)$.  Moreover, $\vec u(t)$ scatters to the constant map as $t \rar \pm \infty$.  
  \end{thm} 
  
The works~\cite{CKLS1, CKLS2} also established soliton resolution for corotational wave maps in $\HH_1$ with energy below $3 \cl E(\vec Q)$.  In this setting only one concentrating bubble is possible, and these works showed that for any such wave map there exists a solution $\vec \varphi_L(t) \in \HH_0$ to the free equation \eqref{free} (the radiation) and a continuous dynamical scale  $\la(t) \in (0, \infty)$ such that 
\EQ{ \label{eq:sr1} 
	\vec \psi(t) = \vec Q_{\la(t)}  +  \vec \fy_L(t) + o_{\HH_0}(1) \mas t \to T_+. 
} 
Proving soliton resolution above $3 \cl E(\vec Q)$ is very challenging since one can conceivably have multiple harmonic maps concentrating at different scales and interacting.  However, there has been exciting recent progress in establishing a weaker form of the conjecture.  The work by Cote \cite{Cote15} (for $1$-equivariant maps) and Jia, Kenig \cite{JK} (for all equivariant maps) established the following soliton resolution result along a well-chosen sequence of times. 

\begin{thm}\emph{ \cite{Cote15, JK}}\label{t:cjk} Let $\vec \psi(t)\in \HH_{n}$ be a smooth $k$-equivariant wave map on $[0, T_+)$. Then there exists a sequence of times $t_n \to T_+$, an integer $J \in \N \cup \{0\}$, a solution $\vec \varphi_L(t) \in \cl H_0$ to \eqref{free}, sequences of scales $\la_{n, j}$ which satisfy $0 < \la_{n,1} \ll \la_{n,2} \ll \cdots \ll \la_{n,J}$ and signs $\iota_j \in \{-1, 1\}$ for $j \in \{1,  \dots, J\}$,  so that 
\EQ{ \label{eq:seq} 
 \vec \psi(t_n)  = \sum_{j =1}^J \iota_j \vec Q^k_{\la_{n,j}}  + \vec \fy_L(t_n) + o_{\HH_0}(1) \mas n \to \infty.
}
If $T_+ < \infty$ then $J \geq 1$, $0 < \la_{n,1} \ll \cdots \ll \la_{n,J} \ll T_+ - t_n$, and if $T_+ = \infty$ then $0 < \la_{n,1} \ll \cdots \ll \la_{n,J} \ll t_n.$  The 
 signs $\iota_j$ are required to satisfy the topological constraint $\vec \psi(t) \in \cl H_n$, i.e. $$\lim_{r \to \infty} \sum_{j =1}^J \iota_j  Q^k_{\la_{n,j}}(r) = n\pi.$$ 
\end{thm}

We remark that the works
~\cite{CKLS1, CKLS2, Cote15, JK, JL} use ideas and techniques inspired by the seminal papers on the focusing quintic nonlinear wave equation in three space dimensions by Duyckaerts, Kenig, and Merle~\cite{DKM1, DKM2, DKM3, DKM4} (see also \cite{KBook-15} for an account of the important techniques and ideas in these papers).   

In \cite{JJ-AJM}, Jendrej showed it is possible for more than one bubble to form in the decomposition \eqref{eq:seq}.

\begin{thm}\emph{\cite{JJ-AJM}}
	\label{thm:deux-bulles-wmap}
	Fix an equivariance class $k > 2$. There exists a solution $\vec\psi: (-\infty, T_+) \to \HH_0$ of \eqref{eq:equiv}
	such that
	\begin{equation}
	\label{eq:mainthm-wmap}
	\lim_{t\to -\infty}\big\|\vec\psi(t) - \big(\vec Q_{c_k |t|^{-\frac{2}{k-2}}} - \vec Q \big)\big\|_{\HH_0} = 0,
	\end{equation}
	where $c_k > 0$ is explicit. \qed
\end{thm}  
A similar construction is possible when $k = 2$ with an explicit exponentially decaying scale as $t \rar -\infty$. By Theorem \ref{t:2EQ}, these solutions are examples of non-dispersing \emph{threshold solutions} to \eqref{eq:equiv} for $k \geq 2$. 

In \cite{JL}, Jendrej and Lawrie classified the dynamics of $k$-equivariant wave maps $\vec \psi(t)$ with \emph{threshold energy} $\E^k(\vec \psi(t)) = 2 \E^k(\vec Q^k)$ for $k \geq 2$.   Their work provided the primary motivation and roadmap for establishing our main results.  To state their results concisely, 
we first introduce some terminology.  Let $\vec \psi(t) : (T_-, T_+) \to \HH_0$ be a $k$-equivariant wave map with $\E^k(\vec \psi) = 2 \E^k(\vec Q^k)$. We say that $\vec\psi(t)$ is a \emph{two-bubble in
	the forward time direction} if there exist $\iota \in \{1, -1\}$
and continuous functions $\lambda(t), \mu(t) > 0$ such that
\EQ{
	\lim_{t \to T_+} \| (\psi(t) - \iota(Q^k_{\la(t)} - Q^k_{\mu(t)}), \psi_t(t))\|_{\HH_0} = 0, \quad \lambda(t) \ll \mu(t)\text{ as }t \to T_+.
}
The notion of a \emph{two-bubble in the backward time direction} is defined similarly.  

\begin{thm}\emph{ \cite{JL}}\label{t:JL} Let $k \geq 2$, and let $\vec\psi :(T_-, T_+) \to \HH_0$ be a $k$-equivariant wave map such that
 	\EQ{
 		\E^k(\vec \psi) = 2 \E^k(\vec Q^k) = 8\pi k.
 	}
 	Then $T_- = -\infty$, $T_+ = +\infty$ and one the following alternatives holds:
 	\begin{itemize}[leftmargin=0.5cm]
 		\item $\vec \psi(t)$ scatters in both time directions,
 		\item $\vec\psi(t)$ scatters in one time direction and is a two-bubble in the other time direction.  Moreover if $\vec \psi(t)$ is a two-bubble in the forward time direction, then there exists $C = C(k) > 0, \mu_0 > 0$ such that $\mu(t) \rar \mu_0$ and 
 		\begin{align}
 		\mu_0 \exp(-Ct) \leq&\,\, \la(t) \leq \mu_0 \exp(-t/C), \quad \mbox{if } k = 2, \\
 		\frac{\mu_0}{C} t^{\frac{-2}{k-2}} \leq &\,\,\la(t) \leq C \mu_0 t^{\frac{-2}{k-2}}, \quad \mbox{if } k \geq 3. 
 		\end{align}
An analogous estimate holds if $\vec  \psi(t)$ is a two-bubble in the backward time direction.  
 	\end{itemize} 
 \end{thm}  

\subsection{Main results}

The two-bubble solutions given by Theorem \ref{thm:deux-bulles-wmap} and the classification result Theorem \ref{t:JL} are for $k$-equivariant wave maps with $k \geq 2$.  The first main result of this paper establishes the existence of a corotational two-bubble solution.  In contrast to higher equivariant wave maps, our solution is in fact a threshold \emph{blow-up solution}.    

 \begin{thm}[Main Theorem 1]\label{t:main2}
 	There exists a corotational wave map $\vec \psi_c : (0,T_+) \to \HH_0$, a continuous scale $\la_c(t) >  0$, and constant $C > 0$ such that
 	\begin{align}
 	\frac{1}{C} t^2 \leq \la_c(t) |\log \la_c(t)| \leq C t^2,
 	\end{align}  
 	and 
 	\begin{align}
 	\lim_{t \rar 0^+} \left  \| \vec \psi_c(t) - \bigl (\vec Q_{\la_c(t)} - \vec Q \bigr ) \right  \|_{\HH_0} = 0. 
 	\end{align}
 	In particular, $\cl E(\vec \psi_c) = 8\pi$ and $T_- = 0$.  
 \end{thm}

By Theorem \ref{t:2EQ}, $\vec \psi_c$ is a minimal energy non-dispersing solution to \eqref{eq:wmk}. Moreover, by Theorem \ref{t:2EQfull} the map $u_c:(0,T_+) \times \R^2 \rar \bbS^2$ given by 
\begin{align*}
u_c(t,r,\theta) = (\sin \psi_c(t,r) \cos \theta, \sin \psi_c(t,r) \sin \theta, \cos \psi_c(t,r)),
\end{align*}   
is a topologically trivial minimal energy non-dispersing solution to the \emph{full wave map equations.} The existence 
of such a solution has been an open question up until now. 
The proof of Theorem \ref{t:main2} is a byproduct of estimates we derive to prove our second main result and the general scheme for constructing multi-soliton solutions introduced by Martel \cite{Martel-AJM15} and Merle \cite{Merle90}.

 \begin{thm}[Main Theorem 2] \label{t:main} 
Let $\vec\psi(t) :(T_-, T_+) \to \HH_0$ be a solution to~\eqref{eq:wmk} such that
\EQ{
\E(\vec \psi) = 2 \E(\vec Q) = 8\pi.
}
Then either $T_- = -\infty$ or $T_+ = +\infty$.  Assume that $T_- = -\infty$. Then $\vec \psi(t)$ scatters in backward time, while in forward time one of the following holds: 
 \begin{itemize}[leftmargin=0.5cm]
\item $T_+ = \infty $ and $\vec \psi(t)$ scatters in forward time,
\item $T_+ < \infty$, $\vec\psi(t)$ is a two-bubble in the forward time direction, and there exists an absolute constant $C > 0$ such that the scales of the bubbles $\lambda(t), \mu(t)$ satisfy
\EQ{
	\lim_{t \rar T_+} \mu(t) = \mu_0 \in (0, \infty), \qquad
\frac{1}{C}(T_+ - t)^2 \leq \frac{\la(t)}{\mu_0}\left |
\log \Bigl ( 
\frac{\la(t)}{\mu_0}
\Bigr )
\right | \leq C (T_+ - t)^2.
}
\end{itemize} 
 If we assume initially that $\vec \psi(t)$ satisfies $T_+ = \infty$, then $\vec \psi(t)$ scatters in forward time, and one of analogous alternatives formulated in backward time must hold.
 \end{thm}   
 
Overall, our main results state that the dynamics of corotational wave maps at threshold energy are very different from the those of higher equivariant wave maps at threshold energy. 
  
We remark that by Theorem \ref{t:main}, the blow-up solution $\vec \psi_c$ from Theorem \ref{t:main2} is global in forward time, $T_+ = \infty$, and scatters.  Thus, $\vec \psi_c$ is a trajectory connecting asymptotically free behavior to blow-up behavior. Theorem \ref{t:main} also asserts that for \eqref{eq:wmk}, the collision of two bubbles produces only radiation and is therefore inelastic.  This is consistent with what is known and expected for nonintegrable dispersive equations (see \cite{MM11-2, MM11, MM17, JL}).  Our main results are in the spirit of the classification results at threshold energy by \cite{DM, DM-NLS, JL}, but one may also draw parallels to the study of minimal blow-up solutions for dispersive equations (see for example \cite{Merle93}, \cite{RaSz11}).  Finally, we remark that apart from the seminal work by Duyckaert, Kenig and Merle \cite{DKM4} which verified the soliton resolution conjecture for the $3d$ radial energy critical wave equation and Theorem \ref{t:JL} due to Jenrej and Lawrie, Theorem \ref{t:main} is the only other result which proves soliton resolution continuously in time at an energy level that a priori allows two solitons in the asymptotic decomposition. In fact, Theorem 1.7 shows that solutions with two concentrating bubbles cannot occur, and any non-scattering solution must blow up precisely one bubble while radiating a second stationary harmonic map outside the inverted light cone. 

\subsection{Outline}

The general framework for proving Theorem \ref{t:main} is inspired by the work \cite{JL} on higher equivariant wave maps, but due to the slow convergence to $\pi$ of the corotational harmonic map $Q(r) = 2 \arctan r$,   
there are serious technical challenges not found in the higher equivariant setting that arise. The main source of these obstacles will be elaborated on below.

A rough outline of the proof of Theorem \ref{t:main} is as follows. By Theorem \ref{t:cjk}, a corotational wave map $\vec \psi$ that does not scatter forward in time must approach the space of two-bubbles along a sequence of times.  Towards a contradiction, we assume that $\vec \psi$ does not approach the space of two-bubbles continuously in time.  We then split time into a sequence of intervals $[a_m, b_m]$ so that $\vec \psi(t)$ is close to the space of two-bubbles on $[a_m,b_m]$ (``bad intervals"), and $\vec \psi(t)$ stays away from the space of two-bubbles on $[b_m, a_{m+1}]$ (``good intervals").  By concentration compactness techniques, the trajectory $\vec \psi(t)$ has a certain \emph{compactness property} on the union of good intervals (see Section 2, Section 4).  Past experience suggests that $\vec \psi(t)$ converges to a degree-0 stationary solution to \eqref{eq:wmk} along a sequence of times in the good intervals (see \cite{DKM16} for example).  Since the only degree-0 stationary solution to \eqref{eq:wmk} is $0$, we conclude $\vec \psi = 0$, a contradiction. 

To prove that $\vec \psi(t)$ approaches a stationary solution to \eqref{eq:wmk}, we use a virial identity for wave maps (see Section 2) which bounds an integral of $\|  \p_t \psi(t) \|_{L^2}^2$ over certain good intervals by small error terms plus an integral of $\bfd(\vec \psi(t))^{1/2}$ over certain bad intervals.  Here $\bfd(\cdot)$ is a measure of the distance to the space of two-bubbles (see Section 2).  The errors can be made small because $\vec \psi$ is close to a two-bubble on the bad intervals and has the compactness property on the good intervals. The time integral of $\bfd(\vec \psi(t))^{1/2}$ can be absorbed into the left-hand side, which shows that $\|  \p_t \psi(t) \|_{L^2}^2$ converges to 0 in a certain averaged (over the good intervals) sense. The compactness property then allows us to conclude that $\vec \psi(t)$ must approach a stationary solution.  The fact that the integral of $\bfd(\vec \psi(t))^{1/2}$ can be absorbed into the left-hand side is due to the following informal fact: leaving the space of two-bubbles on a bad interval causes an appreciable amount of kinetic energy, $\| \p_t \psi(t) \|_{L^2}^2$, to be present on the neighboring good interval (see Proposition \ref{prop:modulation}, Lemma \ref{l:time_split} and Section 4 for precise statements). 
 We prove this fact by studying the interaction of corotational two-bubbles using the modulation method (see Section 3).  This is one of the main novelties of this paper. 
 
 On a time interval where a corotational wave map $\vec \psi$ is close to a two-bubble, we decompose the solution as 
 \begin{align*}
 \vec \psi(t) = \Bigl ( 
 Q_{\la(t)} - Q_{\mu(t)} - g(t), \p_t \psi(t)
 \Bigr )
 \end{align*} 
 where the modulation parameters $\la(t)$ and $\mu(t)$ are chosen by imposing  certain orthogonality conditions on $g$.  The choice also ensures that $\bfd(\vec \psi(t))$ is comparable to $\la(t)/\mu(t)$. The goal of Section 3 is to show and control growth of the ratio $\la(t)/\mu(t)$ in the future of a time $t_0$ where $\frac{d}{dt} |_{t = t_0} \la(t)/\mu(t) > 0$ (see Proposition \ref{p:modp}, Proposition \ref{prop:modulation}).  In contrast to the work by Jendrej and Lawrie \cite{JL} on higher equivariant wave maps, the function $(r \p_r Q)_{\la(t)}$, which is the tangent vector to the curve $t \mapsto Q_{\la(t)}$ is not in $L^2(\R^2)$.  This function plays a key role in the scheme since $\la(t)^{-1} \ang{(r \p_r Q)_{\la(t)} \mid \p_t \psi(t)}_{L^2}$ should heuristically be proportional to $\la'(t)$, so we may then differentiate it and use \eqref{eq:wmk} to get information about $\la''(t)$.  The fact that $r \p_r Q \notin L^2$ is the major obstacle in deriving the estimates, and the technique we introduce in Section 3 to overcome this challenge is a central contribution of this work.
 
  We conclude our discussion of the proof of Theorem \ref{t:main} with the following remarks.  Our overall scheme of proving Theorem \ref{t:main} may also be summarized as showing that a threshold wave map that leaves a small neighborhood of the space of two-bubbles can never return. This type of \emph{ejection} result is similar in appearance to those obtained by Krieger, Nakanishi and Schlag in their study of the dynamics near the unstable ground state for the energy critical wave equation \cite{KNS13, KNS15}.  However, the ejection of a near two-bubble wave map is due to a purely nonlinear mechanism (the interaction of the harmonic maps).
 
We now briefly outline the proof of Theorem \ref{t:main2}. The construction of the blow-up solution $\vec \psi_c(t)$ is quite short due to the results proved in Section 3.  We consider initial data at time $t_n$ of the form
\begin{align*}
\vec \psi_n(t_n) = \Bigl ( Q_{\ell(t_n)} - Q, -\ell'(t_n) \ell(t_n)^{-1} (r \p_r Q)_{\ell(t_n)} \chi_n \Bigr )
\end{align*}
where $\chi_n$ is a cutoff that ensures $\cl E(\vec \psi(t_n)) = 8 \pi$. The function $\ell(t)$ is chosen to satisfy $\ell'(t_n) > 0$ and to essentially saturate the bounds on the modulation parameters in Proposition \ref{p:modp}. Let $\vec \psi_n(t)$ denote the solution to \eqref{eq:wmk} with data $\vec \psi_n(t_n)$ at time $t = t_n$.  By our choice of the data, the control of the growth of the modulation parameters obtained in Proposition \ref{p:modp} and a bootstrap argument, we conclude that there exist absolute constants $\al, C, T > 0$ with $T$ small such that $T_+(\vec \psi_n) > T$ and  
\begin{align*}
\inf_{\substack{\mu \in [1/2,2] \\
		\la |\log \la| \in [t^2/C,Ct^2]}}  \| \vec \psi_n(t) - (Q_\la - Q_\mu) \|_{\HH_0}^2 \leq \al t^2, \quad \forall n, \forall t \in [t_n,T].
\end{align*}
Passing to a weak limit then finishes the proof.  Full details are in Section 5.

 \subsection{Acknowledgments}
Support of the National Science Foundation, DMS-1703180, is gratefully acknowledged.

 \section{Preliminaries}

 The purpose of this section is to recall preliminary facts about solutions to~\eqref{eq:wmk} that will be required in our analysis. Before recalling these facts, we establish some notation. For two quantities $A$ and $B$, we write $A \lesssim B$ if there exists a constant $C > 0$ such that $A \leq C B$, and we write $A \sim B$ if $A \lesssim B \lesssim A$. For the paper, we denote by $\chi$ a smooth cutoff $\chi \in C^{\infty}_{\mathrm{rad}}(\R^2)$, so that, writing $\chi = \chi(r)$ we have 
 \EQ{
 	\chi(r) = 1 \mif r \le 1 \mand  \chi(r) = 0 \mif r \ge 2 \mand \abs{\chi'(r)} \le 2 \quad \forall r \ge 0.
 }
We denote $\chi_R(r) := \chi(r/R)$.  The $L^2$ pairing of two radial functions is denoted by 
\EQ{
\ang{f \mid g}  := \frac{1}{2\pi}\ang{ f \mid g}_{L^2(\R^2)} = \int_0^\infty f(r) g(r) \, r d r.
}
The $\dot H^1$ and $L^2$ re-scalings of a radial function $f$ are denoted by 
\EQ{ \label{eq:scaledef} 
	f_\la(r) = f(r/ \la), \quad
	f_{\ula}(r)  = \frac{1}{\la} f(r/ \la),
}
and the corresponding infinitesimal generators  are given by 
\EQ{ \label{eq:LaLa0} 
	&\La f := -\frac{\partial}{\partial \lambda}\bigg|_{\lambda = 1} f_\la = r \p_r f  \quad (\dot H^1_{\textrm{rad}}(\R^2) \,  \textrm{scaling}), \\
	& \La_0 f := -\frac{\partial}{\partial \lambda}\bigg|_{\lambda = 1} f_{\ula} = (1 + r \p_r ) f  \quad (L^2_{\textrm{rad}}(\R^2) \,  \textrm{scaling}).
}

Recall the definition of the space of degree-0 data with finite energy: 
\EQ{
\HH_0:= \{ (\psi_0, \psi_1) \mid \E(\psi_0, \psi_1)< \infty, \quad \lim_{r \to 0} \psi_0(r) = \lim_{r \to \infty} \psi_0(r) = 0 \}.
}
We define the following norm $H$ via 
 \EQ{
 \| \psi_0 \|_{H}^2 := \int_0^\infty  \left(( \p_r \psi_0(r))^2 +  \frac{(\psi_0(r))^2}{r^2} \right)  r d r 
  }
 and for pairs $\vec \psi = (\psi_0, \psi_1) \in \HH_0$ we write  
 \EQ{
 \| \vec \psi \|_{\HH_0} := \| (\psi_0, \psi_1)\|_{H \times L^2}.
 }
Given $\psi_0 \in H$, if we define $\tilde \psi_0(x) := \psi(e^x)$, $x \in \R$, we see that 
$\|\psi_0 \|_{H} = \| \tilde \psi_0 \|_{H^1(\R)}$. Thus, by Sobolev embedding on $\R$ we conclude that 
\EQ{
 \| \psi_0 \|_{L^{\infty}} \le C \| \psi_0 \|_{H}. 
}
This fact will be used frequently in our analysis.

\subsection{Cauchy theory %
%: $4d$ reduction for solutions to~\eqref{eq:wmk} in $\HH_0$
} \label{s:2-4}
The study of the Cauchy problem for \eqref{eq:wmk} with initial data $(\psi_0, \psi_1) \in \HH_0$, is facilitated by a well-known reduction that takes into account the extra dispersion provided by the nonlinearity.  In particular, the nonlinearity satisfies 
\begin{align*}
\frac{\sin 2 \psi}{2r^2} = \frac{\psi}{r^2} + \frac{\sin 2\psi - 2\psi}{2r^2}.
\end{align*}
The second term on the right-hand side is now cubic in $\psi$. 
Thus, the linear part of \eqref{eq:wmk} is given by 
\begin{align}
\p_t^2 \varphi - \p_r^2 \varphi - \frac{1}{r} \p_r \varphi + \frac{1}{r^2} \varphi = 0, \label{eq:2dlin} 
\end{align}
which due to the strong repulsive potential $\frac{1}{r^2}$ has more dispersion than the wave equation on $\R^{1+2}$. By a change of the dependent variable, one sees that the linearized equation \eqref{eq:2dlin} has the same dispersion as the wave equation on $\R^{1+4}$.  Indeed, for $\vec \varphi(t) \in \HH_0$, we define $\vec v(t) = \frac{1}{r} \vec \varphi(t)$.   
Then 
\EQ{ \label{eq:free} 
\frac{1}{r}\left (\p_t^2 - \Delta_{\R^2} + \frac{1}{r^2} \right ) \varphi = (\p_t^2-\De_{\R^{4}}) v. 
}

We now use this change of variables to study \eqref{eq:wmk}. 
Let $\vec \psi(t)$ be a solution to~\eqref{eq:wmk}, and define  $u$ by $ru = \psi$.  Then $u$ satisfies 
\EQ{\label{eq:4d}
&\p_t^2 u - \p_r^2 u -\frac{3}{r} \p_r u   = Z(ru) u^3 \\
&\vec u(0)= (u_0, u_1). 
}
where the function $$Z(\psi) := \frac{2\psi - \sin 2 \psi}{2\psi^3}$$ is a smooth, bounded, even function. The linear part of~\eqref{eq:4d} is the radial wave equation in $\R^{1+4}$
\EQ{\label{eq:4dlin}
&\p_t^2 v - \p_r^2 v -\frac{3}{r} \p_r v=0.
} 
To compare the size of $\vec u$ to $\vec \psi$, we note that by Hardy's inequality we have 
\EQ{ \label{eq:hardy} 
\int_0^{\infty} \left ( 
(\p_r u)^2 + \frac{u^2}{r^2}
\right ) r^4 dr \sim \| u \|_{\dot{H}^1(\R^4)}^2. 
} 
Thus, the map 
 \EQ{
 (u_0, u_1) \mapsto  ( \psi_0, \psi_1):= (ru_0, ru_1)
}
 satisfies 
 \EQ{ \label{eq:2-4}
\|(u_0, u_1) \|_{\dot{H}^1 \times L^2(\R^4)} \sim  \| (\psi_0, \psi_1) \|_{H \times L^2 (\R^2)},
 }
and we conclude that the Cauchy problem for~\eqref{eq:4d} with initial data in $\dot{H}^1 \times L^2(\R^4)$ is equivalent to the Cauchy problem for~\eqref{eq:wmk} for initial data $(\psi_0, \psi_1) \in \HH_0$. 

We recall the following Strichartz estimates for solutions to the free wave equation on $\R^{1+4}$. Let $v$ be a solution to the wave equation
	\EQ{
		&\p_t^2v  - \Delta_{\R^4}  v = F(t, x), \\ &\vec v(0) = (v_0, v_1) \in \dot{H}^1 \times L^2 (\R^4).
	}
	Then there exists an absolute constant $C > 0$ such that for any time interval $I \subset \R$ we have 
	\EQ{ \label{eq:strich} 
		\| v \|_{L^{3}_t(I; L^6_x(\R^4))} + \sup_{t \in I}\| \vec v(t) \|_{\dot{H}^1 \times L^2(\R^4)} \leq C \left (   \| \vec v(0) \|_{\dot{H}^1 \times L^2(\R^4)} + \| F \|_{L^1_t(I; L^2_x(\R^4)} \right ).
	}  
Using the Strichartz estimates \eqref{eq:strich} and a contraction mapping argument, it is now standard to obtain the following well-posedness and scattering criterion for \eqref{eq:4d}.  These facts will be stated in terms of the original azimuth angle $\psi = ru$.  

\begin{prop}\label{l:scattering}	Let $(\psi_0,\psi_1) \in \cl H_0$.  Then there exists a unique solution $\vec \psi \in C(I_{\max}; \cl H)$ to \eqref{eq:wmk} with $\vec \psi(0) = (\psi_0,\psi_1)$ 
	defined on a maximal time interval of existence $I_{\max}(\psi) = (T_-(\psi), T_+(\psi))$ such that for any $J \Subset I_{\max}$, 
	\begin{align*}
	\| \vec \psi \|_{L^\infty_t(J, H \times L^2(\R^2))} + \| \psi/r \|_{L^3_t(J; L^6_x(\R^4))} < \infty.  
	\end{align*}
	A solution $\vec \psi(t)$ satisfies $T_+(\psi) = \infty$ and scatters in forward time if and only if 
	$$\| \psi/r \|_{L^3_t((0,T_+);L^6_x(\R^4))} < \infty.$$ A similar statement holds for negative times.  
	
	  Finally, we have the standard finite time blow--up criterion:
	$T_+(\psi) < \infty$ if and only if $$\| \psi/r \|_{L^3_t((0,T_+);L^6_x(\R^4))} = \infty.$$  A similar statement holding for negative times. 
\end{prop}

 Using Strichartz estimates and continuity arguments, one also has the following long-time perturbation lemma from \cite{CKLS1}. 
 \begin{lem}\emph{\cite[Lemma 2.18]{CKLS1} \label{l:pert}} There are continuous functions $\e_0,  C_0: (0, \infty) \to (0, \infty)$ with the following property. Let $I\subset \R$ be an open interval, and let $\psi, \fy \in C^0(I; H) \cap C^1(I; L^2) $ be radial functions such that for some $A>0$ 
\begin{align*} 
&\|\vec \psi\|_{L^{\I}_t(I; H \times L^2(\R^2))}+ \|\vec\fy\|_{L^{\I}_t(I; H \times L^2(\R^2))}+ \|\fy/r\|_{L^3_t(I; L^6_x(\R^4))} \le A\\
&\|\textrm{eq}(\psi/r)\|_{L^1_t(I; L^2_x(\R^4))}+\|\textrm{eq}(\fy/r)\|_{L^1_t(I; L^2_x(\R^4))} + \|w_0/r\|_{L^3_t(I; L^6_x(\R^4))} \le \e \le \e_0(A)
\end{align*} 
where $\textrm{eq}(\psi/r):= \Box_{\R^4} (\psi/r) +(\psi/r)^3Z(\psi)$ in the sense of distributions, $\vec w_0(t):= S(t-t_0)(\vec \psi-\vec \fy)(t_0)$ with $t_0 \in I$ fixed, and $S$ denotes the propagator for~\eqref{eq:2dlin}. Then,
\begin{align*} 
\|\vec \psi -\vec \fy - \vec w_0\|_{L^{\I}_t(I; H \times L^2(\R^2))} + \|\frac{1}{r}(\psi-\fy)\|_{L^3_t(I; L^6_x(\R^4))} \le C_0(A) \e
\end{align*} 
In particular, $\|\psi/r\|_{L^3_t(I; L^6_x(\R^4))} < \I$. 
\end{lem} 
 
 \subsection{Concentration Compactness} \label{s:cc} 
Two fundamental tools used in the study of \eqref{eq:wmk} (and in the study of large data solutions of dispersive equations in general) are the linear and nonlinear profile decompositions of Bahouri and Ger\'ard.  The following linear profile decomposition for the azimuth angles follows from the main result in \cite{BG} and the equivalence of \eqref{eq:free} and \eqref{eq:2-4}. 

\begin{lem}\emph{\cite{BG}} \label{c:bg} Let $(\vec \psi_n) \subset \HH_0$ be a sequence that is uniformly bounded in $\HH_0$. Then, after extracting a subsequence if necessary,  there exists a sequence of solutions $\vec \fy^j_L \in \HH_0$ to~\eqref{eq:2dlin},  sequences of times $\{t_{n, j}\}\subset \R$,  sequences of scales $\{\la_{n, j}\}\subset (0, \infty)$, and errors $\vec \ga_n^J(0) \in \cl H_0$ defined by 
\EQ{
\vec \psi_n = \sum_{j=1}^J  (\vec \varphi^j_L)_{\la_{n,j}}\Bigl 
( -\frac{t_{n,j}}{\la_{n,j}} \Bigr )
+ \vec \gamma^J_n(0)
%& \psi_{n, 1}(r)= \sum_{j=1}^k \frac{1}{\la_{n, j}}\dot\fy^j_L( -t_{n, j}/ \la_{n, j}, r/ \la_{n, j}) + \ga_{n, 1}^k(r)
}
with the following properties. Let $\ga_{n, L}^J(t) \in \HH_0$ denote the solution to \eqref{eq:2dlin} with initial data $\vec \ga_n^J(0) \in \HH_0$. Then,  for any $j \le \ell$, 
\EQ{ \label{eq:ga-weak} 
(\ga_n^\ell(  t_{n, j},  \la_{n, j}\cdot) , \la_{n, j} \ga_n^\ell(  t_{n, j},  \la_{n, j}\cdot)) \rightharpoonup 0 \textrm{ weakly in }\HH_0. 
}
In addition, for any $j\neq \ell$ we have
\EQ{ \label{eq:po}
\frac{\la_{n, j}}{\la_{n, \ell}} + \frac{\la_{n, \ell}}{\la_{n, j}} + \frac{\abs{t_{n, j}-t_{n, \ell}}}{\la_{n, j}} + \frac{\abs{t_{n, j}-t_{n, \ell}}}{\la_{n, \ell}} \to \infty \quad \textrm{as} \quad n \to \infty.
}
The errors $\vec \ga_n^J$ vanish asymptotically in the dispersive sense  
\EQ{
\limsup_{n \to \infty} \left\|\frac{1}{r} \ga_{n, L}^J\right\|_{L^{\infty}_tL^4_x \cap L^3_tL^6_x( \R \times \R^4)}  \to 0 \quad \textrm{as} \quad J \to \infty.
}
Finally, we have a Pythagorean expansion of the $\HH_0$ norms: 
\EQ{ \label{ort H} 
\|\vec \psi_n\|_{\HH_0}^2 = \sum_{1 \le j \le J} \| \vec \fy_L^j( - t_{n, j}/ \la_{n, j}) \|_{\HH_0}^2  +  \|\vec \ga_n^J\|_{\HH_0}^2 + o_n(1) \mas n \to \infty
} 
\end{lem}
Applying the concentration-compactness methods of Kenig and Merle \cite{KM06}, \cite{KM08} to the study of \eqref{eq:wmk} requires the following Pythagorean expansion of the nonlinear energy proved in~\cite{CKLS1}. 
\begin{lem}\emph{\cite[Lemma $2.16$]{CKLS1} }\label{l:enorth}
Let  $\vec \psi_n \in \HH_0$ be a bounded sequence with a profile decomposition as in Lemma \ref{c:bg}. Then the following Pythagorean expansion for the nonlinear energy holds: 
\EQ{\label{eq:enorth} 
\E(\vec \psi_n) = \sum_{j=1}^J \E(\vec \fy_L^j(-t_{n, j}/ \la_{n, j})) + \E(\vec \ga_n^J) + o_n(1)  \mas n \to \infty.
} 
 \end{lem}
 
To apply Lemma \ref{c:bg} and Lemma \ref{l:enorth} in the context of the nonlinear problem \eqref{eq:wmk}, we construct the following nonlinear profiles. 
For each linear profile $\varphi^j_L$ with parameters $\{ t_{n,j}, \la_{n,j} \}_n$, we define its associated nonlinear profile $\varphi^j$ to be the unique
solution to \eqref{eq:wmk} such that after passing to a subsequence if necessary, for all $n$ sufficiently large, $-t_{n,j}/\la_{n,j} \in I_{\max}(\vec \varphi^j)$, and 
\begin{align*}
\lim_{n \rar \infty} \| \vec \varphi^j(-t_{n,j}/\la_{n,j}) - \vec \varphi^j_L(-t_{n,j}/\la_{n,j}) \|_{\HH_0} = 0. 
\end{align*}
It is easy to see that a nonlinear profile always exists.  Indeed, if $-t_{n,j}/\la_{n,j} \rar_n t_0 \in \R$, then we set $\varphi^j$ to be the 
solution to \eqref{eq:wmk} with initial data $\vec \varphi^j(t_0) = \vec \varphi^j_L(t_0)$.  If $-t_{n,j}/\la_{n,j} \rar_n \infty$, say, then we set 
$\varphi^j$ to be the unique solution to the integral equation 
\begin{align}%%%%%%%%%%%%%%%%%%%%%%%%%%%%%%%%%
\varphi^j(t) = \vec \varphi^j_L(t) - \int_t^\infty S(t-s)\left (0,\frac{2\varphi - \sin 2 \varphi}{2r^2} \right ) ds. \label{e912}
\end{align}
A unique solution to \eqref{e912} can be shown to exist using contraction mapping arguments and Strichartz esimates (for $u = \psi/r$). A similar construction can be made if $-t_{n,j} / \la_{n,j} \rar -\infty$.

The existence of nonlinear profiles and the long-time perturbation lemma yield the following nonlinear profile decomposition.  

\begin{lem}\emph{\cite[Proposition 2.17]{CKLS1}\cite[Proposition 2.8]{DKM1}}\label{p:nlprof} Let $\vec \psi_n(0)$ be a bounded sequence in $\HH_0$ with a profile decomposition as in Lemma~\ref{c:bg}.
Let $\vec \varphi_j$ be the associated nonlinear profiles. Let $s_n \in (0, \infty)$ be any sequence such that for all $j$ and for all $n$, 
\EQ{
\frac{s_n - t_{n, j}}{\la_{n, j}} < T_+(\vec \fy^j), \quad \limsup_{n \to \infty} \|\fy^j/ r\|_{L^3_t([-\frac{t_{n, j}}{\la_{n, j}}, \frac{s_n - t_{n, j}}{\la_{n, j}}); L^6_x(\R^4))} <\infty.
}
Let $\vec \psi_n(t)$ be the solution of \eqref{eq:wmk} with initial data $\vec \psi_n(0)$.  Then, for all $n$ sufficiently large, $\vec \psi_n(t)$ exists on the interval $s \in (0, s_n)$ and  
\EQ{
\limsup_{n \to \infty} \|\psi_n/r\|_{L^3_t([0, s_n); L^6_x(\R^4))} < \infty.
}
Finally, the following non-linear profile decomposition holds for all $s \in [0, s_n)$,  
\EQ{
\vec \psi_n(s, r) = \sum_{j=1}^J \left(\fy^j\left( \frac{s- t_{n, j}}{\la_{n, j}}, \frac{r}{\la_{n, j}}\right), \frac{1}{\la_n^j}\p_t \fy^j\left(\frac{s-t_{n, j}}{\la_{n, j}}, \frac{r}{\la_{n,j}}\right) \right) + \vec \ga_{n, L}^{J}(s, r)+ \vec \te_n^J(s, r)
}
where $\ga_{n, L}^J(t)$ is defined in Lemma \ref{c:bg} and 
\EQ{\label{eq:nlerror}
\lim_{J \to \infty} \limsup_{n\to \infty} \left( \|\te_n^J/r\|_{L^3_t([0, s_n); L^6_x(\R^4))} + \|\vec \te_n^J\|_{L^{\infty}_t ([0, s_n); \HH_0)} \right) =0.
}
An analogous statement holds for sequences $s_n\in (-\infty, 0)$. 
\end{lem}

The main result obtained from concentration-compactness methods along with Theorem \ref{t:2EQ} is the following compactness statement for nonscattering threshold solutions. The proof is the same as the higher equivariant analog found in \cite{JL} and is omitted. 
 
  \begin{lem}\emph{\cite[Lemma 2.9]{JL}} \label{l:1profile} Let $\vec \psi(t) \in \HH_0$ be a solution to~\eqref{eq:wmk} defined on $[0, T_+(\vec \psi))$. Suppose that $\E(\vec \psi) = 8 \pi$ and $\vec \psi(t)$ does not scatter in forward time. Then if $t_n \to T_+ $ is any sequence of times such that 
  \EQ{ \label{eq:Hbounded} 
  \sup_n \| \vec\psi(t_n)\|_{\HH_0} \le C < \infty,
  }
  there exist a subsequence which we continue to denote by $t_n$, scales $\nu_n>0$ and a nonzero $\vec \fy \in \HH_0$ such that 
  \EQ{
   \vec \psi(t_n)_{\frac{1}{\nu_n}} \to \vec \fy \in \HH_0
  }
  strongly in $\HH_0$. Moreover, $ \E(\vec \fy ) = 8 \pi$, and the solution $\vec \fy(s)$ to \eqref{eq:wmk} with data $\vec \fy(0) = \vec \fy$ is non-scattering in forwards and backwards time. 
 \end{lem} 

 \subsection{Near two-bubble maps} \label{s:hm} 
 
We recall that the unique (up to scaling) nontrivial corotational harmonic map $Q$ is given by 
\begin{align*}
Q(r) = 2 \arctan r. 
\end{align*}
The harmonic map $Q$ has a variational characterization as follows. As in the introduction, let $\cl H_{1}$ be the set of all finite energy corotational maps which map infinity to the south pole, i.e. 
\EQ{
	\HH_1:= \{ (\phi_0, \phi_1) \mid \E(\vec \phi)< \infty, \quad \phi_0(0) = 0, \quad \lim_{r \to \infty} \fy_0(r) = \pi\}.
}
Then for $(\varphi_0, \varphi_1) \in \HH_1$, we have the following 
Bogomol'nyi factorization of the nonlinear energy:
\EQ{ \label{eq:bog}
	\E( \fy_0, \fy_1)  &=   \pi \| \fy_1 \|_{L^2}^2 +  \pi \int_0^{\infty} \left(\p_r \fy_0  - \frac{\sin(\fy_0)}{r}\right)^2 \, r\, dr +  2\pi \int_0^{\infty} \sin(\fy_0) \p_r\fy_0 \, dr\\
	&= \pi \| \fy_1 \|_{L^2}^2  + \pi \int_0^{\infty} \left(\p_r \fy_0  - \frac{\sin(\fy_0)}{r}\right)^2 \, rdr + 2 \pi  \int_{\fy_0(0)}^{\fy_0(\infty)} \sin(\rho)  \, d\rho \\
	&  =  \pi \| \fy_1 \|_{L^2}^2  +  \pi \int_0^{\infty} \left(\p_r \fy_0  - \frac{\sin(\fy_0)}{r}\right)^2 \, rdr  + 4\pi. 
} 
By solving the differential equation in the parentheses, we see that $\cl E(\fy_0, \fy_1) \geq 4\pi$ with equality if and only if $(\fy_0,\fy_1) = (Q_\la,0)$ for some $\la > 0$.  

In our analysis, we will need several technical facts related to the distance of a map $\vec \psi$ to the set of $2$-bubbles. More precisely, given a map $ \vec\phi =  (\phi_0, \phi_1) \in \HH_0$ we define its distance $\bfd(\vec \phi)$ to the set of $2$-bubbles by 
\EQ{ \label{eq:ddef} 
\bfd(\vec \phi) := \inf_{\la, \mu >0, \iota \in \{+1, -1\}}  \Big(  \| (\phi_0 - \iota (Q_\la - Q_\mu), \phi_1) \|_{\HH_0}^2 + \left( \la/\mu \right) \Big).
}
To distinguish between the two cases of a map being close to a pure two-bubble ($\iota = +1$ above) or an anti two-bubble ($\iota = -1$ above), we define 
\EQ{\label{eq:dpm} 
	\bfd_\pm(\vec \phi) := \inf_{\la, \mu >0}  \Big(  \| (\phi_0  \mp  (Q_\la - Q_\mu), \phi_1) \|_{\HH_0}^2 + \left( \la/\mu \right) \Big).
} 
The next two lemmas follow from the same arguments given in \cite{JL} for higher equivariant wave maps, and the proofs will be omitted.  The first lemma shows that the size of a map $\vec \psi$ with threshold energy can be controlled by its distance to the surface of two-bubbles.  The second lemma proves the intuitive fact that a map $\vec \psi$ cannot simultaneously be close to a pure two-bubble and anti two-bubble.  
 
 \begin{lem}\emph{\cite[Lemma 2.13]{JL}} \label{l:d-size}
 Suppose that $\vec \phi = (\phi_0, \phi_1) \in \HH_0$ and
\EQ{
&\E( \vec \phi)  = 2 \E(\vec Q) = 8 \pi. 
%&\inf_{\la, \mu >0}  \Big(  \| (\phi_0 - (Q_\la - Q_\mu), \phi_1) \|_{H \times L^2}^2 + \left( \la/\mu \right)^k \Big)\ge \al_0.
 }
  Then for each $\be>0$ there exists $C(\be)>0$ such that 
 \EQ{ \label{eq:d-big} 
  \bfd(\vec \phi) \ge \be \Longrightarrow  \|(\phi_0, \phi_1) \|_{\HH_0} \le C(\be).
   } 
Conversely, for each $A>0$ there exists $\al  = \al(A)$ such that 
\EQ{\label{eq:d-small}
 \bfd(\vec \phi) \le \al(A) \Longrightarrow \|(\phi_0, \phi_1) \|_{\HH_0}  \ge A.
}
 \end{lem}
 
 \begin{lem}\emph{\cite[Lemma 2.14]{JL}} \label{l:dpm} 
 There exists an absolute constant $\al_0>0$ such that for any $\vec \phi \in \HH_0$ %For all $\al  \le  \al_0$, 
 \EQ{
 \bfd_{\pm}(\vec \phi) \le \al_0 \Longrightarrow \bfd_{\mp}(\vec \phi)  \ge \al_0. 
 }
 \end{lem} 
 
 The final preliminary results we will need for our analysis are related to a virial identity for solutions to \eqref{eq:wmk}. The following virial identity follows easily from  \eqref{eq:wmk} and integration by parts.  
 
 \begin{lem} \label{l:vir} 
 Let $\vec \psi(t)$ be a solution to~\eqref{eq:wmk} on a time interval $I$. Then for any time $t \in I$  and $R>0$ fixed we have 
 \EQ{\label{eq:vir}
 \frac{d}{d t} \ang{ \psi_t \mid \chi_R \, r \p_r \psi}_{L^2}(t)  = - \int_0^\infty \psi_t^2(t, r) \, rdr + \Om_R(\vec \psi(t))
 }
 where 
 \EQ{ \label{eq:OmRdef} 
 \Om_R(\vec \psi(t)) &:=  \int_0^\infty \psi_t^2(t)(1 - \chi_R) \, rdr   \\
 & \quad -\frac{1}{2} \int_0^\infty  \Big( \psi_t^2(t) + \psi_r^2(t) - \frac{\sin^2 \psi(t)}{r^2} \Big)   \frac{r}{R} \chi'(r/R) \, rdr
 }
 satisfies 
 \EQ{ \label{eq:OmRest} 
 \abs{\Om_R(\vec \psi(t))} &\lesssim \int_{R}^\infty \psi_t^2(t, r) \, r d r  \, d t + \int_{R}^{\infty} \abs{ \psi_r^2 - \frac{\sin^2 \psi}{r^2} } \,r  d r d t  \\
 &\lesssim  \int_{R}^\infty \left (
 \psi_t^2(t,r) + \psi_r^2(t,r) + \frac{\sin^2 \psi(t,r)}{2r^2}
 \right ) r \, dr. 
 }
\end{lem}

Finally, using Lemma \ref{l:d-size}, one can bound the virial and the error for threshold solutions by its distance to the set of $2$-bubbles. The proof of this fact is the same as in \cite{JL} and is omitted. 

 \begin{lem}\emph{\cite[Lemma 2.16]{JL}}\label{l:error-estim}
 There exists a number $C_0 > 0$ such that for all $\vec \phi = (\phi_0, \phi_1) \in \HH_0$ with $\E(\vec\phi) = 2\E(Q)$
 and all $R > 0$, we have 
\begin{align}
 |\ang{\phi_1, \chi_R r\partial_r \phi_0}| \leq C_0 R \sqrt{\bfd(\vec \phi)}, \label{eq:virial-end} \\
 \left | \Omega_R(\vec \phi) \right | \leq C_0 \sqrt{\bfd(\vec \phi)}. \label{eq:virial-err}
\end{align}
 \end{lem}

 %%%%%%%%%%%%%%%%%%%%    Section 3 %%%%%%%%%%%%%%%%%%%%%%%%%%%%%%%%%%%%%%%%%%%%%%%%%%%%

\section{The modulation method for two-bubble solutions} \label{s:mod} 

In this section we analyze the modulation equations that govern the
evolution of corotational near $2$-bubble solutions.  As in the case of higher equivariant wave maps studied by Jendrej and Lawrie \cite{JL}, the scale of the less concentrated bubble does not change, but it does affect the evolution of the more concentrated bubble. A central challenge which arises in the analysis of corotational maps which is not found in the higher equivariant setting is the fact that the zero mode of the operator obtained by linearizing about the harmonic map $Q$ is a \emph{resonance} rather than an eigenvalue.  A rough outline of this section is as follows.  For a solution $\vec \psi(t)$ with $\bfd(\vec \psi(t))$ small on a time interval $J$, we first use the implicit function theorem to find modulation parameters $\la(t), \mu(t)$ defined on $J$ such that 
$g(t) := \psi(t) - (Q_{\la(t)} - Q_{\mu(t)})$ satisfies appropriate orthogonality conditions and $\bfd(\vec \psi(t)) \simeq \frac{\la(t)}{\mu(t)}$.  We would like to then prove that if the modulation parameters $\la(t), \mu(t)$ are approaching each other in scale, i.e. if $\frac{d}{dt} \la(t)/\mu(t) |_{t = t_0} \geq 0$, then $\la(t)/\mu(t)$ continues to grow in a controlled way in forward time near $t_0$.  In particular, this would imply that $\vec \psi(t)$ has to leave a small neighborhood of the set of two-bubbles.    However, the slow decay of $Q$ requires us to deal with additional technical obstacles not encountered in the case of higher equivariant wave maps.  In particular, we must replace $\la(t)$ with a carefully chosen logarithmic correction.      

\subsection{Modulation Equations}
In this section, we study solutions near two-bubble solutions $\vec \psi(t)$ to~\eqref{eq:wmk}.  More precisely, we consider maps such that $ \bfd( \vec \psi(t))$ (defined by \eqref{eq:ddef}) is small on a time interval $J$.

The operator corresponding to linearizing \eqref{eq:wmk} about the harmonic map $Q_\la$ is the Schr\"odinger operator
\EQ{
\LL_\la:= - \p_r^2 - \frac{1}{r} \p_r + \frac{\cos 2Q_\la}{r^2}.
} 
For convenience we write $\LL := \LL_1$. Differentiating the equation
\begin{align*}
\p_r^2 Q_\lambda + \frac{1}{r} \p_r Q_\lambda - \frac{\sin 2 Q_\lambda}{2r^2} = 0
\end{align*}
with respect to $\lambda$ and setting $\lambda = 1$ implies that $\La Q$ is a zero mode for  $\LL$, i.e., 
\EQ{
\LL \La Q = 0, \quad \La Q \in L^\infty(\R^2).
}

Note that $\La Q \sim \frac{1}{r}$ as $r \rar  \infty$ so that $\La Q$ fails (logarithmically) to be in $L^2(\R^2)$.  We say that $\La Q$ is a \emph{resonance} of $\LL$.  In the $k$-equivariant setting with $k \geq 2$, $\La Q \in L^2(\R^2)$.  This weak decay of $\La Q$ requires more care when studying the modulation equations compared to the higher equivariant setting. We note that in general, we have 
$$\LL_\la Q_\la = 0.$$  

Define 
\begin{align}
\ZZ(r) := \chi_{L}(r) \La Q(r).  
\end{align}
where, as before, $\chi$ is a smooth cutoff.
The parameter $L > 0$ will be chosen later. 
We use $\ZZ$ to obtain a useful choice of modulation parameters (the scales) for the near two-bubble solution $\vec \psi(t)$.  We first recall the following modulation lemma from \cite{JL}, which follows from standard arguments involving the implicit function theorem, an expansion of the nonlinear energy and coercivity properties of $\LL_\la$. 

\begin{lem}\emph{\cite[Lemma $3.1$]{JL}}  \label{l:modeq} There exist $\eta_0 = \eta_0(L) >0$ and $C = C(L) > 0$
 such that the following holds. Let  $\psi(t)$ be a solution to~\eqref{eq:wmk} defined on a time interval $J \subset \R$,  and assume that
 \EQ{
 \bfd_+(\vec\psi(t)) \leq \eta_0\qquad \forall t\in J.
 }
 Then there exist unique $C^1(J)$ functions $\la(t), \mu(t)$ so that the function
\EQ{ \label{eq:gdef1} 
g(t):= \psi(t) - Q_{\la(t)} + Q_{\mu(t)} \in H , 
}
satisfies for all $t \in J$
\begin{gather} 
\ang{ \cZ_{\uln{\lambda(t)}} \mid  g(t)}  = 0  \label{eq:ola},   \\
\ang{\ZZ_{\uln{\mu(t)}} \mid g(t) } = 0  \label{eq:omu},  \\
%\EQ{ 
\bfd_+(\vec \psi(t)) \le \| (g(t), \p_t \psi(t)) \|_{\HH_0}^2 + \frac{\la(t)}{\mu(t)} \le C \bfd_+(\vec\psi(t))\label{eq:gdotgd} . 
\end{gather} 
Moreover, 
\begin{align} 
\| (g(t), \p_t \psi(t) ) \|_{\HH_0} \le C\left( \frac{\la(t)}{\mu(t)} \right)^{1/2} \label{eq:gH} , 
\end{align} 
and hence 
\EQ{ \label{eq:lamud1}
\bfd_+(\vec \psi(t)) \simeq \frac{\la(t)}{\mu(t)}. 
}
%and that for all $t \in J$ we have 
%\EQ{
%\frac{1}{2} \sup_{s \in J} \mu(s) \le \mu(t)  \le 2 \inf_{s \in J} \mu(s)
%} 
Finally, we have the explicit bound for the kinetic energy
\begin{align}
\| \p_t \psi(t) \|_{L^2}^2 \leq 16 \frac{\la(t)}{\mu(t)} + o\left ( \frac{\la(t)}{\mu(t)} \right ). 
\label{eq:leading-psit}
\end{align}
\end{lem}

\begin{rem}
The little-oh term in \eqref{eq:leading-psit} depends on the parameter $L$,  but it will be important that the leading order term is \emph{independent} of $L$.
\end{rem}

Given the modulation parameters $\la(t), \mu(t)$ we define  
\EQ{ \label{eq:gdef} 
	&g(t) := \psi(t) - Q_{\la(t)} + Q_{\mu(t)} \\
	&\dot g(t):= \p_t \psi(t).
}
Then the vector $\vec g:= (g, \dot g)$ satisfies the equations
\begin{align} \label{eq:ptg} 
&\p_t g = \dot{g} + \la' \La Q_{\ula} - \mu' \La Q_{\umu}, \\
& \p_t \dot g = \p_r^2 g + \frac{1}{r} \p_r g  - \frac{1}{r^2} \left( f( Q_\la - Q_\mu + g) - f(Q_\la) + f(Q_\mu)\right) \label{eq:ptgdot}.
\end{align} 

As a first step towards understanding the behavior of the modulation parameters, we establish bounds on the first derivatives of $
\la(t), \mu(t)$.  This information is not enough to study the interaction of the bubbles for the near two-bubble solution $\vec \psi(t)$ and achieve the goal outlined at the start of the section.  This should also be intuitively clear since $\vec \psi(t)$ satisfies a second-order equation in time, and thus, the interaction of the bubbles should be governed by second derivatives of $\la(t), \mu(t)$. 

\begin{prop} \label{p:modp} 
	% Let $\de>0$ be arbitrary and
	There exists a constant $C > 0$ and $\eta_0 = \eta_0(L) > 0$ with the following property.  Let $J \subset \R$, and let
	$\vec \psi(t)$ be a solution to~\eqref{eq:wmk} on $J$ such that 
	\EQ{
		\bfd(\vec \psi(t))\le \eta_0 \quad \forall t \in J.
	}
	Let $\la(t), \mu(t)$ be the modulation parameters given by Lemma~\ref{l:modeq}. Then for all $t \in J$, we have: 
	\begin{align}
	\abs{\la'(t)} &\leq C (\log L)^{-1/2} 
	\left ( \frac{\la(t)}{\mu(t)} \right )^{1/2}, \label{eq:la'} \\ 
	\abs{ \mu'(t)}& \leq C (\log L)^{-1/2} 
	\left ( \frac{\la(t)}{\mu(t)} \right )^{1/2}.  \label{eq:mu'} 
	\end{align} 
\end{prop}

\begin{proof}
	Differentiating the orthogonality conditions~\eqref{eq:ola} and~\eqref{eq:omu} and using \eqref{eq:ptg} we obtain the relations
	\begin{align*}
		- \ang{ \ZZ_{\ula} \mid \dot g} &=    \la' \left(\ang{\ZZ_{\ula} \mid \La Q_{\ula}} -  \ang{\frac{1}{\la} [\La_0 \ZZ]_{\ula} \mid  g} \right) -  \mu' \ang{ \ZZ_{\ula} \mid \La Q_{\umu}}, \\
		- \ang{ \ZZ_{\umu} \mid \dot g} &=    \la' \ang{\ZZ_{\umu} \mid \La Q_{\ula}} +   \mu' \left({-} \ang{ \ZZ_{\umu} \mid \La Q_{\umu}} -  \ang{\frac{1}{\mu} [\La_0 \ZZ]_{\umu} \mid  g} \right).
	\end{align*}
	These two equations yield the following linear system for $(\la', \mu')$, 
	\EQ{
		\pmat{ A_{11} & A_{12} \\ A_{21} & A_{22}} \pmat{ \la' \\ \mu'} =  \pmat{ {-}\ang{\ZZ_{\ula} \mid  \dot g} \\ -  \ang{ \ZZ_{\umu} \mid \dot g} } \label{eq:Mmat} 
	}
	where
	\EQ{ 
		& A_{11} :=  \ang{\ZZ_{\ula} \mid \La Q_{\ula}} -  \ang{\frac{1}{\la} [\La_0 \ZZ]_{\ula} \mid  g},  \\
		& A_{12} := -\ang{ \ZZ_{\ula} \mid \La Q_{\umu}},  \\ 
		&A_{21}:= \ang{\ZZ_{\umu} \mid \La Q_{\ula}},  \\
		& A_{22}:= - \ang{ \ZZ_{\umu} \mid \La Q_{\umu}} -  \ang{\frac{1}{\mu} [\La_0 \ZZ]_{\umu} \mid  g}.
	}
	We now estimate the coefficients of the matrix $A = (A_{ij})$ so that we may invert \eqref{eq:Mmat} and obtain estimates for $(\la', \mu')$. We define 
	\begin{align}
	\al_L := \ang{ \ZZ \mid \La Q} = \int_0^\infty \chi_L |\La Q|^2 r dr.
	\end{align}
    Note that since $|\La Q(r)| \lesssim \frac{1}{1+r}$, we have for all $L > 0$ sufficiently large
    \begin{align}
    \log L \lesssim \al_L \lesssim \log L, \label{eq:al_est}
    \end{align}
    where the implied constants are absolute. 
    
	\begin{claim}
	For $\lambda/ \mu$ sufficiently small (depending on $L$), the diagonal terms satisfy 
	\begin{align}
	A_{11} &= \al_L\left [ 1 + O_L((\la / \mu)^{1/2})\right ], \label{eq:A11est}, \\
	A_{22} &= -\al_L\left [ 1 + O_L((\la/ \mu)^{1/2})\right] \label{eq:A22est}. 
	\end{align}
	\end{claim}
	To prove the claim we simply observe that 
	\begin{align*}
	\left | 
	\ang{\frac{1}{\la} [\La_0 \ZZ]_{\ula} \mid  g} 
	\right | \lesssim \| g \|_{L^\infty} \| \La_0 \ZZ \|_{L^1(rdr)} \lesssim_L \|g \|_{H} \lesssim_L (\la/\mu)^{1/2}. 
	\end{align*}
	Thus, 
	\begin{align*}
	A_{11} = \ang{\ZZ_{\ula} \mid \La Q_{\ula}} -  \ang{\frac{1}{\la} [\La_0 \ZZ]_{\ula} \mid  g} = \al_L + O_L((\la/\mu)^{1/2}),
	\end{align*}
	which establishes \eqref{eq:A11est}.  The estimate \eqref{eq:A22est} is established analogously, and the claim is proved. 

    We now estimate the off-diagonal terms. 
	\begin{claim} \label{c:M12est}
		For $\lambda/ \mu$ sufficiently small (depending on $L$) we have
		\begin{align} \label{eq:M12}
		|A_{12}| \lesssim_L ({\la/\mu})^{2}, \quad 
		|A_{21}| \lesssim \log L,
		\end{align} 
where the implied constant in the estimate for $A_{21}$ is absolute. 
	\end{claim}
 
	Since $r \ZZ(r) \in C^\infty_0$, $\la / \mu \ll 1$ and $|\La Q| \lesssim r$ for small $r$, we conclude that 
	 \begin{align*}
	 |A_{12}| = \Bigl | 
\ang{ \ZZ_{\ula} \mid \La Q_{\umu}} 
	 \Bigr | 
	 = \left | 
\int_0^{2L \la / \mu} \frac{\mu r}{\la} 
Z ( r \mu / \la ) \La Q(r) dr \right |
	 \lesssim_L \int_0^{2L \la / \mu} |\La Q| dr \lesssim_L (\la / \mu)^2. 
	 \end{align*}
	 This proves the first estimate in \eqref{eq:M12}.
	 Let $\s = \la/\mu$. By a change of variables and the explicit expression for $\cl Z$ we have 
	 \begin{align*}
	 |A_{21}| = \Bigl | 
\ang{ \ZZ_{\umu} \mid \La Q_{\ula}} 
	 \Bigr | 
	 &= 	 \Bigl | 
	 \ang{ \ZZ \mid \La Q_{\underline{\la/\mu}}} \Bigr | \\
	&\lesssim \frac{1}{\s}  
\int_0^{2L} \frac{r}{1+r^2} \frac{(r/\sigma)}{1+ (r/\s)^2} r dr \\
&\lesssim \log L
	 \end{align*}
	which proves the second estimate in \eqref{eq:M12} and the claim.
	
	We now solve for  $(\la', \mu')$ by inverting $A$:
	\EQ{
		\pmat{  \la' \\ \mu ' } =  \frac{1}{\det A}  \pmat{ - A_{22} \ang{ \ZZ_{\ula} \mid  \dot g} + A_{12} \ang{\ZZ_{\umu} \mid  \dot g}  \\  A_{21}  \ang{\ZZ_{\ula} \mid  \dot g}- A_{11}  \ang{\ZZ_{\umu} \mid  \dot g}}
	}
	The previous two claims imply that 
	\EQ{
		\det A =   A_{11}A_{22} - A_{12}A_{21} = -\al_L^2\left [1 + O_L((\la/\mu)^{1/2})\right] \label{eq:detA}
	}
	as long as $\la / \mu$ is sufficiently small. It is easy to see that the function $\ZZ = \chi_L \La Q$ satisfies $\| \ZZ \|_{L^2} \lesssim (\log L)^{1/2}$.  Then by Cauchy-Schwarz and \eqref{eq:leading-psit} we have, for $\la/ \mu$ sufficiently small, 
	\begin{align}
	\left | \ang{Z_{\ula} \mid \dot g} \right | + 	\left | \ang{Z_{\umu} \mid \dot g} \right | \lesssim |\log L|^{1/2}
	(\la / \mu)^{1/2} \lesssim \al_L^{1/2} (\la / \mu)^{1/2} \label{eq:rest}
	\end{align}
	where the implied constant is absolute.  Our two claims, \eqref{eq:detA} and \eqref{eq:rest} imply that as long as $\la/\mu$ is sufficiently small  
	\begin{align*}
	|\la'| &\lesssim |\det A|^{-1} \left ( 
	|A_{22}| \left | \ang{Z_{\ula} \mid \dot g} \right | + 
	|A_{12}| \left | \ang{Z_{\umu} \mid \dot g} \right |
	\right ) \\
	&\lesssim \al_L^{-1/2} (\la / \mu)^{1/2} \\
	&\lesssim (\log L)^{-1/2} (\la / \mu)^{1/2}
	\end{align*}
	as desired. 
	A similar argument establishes
	\EQ{ 
		\abs{ \mu'} \lesssim (\log L)^{-1/2} (\la / \mu)^{1/2}
	}
	as well which finishes the proof. 
\end{proof}

\subsection{Refined control of the modulation parameters}\label{s:modest} 

As stated previously, information about the first derivatives of the modulation parameters is not enough to study the evolution of two-bubbles since \eqref{eq:wmk} is second order in time. Due to the slow decay of the $\La Q$, we will in fact need to study second order derivatives of $2 \la \left |\log \la/\mu \right |$ and $\mu$.  Moreover, for technical reasons we will study a function $\zeta = \zeta(t)$ which approximates $2 \la |\log \la/\mu|$ and a function $b = b(t)$ which approximates $\zeta'(t)$ (see Proposition \ref{p:modp2}).  

We first define a truncated virial functional and state some relevant properties. This functional played a fundamental role in the work of Jendrej and Lawrie on threshold dynamics for higher equivariant wave maps \cite{JL} and in the two-bubble construction by Jendrej in~\cite{JJ-AJM}.  It will play a very important role in our work as well. 
For the proofs of the following statements we refer the reader to~\cite[Lemma 4.6]{JJ-AJM} and~\cite[Lemma 5.5]{JJ-AJM}. In what follows, we denote the nonlinearity by 
$f(\rho):= \frac{1}{2} \sin 2 \rho$. 

\begin{lem} \emph{\cite[Lemma 4.6]{JJ-AJM}}
  \label{lem:fun-q}
  For each $c, R > 0$ there exists a function $q(r) = q_{c, R}(r) \in C^{3,1}((0, +\infty))$ with the following properties:
  \begin{enumerate}[label=(P\arabic*)]
    \item $q(r) = \frac{1}{2} r^2$ for $r \leq R$, \label{enum:approx-q}
    \item there exists an absolute constant $\kappa > 0$  such that $q(r) \equiv \tx{const}$ for $r \geq \ti R := \kappa e^{\kappa/c} R$, \label{enum:support-q}
    \item $|q'(r)| \lesssim r$ and $|q''(r)| \lesssim 1$ for all $r > 0$, with constants independent of $c, R$, \label{enum:gradlap-q}
    \item $q''(r) \geq -c$ and $\frac 1r q'(r) \geq -c$, for all $r > 0$, \label{enum:convex-ym}
    \item $(\frac{d^2}{d r^2} + \frac 1r \frac{d}{dr} r)^2 q(r) \leq c\cdot r^{-2}$, for all $r > 0$, \label{enum:bilapl-ym}
    \item $\big|r\big(\frac{q'(r)}{r}\big)'\big| \leq c$, for all $r > 0$. \label{enum:multip-petit-ym}
  \end{enumerate}
\end{lem}

For each $\lambda > 0$ we define the operators $\A(\lambda)$ and $\A_0(\lambda)$ as follows:
\begin{align}
  [\A(\lambda)g](r) &:= q'\big(\frac{r}{\lambda}\big)\cdot \p_r g(r), \label{eq:opA-wm} \\
  [\A_0(\lambda)g](r) &:= \big(\frac{1}{2\lambda}q''\big(\frac{r}{\lambda}\big) + \frac{1}{2r}q'\big(\frac{r}{\lambda}\big)\big)g(r) + q'\big(\frac{r}{\lambda}\big)\cdot\p_r g(r). \label{eq:opA0-wm}
\end{align}
Since $q(r) = \frac{1}{2} r^2$ for $r \leq R$, $\A(\la) g(r) = \frac{1}{\la} \La g(r)$ and $\A_0(\la) g(r) = \frac{1}{\la} \La_0 g(r)$ for $r \leq R$.  One may intuitively think of $\A(\la)$ and $\A_0(\la)$ as extensions of $\frac{1}{\la} \La$ and $\frac{1}{\la} \La_0$ to $r \geq R$ which have good boundedness properties.  The following lemma makes this precise.  In what follows, we denote 
\EQ{
X:= \{ g \in H \mid \frac{g}{r}, \p_r g \in H\}.
}

\begin{lem} \emph{\cite[Lemma 5.5]{JJ-AJM}}
  \label{lem:op-A-wm}
  Let $c_0>0$ be arbitrary. There exists $c>0$ small enough and $R, \ti R>0$ large enough in Lemma~\ref{lem:fun-q} so that the operators $\A(\lambda)$ and $\A_0(\lambda)$ defined in~\eqref{eq:opA-wm} and~\eqref{eq:opA0-wm} have the following properties:
  \begin{itemize}[leftmargin=0.5cm]
    \item the families $\{\A(\lambda): \lambda > 0\}$, $\{\A_0(\lambda): \lambda > 0\}$, $\{\lambda\partial_\lambda \A(\lambda): \lambda > 0\}$
      and $\{\lambda\partial_\lambda \A_0(\lambda): \lambda > 0\}$ are bounded in $\mathscr{L}(H; L^2)$, with the bound depending only on the choice of the function $q(r)$,
    \item %Let $f(\rho):= \frac{k^2}{2}\sin 2 \rho$.  
    For all $\lambda > 0$ and $g_1, g_2 \in X$  there holds
      \begin{multline}  \label{eq:A-by-parts-wm}
      \Big| \ang{ \A(\lambda)g_1\mid  \frac{1}{r^2}\big(f(g_1 + g_2) - f(g_1) - f'(g_1)g_2\big)}  \\ +\ang{ \A(\lambda)g_2\mid \frac{1}{r^2}\big(f(g_1+g_2) - f(g_1) -g_2\big)}\Big| 
        \leq \frac{c_0}{\lambda} \|g_2\|_H^2, 
      \end{multline}
     % with a constant $\eps_0$ arbitrarily small,
    \item For all $g \in X$ we have  %for any $\eps_0 > 0$, if the constants $c$ and $R$ in the definition of $q(r)$ are chosen appropriately, then
\EQ{
        \label{eq:A-pohozaev-wm}
        \ang{\A_0(\lambda)g | \big(\partial_r^2 + \frac 1r\partial_r - \frac{1}{r^2}\big)g} \leq \frac{c_0}{\lambda}\|g\|_{H}^2 - \frac{1}{\lambda}\int_0^{R\lambda}\Big((\partial_r g)^2 + \frac{1}{r^2}g^2\Big) dr, 
        }
        \item Moreover, for $\la, \mu >0$ with $\la/\mu \ll 1$,
 \begin{gather}
      \label{eq:L-A-wm}
      \|\Lambda Q_\uln\lambda - \A(\lambda)Q_\lambda\|_{L^\infty} \leq \frac{c_0}{\lambda},  \\
      \|\La_0 \Lambda Q_\uln\lambda - \A_0(\lambda) \La Q_\lambda\|_{L^2} \leq c_0, \label{eq:Al2}\\
    \| \A(\la) Q_\mu \|_{L^\infty} + \| \A_0(\la) Q_\mu \|_{L^\infty}  \lesssim \frac{1}{\mu},   \label{eq:Ainfty} 
     \end{gather} 
    and, for any $g \in H$, 
    \begin{multline}   \label{eq:approx-potential-wm}
        \bigg|\int_0^{+\infty}\frac 12 \Big(q''\big(\frac{r}{\lambda}\big) + \frac{\lambda}{r}q'\big(\frac{r}{\lambda}\big)\Big)\frac{1}{r^2}\big(f({-}Q_\mu + Q_\lambda + g) - f({-}Q_\mu + Q_\lambda)-g\big)g dr \\
        - \int_0^{+\infty} \frac{1}{r^2}\big(f'(Q_\lambda)-1\big)g^2 dr\bigg| \leq c_0(\|g\|_H^2 + (\lambda/\mu)).  
  \end{multline} 
  \end{itemize}
\end{lem}

\begin{rem}
The argument for the estimate \eqref{eq:Al2} from \cite{JJ-AJM} does not quite apply our case due to the slow decay of $Q$.  We provide a different argument here. We first note that $\La_0 \La Q = \frac{4r}{(1+r^2)^2} \in L^2(\R^2)$ and the estimate \eqref{eq:Al2} is scaling invariant so we can take $\la = 1$.  Since $\La_0 \La Q = \cl A_0(1) \La Q$ for $r \leq R$ and $\cl A_0(1)\La Q = 0$ for $r \geq \tilde R = R \kappa e^{\kappa/c}$, we have
\begin{align}
\left \|\La_0 \Lambda Q - \A_0(1) \La Q\right \|_{L^2}^2 
\leq \int_{R}^\infty |\La_0 \La Q|^2 r dr + \int_{R}^{\tilde R} |\cl A_0(1) \La Q|^2 r dr. 
\end{align}
The first term on the right-hand side above can be made $<c_0^2/2$ as long as $R > 0$ is sufficiently large since $\La_0 \La Q \in L^2$.  For the second term, we write 
\begin{align*}
\cl A_0(1) \La Q = \frac{r}{2} \left (\frac{q'(r)}{r} \right )' \La Q + \frac{q'(r)}{r} \La_0 \La Q.   
\end{align*}
Then by properties \ref{enum:multip-petit-ym} and \ref{enum:gradlap-q} in Lemma \ref{lem:fun-q} we have 
\begin{align*}
 \int_{R}^{\tilde R} |\cl A_0(1) \La Q|^2 \, r dr
& \lesssim c^2 \int_{R}^{\tilde R} |\La Q|^2 \, r dr 
+ \int_{R}^\infty |\La_0 \La Q|^2 \, rdr \\
&\lesssim c^2 \int_R^{R \kappa e^{\kappa/c}} \frac{1}{r} \, dr 
+ \int_R^\infty \frac{1}{r^5} dr \\
&\lesssim c +R^{-4} \leq c_0^2 /2
\end{align*}
as long as $c$ is sufficiently small and $R$ is sufficiently large.  We conclude that for $c,R$ chosen appropriately, we have 
\begin{align*}
\left \|\La_0 \Lambda Q - \A_0(1) \La Q\right \|_{L^2}^2 \leq c_0^2,
\end{align*}
as desired. 
\end{rem}

As before, we let $\chi \in C^\infty_c(\R^2)$ be a smooth radial cutoff.
We then define the function $b(t)$ by  
\EQ{ \label{eq:bdef} 
b(t):= - \ang{ \chi_{M \sqrt{\la(t)\mu(t)}} \La Q_{\underline{\la(t)}}  \mid \dot g(t)}  - \ang{ \dot g(t) \mid \A_0( \la(t) ) g(t)}.
}
Here $M > 0$ is a constant which we will later fix.  Finally, we define 
\EQ{ \label{eq:zetadef} 
\zeta(t) := 2 \la(t) |\log (\la(t)/\mu(t))| - \langle \chi_{M \sqrt{\la(t) \mu(t)}}\Lambda Q_{\uln{\lambda(t)}} \mid g(t)\rangle
}
Note that $\zeta(t)$ is $C^1$ since $\partial_t g(t)$ is continuous in $L^2$ with respect to $t$.  We will now show that we may roughly view $\zeta(t)$ as $2 \la(t) \log \la(t)$ and $b(t)$ as a subtle correction to $\zeta'(t)$.  The essential feature of this correction is that $b'(t)$ (which intuitively is connected to $\la''(t)$) is bounded from below.  More precisely, we prove the following.   

\begin{prop}[Modulation Control] \label{p:modp2} 
Assume the same hypothesis as in  Proposition~\ref{p:modp}. Let $0 < \de < 1/2$ be arbitrarily small, and let $\eta_0$ be as in Lemma~\ref{l:modeq}. 
There exist functions $L_0 = L_0(\de) > 0$, $M_0 = M_0(\de,L) > 0$ and $\eta_1 = \eta_1(\de,L,M) > 0$ such that if $L > L_0$, $M > M_0$ and $\bfd_+(\vec \psi(t)) \le \eta_1 < \eta_0$, then for all $t \in J$ the functions $\la(t), \mu(t), \zeta(t)$ and $b(t)$ (which implicitly depend on $L$ and $M$) satisfy
\begin{align}
&\abs{\frac{\zeta(t)}{2\la(t)|\log (\la(t)/\mu(t))|} - 1} \le \de , \label{eq:bound-on-l} \\
& \abs{\zeta'(t) - b(t) } \le \de \left [\frac{\la(t)}{\mu(t)}
\right]^{\frac{1}{2}} \left |
\log \frac{\la(t)}{\mu(t)} \right |^{\frac{1}{2}}  \le  \de \left [\frac{\zeta(t)}{\mu(t)}
\right]^{\frac{1}{2}},  \label{eq:kala'}  \\ 
\begin{split}
&\abs{b(t)} \le 4 \left [\frac{\la(t)}{\mu(t)}
\right]^{\frac{1}{2}} \left |2
\log \frac{\la(t)}{\mu(t)} \right |^{\frac{1}{2}}  + \de \left [\frac{\la(t)}{\mu(t)}
\right]^{\frac{1}{2}} \left |
\log \frac{\la(t)}{\mu(t)} \right |^{\frac{1}{2}}   \le 5\left [\frac{\zeta(t)}{\mu(t)}
\right]^{\frac{1}{2}}.
\end{split} \label{eq:b-bound} 
\end{align}

Moreover, $b(t)$ is locally Lipschitz and there exists $C_1 = C_1(L) > 0$ such that 
\begin{align}
&|b'(t)| \leq C_1 /\mu(t), \label{eq:b'} \\
&b'(t) \ge (8 - \de) /\mu(t). \label{eq:b'lb}  
\end{align} 
\end{prop}

\begin{proof} Since we will take $\eta_1 < \eta_0$, the modulation parameters are well-defined and $C^1$ on the interval $J$.   We also note that by rescaling $\vec \psi(t_0)$ for some $t_0 \in J$ and shrinking the interval $J$ if necessary, we can assume that $\frac 12 \le \mu(t) \le 2$ on $J$.  Throughout the argument, implied constants and big-oh terms will depend on the parameters $L$ and $M$ unless stated otherwise.  

We first prove \eqref{eq:bound-on-l}. By Proposition \ref{p:modp} we have $\|g\|_{L^\infty} \leq \| g \|_H \lesssim \lambda^\frac 12$. Thus,
\begin{align*}
\left | 
\ang{\chi_{M \sqrt{ \mu \la}} \La Q_{\ula} \mid g} 
\right | &\lesssim 
\la^{3/2} \int_0^{4M / \sqrt{\la}} |\La Q| r dr \\
&\lesssim  \la.  
\end{align*}
We conclude that 
\begin{align*}
\frac{1}{2\la |\log (\la/\mu)|}\left | 
\ang{\chi_{\sqrt{\mu \la}} \La Q_{\ula} \mid g} 
\right | \lesssim |\log \la|^{-1} 
\end{align*}
which can be made smaller than $\de$ as long as $\la/\mu$ is sufficiently small compared to $L$ and $M$. This proves \eqref{eq:bound-on-l}. 

Now we prove~\eqref{eq:kala'}.
From \eqref{eq:ptg} we have
\begin{equation}
\label{eq:bound-on-l-666}
\begin{aligned}
\frac{d}{dt} \ang{\chi_{M \sqrt{\la \mu}} \Lambda Q_{\ula} \mid g} &= \ang{\chi_{M \sqrt{\la \mu}} \Lambda Q_{\ula}\mid \dot g} + \lambda'\ang{ \chi_{M\sqrt{\la \mu}} \Lambda Q_{\ula}\mid \Lambda Q_{\ula}} \\
&- \mu'\ang{\chi_{M \sqrt{\la \mu}} \Lambda Q_{\ula}\mid \Lambda Q_{\umu}} 
-\frac{\lambda'}{\lambda} \ang{\chi_{M \sqrt{\la \mu}} \Lambda_0\Lambda Q_{\ula} \mid g} \\&- 
\Bigl ( \frac{\la'}{2\la} + \frac{\mu'}{2\mu} \Bigr ) \ang{ \Lambda \chi_{ M \sqrt{\la \mu}} \Lambda Q_{\ula}\mid g},
\end{aligned}
\end{equation}
Since $|\La Q| \lesssim r^{-1}$, 
\begin{align*}
\int_{\sqrt{\la \mu}}^{2M \sqrt{\la \mu}} |\La Q_{\ula}|^2 rdr 
\lesssim 
\int_{\sqrt{\frac{\mu}{\la}}}^{2M \sqrt{\frac{\mu}{\la}}} r^{-1} dr 
\lesssim 1. 
\end{align*}
Thus, 
\begin{align*}
\la' \int_0^\infty \chi_{M \sqrt{\la \mu}} |\La Q_{\ula}|^2 r dr &= 
\la' \int_0^{\sqrt{\la \mu}} |\La Q_{\ula}|^2 rdr + O(\la') \\
&= 2 \la' |\log (\la/\mu)| + O(\la^{1/2}). 
\end{align*}
We now show that the remaining terms on the right hand side of \eqref{eq:bound-on-l-666} are $\ll |\log \la|^{1/2} \lambda^{1/2}$ for all $L$ and $M$ large and $\la/\mu$ sufficiently small compared to $L$ and $M$.
Since we are assuming  $\frac 12 \le  \mu \leq 2$, we have by \eqref{eq:mu'} 
  \EQ{
  |\mu'|\abs{\ang{\chi_{M \sqrt{\la \mu}} \Lambda Q_{\ula}\mid \Lambda Q_{\umu}}}& \lesssim \la^{1/2} \int_0^{4M\sqrt{\la}} \frac{1}{\la} \frac{r}{\la} |Q_r(r/ \la)| \frac{1}{\mu} \frac{r}{\mu}  |Q_{r}(r/ \mu)| \, r d r \\
  & \lesssim \la^{-1/2} \int_0^{4M\sqrt{\la}}  \frac{ (r/ \la)}{ 1+ (r/\la)^{2}} \frac{r^{2}}{ 1+ r^{2}} \, d r \\
  & \lesssim \la^{1/2} \int_0^{4M\sqrt{\la}}  \frac{ r^{2}}{ \la^{2} + r^{2}} \frac{r}{ 1+ r^{2}} \, d r  \\
  &\lesssim \la^{3/2}. 
  }
Thus, the third term in \eqref{eq:bound-on-l-666}
 is $\ll \la^{1/2} |\log \la |^{1/2}$. 
For the fourth term, we have
\begin{equation}
\Big|\frac{\lambda'}{\lambda}\ang{\chi_{M \sqrt{\la\mu}} \Lambda_0\Lambda Q_{\ula}\mid g}\Big| \lesssim |\lambda'|\|g\|_{L^\infty}\left \|\chi_{M \sqrt{\mu/\la}}\Lambda_0\Lambda Q\right \|_{L^1} \lesssim \lambda \|\chi_{4M/\sqrt{\lambda}}\Lambda_0\Lambda Q\|_{L^1}. 
\end{equation}
Now $\La_0 \La Q = \frac{4r}{(1+r^2)^2}$ so 
\begin{align*}
\|\chi_{4M/\sqrt{\lambda}}\Lambda_0\Lambda Q\|_{L^1} \lesssim 1.
\end{align*}
Thus, 
\begin{align*}
\Big|\frac{\lambda'}{\lambda}\ang{\chi_{M \sqrt{\la\mu}} \Lambda_0\Lambda Q_{\ula}\mid g}\Big| \lesssim \la \ll \la^{1/2} |\log \la|^{1/2}. 
\end{align*}
For the fifth term appearing in \eqref{eq:bound-on-l-666}, we have 
\begin{align*}
\left |  \ang{ \Lambda \chi_{M \sqrt{\la \mu}} \Lambda Q_{\ula}\mid g}
\right | &\lesssim \| g \|_{L^\infty} \int_{M \sqrt{\la \mu}}^{2M \sqrt{\la \mu}} |\La Q_{\ula}| r dr \\
&\lesssim \la^{3/2} \int_{M/2\sqrt{\la}}^{4M/\sqrt{\la}} |\La Q| r dr \\
&\lesssim \la . 
\end{align*}
By \eqref{eq:la'} and \eqref{eq:mu'} we conclude that for all $\la$ sufficiently small depending on $L$ and $M$,  
\begin{align*}
\left |
\Bigl ( \frac{\la'}{2\la} + \frac{\mu'}{2\mu} \Bigr ) \ang{ \Lambda \chi_{M \sqrt{\la \mu}} \Lambda Q_{\ula}\mid g}
\right | \lesssim  \la^{1/2} \ll \la^{1/2} |\log \la|^{1/2}. 
\end{align*}

From \eqref{eq:bound-on-l-666} and the previous bounds we conclude that
\begin{equation}
\label{eq:bound-on-l-6666}
\left |2 \la' \log (\la/\mu) - \frac{d}{dt }\left \langle \chi_{M \sqrt{\la \mu}}\Lambda Q_{\uln{\lambda}} \mid g \right \rangle + \ang{ \chi_{M \sqrt{\la \mu}}\Lambda Q_{\ula}\mid \dot g} \right | \ll \la^{1/2} |\log \la|^{1/2}.
\end{equation}
By \eqref{eq:la'} and \eqref{eq:mu'}
\begin{align*}
\frac{d}{dt} 2 \la \log (\la / \mu) = 2 \la' \log (\la / \mu) + 2 (\la' \mu - \mu' \la ) / \mu = 2 \la' \log (\la / \mu) + O(\la^{1/2}).
\end{align*}
From this estimate and \eqref{eq:bound-on-l-6666} we obtain 
\begin{align}
\left |\zeta' + \ang{\chi_{\sqrt{\la \mu}}\Lambda Q_{\ula}\mid \dot g} \right | \ll \la^{1/2} |\log \la |^{1/2}. \label{eq:bound6667}
\end{align}
Recall that 
 \EQ{
 b(t):= - \ang{ \chi_{\sqrt{\la \mu}}\La Q_{\ula} \mid  \dot g}  - \ang{ \dot g \mid \A_0(\la) g}
}
 By~\eqref{eq:gH}  and Lemma~\ref{lem:op-A-wm} we have
\EQ{  \label{eq:dotgAg} 
 \left | \ang{ \dot g \mid \A_0(\la) g} \right | \lesssim  \| \dot g \|_{L^2} \| \A_0(\la) g \|_{L^2} \lesssim \| (g, \dot g) \|_{\HH_0}^2 \lesssim \la \ll \la^{1/2} |\log \la|^{1/2}.
 }
This estimate and \eqref{eq:bound6667} imply
\EQ{
\abs{\zeta' - b}  \ll \lambda^{1/2} |\log \la|^{1/2}.
}
for $L$ and $M$ large and $\la/\mu$ sufficiently small depending on $L$ and $M$.  This completes the proof of \eqref{eq:kala'}. 

To prove \eqref{eq:b-bound}, we argue as above and obtain 
\EQ{
\abs{b(t)} &\le \| \chi_{M \sqrt{\la \mu}} \La Q_{\ula}\|_2 \| \p_t \psi\|_{2}  - O(\la) \\
&= \bigl [ 2 \log (\la / \mu) + O(1) \bigr ]^{1/2} \| \p_t \psi \|_2 - O(\la).
} 
By \eqref{eq:leading-psit} we have 
\EQ{
\| \p_t \psi(t) \|_2^2 \le 16 (\la/ \mu)  + o(\la).
}
The previous two estimates combined yield ~\eqref{eq:b-bound}. 

We now turn to proving ~\eqref{eq:b'lb} and ~\eqref{eq:b'}.  By approximating the initial data $\vec \psi(t_0)$ for some $t_0 \in J$ by smooth functions and using the well-posedness theory, we may assume that $\vec \psi(t)$ is smooth on $J$. We differentiate $b(t)$ and use the formulae~\eqref{eq:ptg},~\eqref{eq:ptgdot} to obtain 
\EQ{ \label{eq:b'1} 
b'(t) &= \frac{\la'}{\la} \ang{\chi_{M\sqrt{\la \mu}}[ \La_0 \La Q]_{\ula} \mid \dot g }   - \ang{ \chi_{M\sqrt{\la \mu}}\La Q_{\ula}  \mid \p_t \dot g} 
 - \ang{\p_t  \dot g \mid \A_0( \la) g}  \\ 
& \quad - \frac{\la'}{\la} \ang{ \dot g \mid  \la \p_\la \A_0(\la) g}  -  \ang{ \dot g \mid \A_0(\la) \p_t g} \\
& \quad + \left ( \frac{\la'}{2\la} + \frac{\mu'}{2\mu} \right ) 
\ang{\La \chi_{M\sqrt{\la \mu}} \La Q_{\ula} \mid \dot g } \\
& = \frac{\la'}{\la} \ang{\chi_{M \sqrt{\la \mu}} [ \La_0 \La Q]_{\ula} \mid \dot g }   \\
 &\quad -  \ang{ \chi_{M \sqrt{\la \mu}} \La Q_{\ula} \mid  \p_r^2 g + \frac{1}{r} \p_r g  - \frac{1}{r^2} \left( f( Q_\la - Q_\mu + g) - f(Q_\la) +f(Q_\mu)\right)} \\
& \quad - \ang{ \p_r^2 g + \frac{1}{r} \p_r g  - \frac{1}{r^2} \left( f( Q_\la - Q_\mu + g) - f(Q_\la) + f(Q_\mu)\right) \mid \A_0(\la) g}  \\
& \quad - \frac{\la'}{\la} \ang{ \dot g \mid  \la \p_\la \A_0(\la) g} - \ang{ \dot g \mid \A_0(\la) \dot g} \\
&\quad  -  \la' \ang{ \dot g \mid \A_0(\la) \La Q_{\ula}} 
 + \mu' \ang{ \dot g \mid \A_0(\la) \La Q_{\umu}} \\
&\quad +  \left ( \frac{\la'}{2\la} + \frac{\mu'}{2\mu} \right ) 
\ang{\La \chi_{M\sqrt{\la \mu}} \La Q_{\ula} \mid \dot g }.
}

We first discard those terms which are $\ll 1$ as long as $L > 0 $ is sufficiently large, $M > 0$ is sufficiently large depending on $L$, and $\la / \mu$ is sufficiently small depending on $L$ and $M$. Consider the last term appearing above.  Here we will choose the size of $L$.  For some absolute constant $C_2 > 0$, we have
\begin{align}\label{eq:Qanulus2}
\| \La \chi_{M \sqrt{\la \mu}} \La Q_{\ula} \|_{L^2} \leq C_2.
\end{align}
If $C$ is the constant in \eqref{eq:la'}, then we choose $L > 0$ so large so that 
\begin{align}\label{eq:chooseL}
80 C C_2 (\log L)^{-1/2} \leq \frac{\de}{100}.  
\end{align}
Then by Cauchy Schwarz, \eqref{eq:Qanulus2}, \eqref{eq:leading-psit} and \eqref{eq:chooseL}, we conclude that 
\begin{align}
\left | 
\frac{\la'}{\la} 
\ang{\La \chi_{M\sqrt{\la \mu}} \La Q_{\ula} \mid \dot g }
\right | \leq \frac{|\la'|}{\la} \| \La \chi_{M \sqrt{\la \mu}} \La Q_{\ula} \|_{L^2} \| \dot g \|_{L^2}
\leq \frac{2 C (\log L)^{-1/2} \la^{1/2}}{\la} C_2 40 \la^{1/2} 
\leq \frac{\de}{100}
\end{align}
as long as $\la / \mu$ is sufficiently small.  Similarly, we have
\begin{align}
\left | 
\frac{\mu'}{\mu} 
\ang{\La \chi_{M\sqrt{\la \mu}} \La Q_{\ula} \mid \dot g }
\right | \leq \frac{|\la'|}{\la} \| \La \chi_{M \sqrt{\la \mu}} \La Q_{\ula} \|_{L^2} \| \dot g \|_{L^2}
\leq 4 C_1 (\log L)^{-1/2} \la^{1/2} C_2 40 \la^{1/2} 
\leq \frac{\de}{100} 
\end{align}
as long as $\la / \mu$ is sufficiently small.  Thus, the last term above can be made $\leq \de/100$.  
We now consider the first and sixth term appearing above. By Cauchy Schwarz and the fact that $\La_0 \La Q \in L^2$, we have 
\begin{align*}
\left | 
\frac{\la'}{\la} \ang{(1-\chi_{M \sqrt{\la \mu}}) [ \La_0 \La Q]_{\ula} \mid \dot g }  
\right | 
\lesssim \frac{|\la'|}{\la}\| \dot g \|_{L^2} \| \La_0 \La Q \|_{L^2(r \geq M \sqrt{\mu/\la})} \lesssim \| \La_0 \La Q \|_{L^2(r \geq M \sqrt{\mu/\la})} \ll 1. 
\end{align*}
Then the first term and the sixth term combined yield 
\begin{align*}
\frac{\la'}{\la} \ang{\chi_{M \sqrt{\la \mu}} [ \La_0 \La Q]_{\ula} \mid \dot g }  -   \la' \ang{ \dot g \mid \A_0(\la) \La Q_{\ula}}
&=
 \frac{\la'}{\la} \ang{[ \La_0 \La Q]_{\ula} \mid \dot g }  -   \la' \ang{ \dot g \mid \A_0(\la) \La Q_{\ula}} + o(1)\\
 &= \frac{\la'}{\la}  
 \ang{[ \La_0 \La Q]_{\ula} - \A_0(\la) \La Q_{\ula} \mid \dot g} + o(1),
\end{align*}
where the little-oh satisfies $|o(1)| \ll 1$ as long as $L > 0 $ is sufficiently large, $M > 0$ is sufficiently large depending on $L$, and $\la / \mu$ is sufficiently small depending on $L$ and $M$.
By \eqref{eq:Al2} 
\begin{align*}
\frac{|\la'|}{\la} \left | 
\ang{[ \La_0 \La Q]_{\ula} - \A_0(\la) \La Q_{\ula} \mid \dot g} \right |
\leq C \la^{-1/2} \| \dot g \|_{L^2} \| [ \La_0 \La Q]_{\ula} - \A_0(\la) \La Q_{\ula} \|_{L^2} \lesssim c_0 \ll 1,
\end{align*}
as long as $c_0$ is sufficiently small.  We conclude that
\begin{align}
\left | \frac{\la'}{\la} \ang{ \chi_{M \sqrt{\la \mu}} [ \La_0 \La Q]_{\ula} \mid \dot g }  -   \la' \ang{ \dot g \mid \A_0(\la) \La Q_{\ula}} \right | 
\ll 1. 
\end{align}

Since $(\la \p_\la \A_0(\la)): H \to L^2$ is bounded, we have that the fourth term satisfies  
\EQ{
\left |
\frac{\la'}{\la} \ang{ \dot g \mid (\la \p_\la \A_0(\la)) g}
\right | \lesssim \la^{-\frac 12} \|(g, \dot g) \|_{\HH_0}^2 \lesssim \la^{\frac 12} \ll 1. 
}
Via integration by parts, the fifth term appearing above satisfies 
\EQ{
	\ang{ \dot g \mid \A_0(\la) \dot g} = 0.
}
Finally, since $1/2 \leq \mu \leq 2$ we have  
\EQ{
\left | \mu' \ang{ \dot g \mid \A_0(\la) \La Q_{\umu}} \right | &= \frac{|\mu'|}{\mu} \left | \ang{ \dot g \mid \A_0(\la) \La Q_{\mu}} \right |\\
& \lesssim  \abs{ \mu'} \| \dot g \|_{L^2} \lesssim  \la \ll 1. 
}
 
We now introduce some notation.  Until the end of the proof, we write $ A \simeq B$ if $A = B$ up to terms which which can be made $< \de$ as long as $L > 0 $ is sufficiently large, $M > 0$ is sufficiently large depending on $L$, and $\la / \mu$ is sufficiently small depending on $L$ and $M$.
We have shown so far that 
\EQ{ \label{eq:b'2}
b'(t) & \simeq   -  \ang{ \chi_{M \sqrt{\la \mu}} \La Q_{\ula} \mid  \p_r^2 g + \frac{1}{r} \p_r g  - \frac{1}{r^2} \left( f( Q_\la - Q_\mu + g) - f(Q_\la) + f(Q_\mu)\right)} \\
& \quad - \ang{ \p_r^2 g + \frac{1}{r} \p_r g  - \frac{1}{r^2} \left( f( Q_\la - Q_\mu + g) - f(Q_\la) + f(Q_\mu)\right) \mid \A_0(\la) g}  
}
We now choose the size of $M > 0$ (depending on $L$).  Recall that 
\EQ{
 \LL_\la \La Q_{\ula} := \left (- \p_{rr} - \frac{1}{r} \p_r + \frac{f'(Q_\la)}{r^2} \right ) \La Q_{\ula} = 0. 
}
In fact, since we have the factorization $\LL_\la = A_\la^* A_\la$ with 
$A_\la = -\p_r + \frac{\cos Q_\la}{r}$,  we must have 
\begin{align*}
A_\la \La Q_{\ula} = 0.  
\end{align*}
Thus, 
\EQ{
 \ang{ \chi_{M\sqrt{\la \mu}} \La Q_{\ula} \mid  \p_r^2 g + \frac{1}{r} \p_r g - \frac{f'(Q_\la)}{r^2} g } &=  -\ang{ \chi_{M\sqrt{\la \mu} } \La Q_{\ula} \mid 
 A^*_{\la} A_\la g} \\
&= -\ang{ A_\la \left ( 
	\chi_{M \sqrt{\la \mu}} \La Q_{\ula} 
	\right ) \mid A_\la g } \\
&= \frac{1}{M \sqrt{\la \mu}} \ang{ 
	\chi'_{M \sqrt{\la \mu}} \La Q_{\ula} 
	 \mid A_\la g }. 
}
Since $\chi'_{M \sqrt{\la \mu}}$ is bounded by 2 and is supported on the annulus $\{ M \sqrt{\la \mu} \leq r \leq 2M \sqrt{\la \mu}\}$, Cauchy-Schwarz and Proposition \ref{p:modp} imply
\begin{align*}
\frac{1}{M \sqrt{\la \mu}} \left | \ang{ 
	\chi'_{M \sqrt{\la \mu}} \La Q_{\ula} 
	\mid A_\la g } \right | \lesssim M^{-1} \la^{-\frac 12} \| A_\la g \|_{L^2} 
\lesssim_L M^{-1} \la^{-\frac 12} \| g \|_{H} \lesssim_L M^{-1}. 
\end{align*}
Thus, for $M > M_0(L)$, the above term is $\ll 1$.  We conclude that
\begin{align*}
 \ang{ \chi_{M \sqrt{\la \mu}} \La Q_{\ula} \mid  \p_r^2 g + \frac{1}{r} \p_r g } \simeq 
 \ang{ \chi_{M \sqrt{\la \mu}} \La Q_{\ula} \mid \frac{f'(Q_\la)}{r^2} g}.
\end{align*}

We now rewrite~\eqref{eq:b'2} as 
\EQ{ \label{eq:b'3} 
b'(t) & \simeq     \ang{  \chi_{M \sqrt{\la \mu}} \La Q_{\ula} \mid  \frac{1}{r^2} \Big( f( Q_\la - Q_\mu + g) - f(Q_\la) + f(Q_\mu)- f'(Q_\la)g\Big)} \\
& \quad - \ang{ \p_r^2 g + \frac{1}{r} \p_r g  - \frac{1}{r^2} \Big( f( Q_\la - Q_\mu + g) - f(Q_\la) + f(Q_\mu)\Big) \mid \A_0(\la) g}.  
} 
We add, subtract and regroup to obtain
\begin{align}  \label{eq:lead3}
b'(t)  &\simeq 
% \ang{ \La Q_{\ula} \mid  \frac{1}{r^2} \Big( f( Q_\la - Q_\mu + g) - f(Q_\la) + f(Q_\mu)- f'(Q_\la)g\Big)} \\ 
 \ang{\chi_{M \sqrt{\la \mu}} \La Q_{\ula} \mid \frac{1}{r^2}  \Big( f(Q_{\la} - Q_\mu) - f(Q_\la) + f(Q_\mu) \Big)} \\ 
  &+  \ang{\chi_{M \sqrt{\la \mu}} \La Q_{\ula} \mid \frac{1}{r^2} \Big( f'(Q_\la - Q_\mu) - f'(Q_\la) \Big) g}  \label{eq:2ndterm} \\ 
& +  \ang{ \chi_{M \sqrt{\la \mu}}\La Q_{\ula} \mid \frac{1}{r^2} \Big( f(Q_\la - Q_\mu + g) - f(Q_\la - Q_\mu) - f'(Q_\la - Q_\mu) g \Big)} \label{eq:3rdterm}  \\
 & - \ang{ \p_r^2 g + \frac{1}{r} \p_r g  - \frac{1}{r^2} \Big( f( Q_\la - Q_\mu + g) - f(Q_\la) + f(Q_\mu)\Big) \mid \A_0(\la) g}  \label{eq:4thterm}.
\end{align}
We now identify the first term above as the leading order contribution.
\begin{claim} \label{c:lead} 
\begin{align}
\ang{ \chi_{M \sqrt{\la \mu}} \La Q_{\ula} \mid \frac{1}{r^2}  \Big( f(Q_{\la} - Q_\mu) - f(Q_\la) + f(Q_\mu) \Big)}  \simeq \frac{8}{\mu}.
\label{eq:mlead}
\end{align}
\end{claim} 
By trigonometric identities
\begin{align} 
f(Q_{\la} - Q_\mu) - f(Q_\la) + f(Q_\mu)  &= \frac{1}{2} ( \sin2 Q_\la( \cos 2Q_\mu - 1) + \sin 2 Q_\mu( 1 -  \cos 2 Q_\la)) \\
& = -   \sin 2 Q_\la \sin^2 Q_\mu + \sin 2 Q_\mu  \sin^2 Q_\la \\
& =  - \sin 2 Q_\la (\La Q_\mu)^2 + \sin 2 Q_\mu (\La Q_\la)^2.     \label{eq:trig1} 
\end{align}
We show that the first term in the above expansion gives a negligible contribution to the $L^2$ pairing on the left side of \eqref{eq:mlead}.  Indeed, if we denote $\sigma := \la / \mu$, then as long as $\sigma \ll 1$ depending on $L$ and $M$,   
\begin{align*}
\left | \ang{
\chi_{M \sqrt{\la \mu}} \La Q_{\ula} \mid 
\frac{\sin 2Q_{\la}}{r^2} (\La Q_{\mu})^2 } \right |
&\lesssim \frac{1}{\la} \int_0^{2M \sqrt{\la \mu}} |\La Q_{\la}|^2 |\La  Q_{\mu}|^2 \frac{dr}{r} \\
&\lesssim 
\frac{1}{\sigma} \int_0^{2M \sqrt{\sigma}} |\La Q_{\sigma}|^2 |\La Q|^2 \frac{dr}{r} \\
&= 
\frac{1}{\sigma} \int_0^{2M \sqrt{\sigma}} \frac{(r/\sigma)^2}{(1 + (r/\sigma)^2)^2} \frac{r^2}{(1+r^2)^2} \frac{dr}{r} \\
&\lesssim 
\sigma \Bigl [ 
\int_0^\sigma \sigma^{-4} r^3 dr + 
\int_\sigma^{2M \sqrt{\sigma}} \frac{r^3}{(\sigma^2 + r^2)^2} dr
\Bigr ] \\
&\lesssim \sigma [|\log \sigma| + \log M] \ll 1. 
\end{align*}
Thus, 
\begin{align}
\ang{\chi_{M \sqrt{\la \mu}} \La Q_{\ula} \mid  \Big( f(Q_{\la} - Q_\mu) - f(Q_\la) + f(Q_\mu) \Big)} \simeq \big\langle \chi_{M \sqrt{\la \mu}} \La Q_{\ula} \mid  \frac{1}{r^2}(\La Q_{\la})^2 \sin 2Q_\mu \big\rangle \label{eq:firstlead}.
\end{align}

We now compute 
\EQ{ \label{eq:lead1} 
\Big\langle \chi_{M \sqrt{\la \mu}} \La Q_{\ula} &\mid  \frac{1}{r^2}(\La Q_{\la})^2 \sin 2Q_\mu \Big\rangle= \frac{1}{ \la} \int_0^\infty \chi_{M \sqrt{\s}} ( \La Q_{\s})^3 \sin 2 Q \frac{ d r }{r}  \\
& = \frac{1}{\la}  \int_0^{\sqrt{\s}} ( \La Q_{\s})^3 \sin 2 Q\frac{ d r }{r}   + \frac{1}{\la}  \int_{\sqrt{\s}}^\infty \chi_{M \sqrt{\sigma}}( \La Q_{\s})^3 \sin 2 Q \frac{ d r }{r} 
}
Since $|\La Q| \lesssim r^{-1}$ for $r$ large and $\sigma \sim \la$, we have  
\begin{align}
 \frac{1}{\la}  \int_{\sqrt{\s}}^\infty  | \La Q_{\s}|^3 \frac{ d r }{r} 
&\lesssim 
\frac{1}{\s}  
 \int_{\sqrt{\s}}^\infty |\La Q_{\s}|^3 \frac{ d r }{r} \\
&\lesssim \frac{1}{\sigma} \int_{1/\sqrt{\sigma}}^\infty |\La Q|^3 \frac{dr}{r} \\
&\lesssim \sigma^{1/2} \ll 1.  
\end{align}
Thus, from \eqref{eq:lead1} it follows that 
\begin{align}
\big\langle \chi_{\sqrt{\la \mu}} \La Q_{\ula} &\mid  \frac{1}{r^2}(\La Q_{\la})^2 \sin 2Q_\mu \big\rangle \simeq 
\frac{1}{\la}  \int_0^{\sqrt{\s}} ( \La Q_{\s})^3 \sin 2 Q \frac{ d r }{r}.  
\end{align}

Since $\s = \la/ \mu  \ll 1$, on the interval $[0, \sqrt{\s}]$ we write 
\EQ{ \label{eq:sin2Q-small} 
	\sin 2 Q = 4 r \frac{ 1 - r^{2}}{ (1+ r^{2})^2} = 4 r + O(r^{3})
}
We compute 
\EQ{
\frac{1}{\la} \int_0^{\sqrt{\s}} ( \La Q_{\s})^3 4r \, \frac{d r}{r} &=  \frac{4\s}{\la} \int_0^{\frac{1}{\sqrt{\s}}}  (\La Q)^3 \, d r \\
& = \frac{4\s}{\la} \int_0^\infty  (\La Q)^3  \, d r -  4 \frac{\s}{\la} \int_{\frac{1}{\sqrt{\s}}}^{\infty}  (\La Q)^3 \, d r \\
& = \frac{ 8 }{\mu} + O( \s)
}
where the integral $\int_0^\infty  (\La Q)^3  \, dr = 2$ is evaluated using substitution. By~\eqref{eq:sin2Q-small}, 
\EQ{\left |
\frac{1}{\la}  \int_0^{\sqrt{\s}}( \La Q_{\s})^3( \sin 2Q - 4r) \, \frac{d r}{r}\right |& \lesssim \frac{1}{\la}  \int_0^{\sqrt{\s}}| \La Q_{\s}|^3 r^{2} \, d r \\
& = \frac{\s^{3}}{\la} \int_0^{\frac{1}{\sqrt{\s}}}|\La Q|^3 r^{2} \, d r  \lesssim \s^{2} \abs{\log \s} \ll 1.   
}
Thus, 
\begin{align}
\big\langle \chi_{\sqrt{\la \mu}} \La Q_{\ula} &\mid  \frac{1}{r^2}(\La Q_{\la})^2 \sin 2Q_\mu \big\rangle \simeq 
\frac{1}{\la}  \int_0^{\sqrt{\s}} ( \La Q_{\s}(r))^3 \sin 2 Q(r) \frac{ d r }{r} \simeq \frac{8}{\mu}  \label{eq:secondlead}. 
\end{align}
Combining \eqref{eq:firstlead} and \eqref{eq:secondlead} we conclude that 
\EQ{
\ang{\chi_{\sqrt{\la \mu}} \La Q_{\ula} \mid  \Big( f(Q_{\la} - Q_\mu) - f(Q_\la) + f(Q_\mu) \Big)} \simeq 
\frac{8}{\mu}
}
as desired. 
\end{proof} 
For what follows, we list the following useful identities: 
\begin{align}
& \La^2 Q = \frac{1}{2} \sin 2 Q = 2 r \frac{ 1-r^{2}}{ (1 + r^{2})^2} \label{eq:La2Q}, \\ 
&\La^3 Q  =2 r \left( \frac{1+ r^{2} - 5r^{4} - r^{6}}{(1+r^{2})^4}\right) \label{eq:La3Q}, \\
& \La_0 \La Q = ( r\p_r +1)(r \p_r Q) = 2 \La Q + r^2 \p_r^2 Q \label{eq:DLa}. 
\end{align}

We now claim that the term~\eqref{eq:2ndterm} in the expansion of $b'(t)$ satisfies
\EQ{ \label{eq:2term} 
 \left | \ang{ \chi_{M \sqrt{\la \mu}} \La Q_{\ula} \mid \frac{1}{r^2} \Big( f'(Q_\la - Q_\mu) - f'(Q_\la) \Big) g} \right | \lesssim (\la / \mu)^{1/2}.
}
First note that the we have 
\EQ{
 f'(Q_\la - Q_\mu) - f'(Q_\la)& =  \sin 2Q_\la \sin 2Q_\mu -2 \cos 2Q_\la \sin^2Q_\mu \\
&= 4 \La^2 Q_\la \La^2 Q_\mu- (\La Q_\mu)^2\cos 2Q_\la \label{trig}
}
By ~\eqref{eq:La2Q} and ~\eqref{eq:La3Q} we have 
\begin{align}
|\La Q| + |\La^2 Q| \lesssim \frac{r}{1+r^2}. 
\end{align}
We first estimate 
\begin{align}
\abs{ \ang{ \chi_{M \sqrt{\la \mu}} \La Q_{\ula} \mid \frac{4}{r^2} \La^2 Q_\la \La^2 Q_\mu g}} &\lesssim \| g \|_{L^\infty}
\frac{1}{\la} \int_0^{2M\sqrt{\la \mu}} 
|\La Q_{\la}| |\La^2 Q_{\la}| |\La^2 Q_{\mu}| \frac{d r}{r} \\
&\lesssim \| g \|_{H}
\frac{1}{\s} \int_0^{2M\sqrt{\s}} 
|\La Q_{\s}| |\La^2 Q_{\s}| |\La^2 Q| \frac{d r}{r} \\
&\lesssim \| g \|_{H} \int_0^{2M/\sqrt{\s}} |\La Q| |\La^2 Q| d r 
\lesssim \s^{1/2}.
\end{align}
where $\s = \la/ \mu$ as before.  We then estimate 
\EQ{
\abs{\ang{ \chi_{M \sqrt{\la \mu}} \La Q_{\ula} \mid \frac{1}{r^2}( (\La Q_\mu)^2\cos 2Q_\la)  g}} &\lesssim \| g\|_{H} \left( \int_0^\infty  (\La Q_{\s})^2 (\La Q)^4 \frac{ d r}{r} \right)^{\frac{1}{2}} \\
& \lesssim \s^{1/2}  \left( \int_0^\infty  (\La Q_{\s})^2 \frac{ d r}{r} \right)^{\frac{1}{2}} \lesssim \s^{1/2}.
}
The previous two bounds along with \eqref{trig} imply \eqref{eq:2term}. 

In summary, we have shown thus far that 
\begin{align} 
b'(t)  -
\frac{8}{\mu}
&\simeq  \ang{ \chi_{M \sqrt{\la \mu}}\La Q_{\ula} \mid \frac{1}{r^2} \Big( f(Q_\la - Q_\mu + g) - f(Q_\la - Q_\mu) - f'(Q_\la - Q_\mu) g \Big)} \label{eq:3rdtermb}  \\
& - \ang{ \p_r^2 g + \frac{1}{r} \p_r g  - \frac{1}{r^2} \Big( f( Q_\la - Q_\mu + g) - f(Q_\la) + f(Q_\mu)\Big) \mid \A_0(\la) g}.  \label{eq:4thtermb}
\end{align}

We now rewrite~\eqref{eq:3rdtermb} as  
\begin{multline}
\Big\langle \chi_{M \sqrt{\la \mu}} \Lambda Q_\uln\lambda \mid
\frac{1}{r^2}\big(f({-}Q_\mu + Q_\lambda + g) - f({-}Q_\mu + Q_\lambda) - f'({-}Q_\mu + Q_\lambda)g\big)\Big\rangle \\ 
=  -  \ang{  \A(\la) g \mid \frac{1}{r^2}\big(f({-}Q_\mu + Q_\lambda + g) - f({-}Q_\mu + Q_\lambda) -g\big)} \\
+\ang{ \A(\la) g \mid \frac{1}{r^2}\big(f({-}Q_\mu + Q_\lambda + g) - f({-}Q_\mu + Q_\lambda) -g\big)} \\
+\ang{ \A(\la)  (Q_\la- Q_\mu) \mid \frac{1}{r^2}\big(f({-}Q_\mu + Q_\lambda + g) - f({-}Q_\mu + Q_\lambda) - f'({-}Q_\mu + Q_\lambda)g\big)}  \\
+ \ang{ \A(\la) Q_\mu \mid \frac{1}{r^2}\big(f({-}Q_\mu + Q_\lambda + g) - f({-}Q_\mu + Q_\lambda) - f'({-}Q_\mu + Q_\lambda)g\big)}  \\
+ \ang{ \chi_{M\sqrt{\la \mu}} \bigl ( \La Q_{\ula} - \A(\la)  Q_\la \bigr )  \mid  \frac{1}{r^2}\big(f({-}Q_\mu + Q_\lambda + g) - f({-}Q_\mu + Q_\lambda) - f'({-}Q_\mu + Q_\lambda)g\big)}
\end{multline}
We remark that we used the fact that $\chi_{M \sqrt{\la \mu}} \A(\la) Q_\la = \A(\la) Q_\la$ (as long as $\la/\mu$ is small) to obtain the previous expression.   
The second and third terms above can be estimated using ~\eqref{eq:A-by-parts-wm} with $g_1 = Q_{\la} - Q_\mu$ and $g_2 = g$: 
\begin{multline}
\bigg|\ang{ \A(\la) g \mid \frac{1}{r^2}\big(f({-}Q_\mu + Q_\lambda + g) - f({-}Q_\mu + Q_\lambda) -g\big)} 
\\ +\ang{ \A(\la)  (Q_\la- Q_\mu) \mid \frac{1}{r^2}\big(f({-}Q_\mu + Q_\lambda + g) - f({-}Q_\mu + Q_\lambda) - f'({-}Q_\mu + Q_\lambda)g\big)} \bigg|
 \lesssim_L c_0 
\end{multline}
which is $\ll 1$ as long as $c_0$ is taken sufficiently small. 
The pointwise bound
\begin{multline}
\abs{ f(Q_\la - Q_\mu + g) - f(Q_\la - Q_\mu) - f'(Q_\la - Q_\mu) g }  \\
= \frac{1}{2}\abs{ \sin(2Q_\la - 2Q_\mu)[ \cos 2 g - 1] + \cos(2Q_\la - 2Q_\mu)[ \sin 2g -2 g]} 
\lesssim \abs{g}^2
\end{multline} 
and~\eqref{eq:Ainfty} imply that the second to last line of the above satisfies
\EQ{
 \left | \ang{ \A(\la) Q_\mu \mid \frac{1}{r^2}\big(f({-}Q_\mu + Q_\lambda + g) - f({-}Q_\mu + Q_\lambda) - f'({-}Q_\mu + Q_\lambda)g\big)} \right | \lesssim \frac{1}{\mu} \| g \|_{H}^2 \lesssim \la \ll 1. 
}
Using \eqref{eq:L-A-wm} we estimate the last line of the expansion of~\eqref{eq:3rdterm} similarly:  
\begin{multline}
\abs{\ang{ \chi_{M \sqrt{\la \mu}}(\La Q_{\ula} - \A(\la)  Q_\la)  \mid  \frac{1}{r^2}\big(f({-}Q_\mu + Q_\lambda + g) - f({-}Q_\mu + Q_\lambda) - f'({-}Q_\mu + Q_\lambda)g\big)}}
\\ \lesssim \|\La Q_{\ula} - \A(\la)  Q_\la \|_{L^\infty} \| g\|^2_H  \lesssim_L c_0
\ll 1. 
\end{multline} 
Thus, we have shown that 
\begin{multline}  \label{eq:3rdterm2} 
\Big\langle \chi_{M \sqrt{\la \mu}}\Lambda Q_\uln\lambda \mid
\frac{1}{r^2}\big(f({-}Q_\mu + Q_\lambda + g) - f({-}Q_\mu + Q_\lambda) - f'({-}Q_\mu + Q_\lambda)g\big)\Big\rangle  \\ \simeq -  \ang{ \A(\la) g \mid \frac{1}{r^2}\big(f({-}Q_\mu + Q_\lambda + g) - f({-}Q_\mu + Q_\lambda) -g\big)},
\end{multline}
which by \eqref{eq:3rdtermb} implies 
\begin{align} 
b'(t)  -
\frac{8}{\mu}
&\simeq -  \ang{ \A(\la) g \mid \frac{1}{r^2}\big(f({-}Q_\mu + Q_\lambda + g) - f({-}Q_\mu + Q_\lambda) -g\big)} \\
& - \ang{ \p_r^2 g + \frac{1}{r} \p_r g  - \frac{1}{r^2} \Big( f( Q_\la - Q_\mu + g) - f(Q_\la) + f(Q_\mu)\Big) \mid \A_0(\la) g}.  \label{eq:4thtermb}
\end{align}
 
We now consider the line~\eqref{eq:4thtermb}.  By adding and subtracting terms and ~\eqref{eq:A-pohozaev-wm} we have
\begin{equation}
\label{eq:mod-dtb-2nd-line}
\begin{aligned}
&{-} \Big\langle \partial_r^2 g + \frac 1r \partial_r g - \frac{1}{r^2}\big(f(Q_\la- Q_\mu + g) - f(Q_\la) +f(Q_\mu)\big) \mid  \A_0(\lambda)g\Big\rangle \\
& = -\ang{ \A_0(\la) g \mid  \partial_r^2 g + \frac 1r \partial_r g - \frac{1}{r^2}g} \\
&\quad + \ang{ \A_0(\la) g \mid \frac{1}{r^2} \Big( f(Q_\la - Q_\mu) - f(Q_\la) + f(Q_\mu) \Big)}\\
& \quad +\ang{ \A_0(\la) g \mid  \frac{1}{r^2} \Big( f(Q_\la - Q_\mu + g) - f(Q_\la-Q_\mu) -  g \Big)} \\
&\geq -\frac{c_0}{\lambda}\|g \|_H^2 + \frac{1}{\lambda}\int_0^{R\lambda}\Big((\partial_r g)^2 + \frac{1}{r^2}g^2\Big)rd r \\
&+ \ang{ \A_0(\la) g \mid \frac{1}{r^2} \Big( f(Q_\la - Q_\mu) - f(Q_\la) + f(Q_\mu) \Big)}\\
&+ \Big\langle \A_0(\lambda)g \mid 
\frac{1}{r^2}\big(f({-}Q_\mu + Q_\lambda + g) + f(-Q_\mu + Q_\la) - g\big)\Big\rangle
\end{aligned}
\end{equation}
where $R$ is defined in the statement of Lemma~\ref{lem:op-A-wm}. 
From ~\eqref{eq:trig1} we have the pointwise estimate  
\EQ{
\abs{ f(Q_\la - Q_\mu) - f(Q_\la) + f(Q_\mu)} \lesssim (\La Q_\la)^2 (\La Q_\mu) + \La Q_\la (\La Q_{\mu})^2.
}
By Lemma \ref{lem:op-A-wm} $\| \A_0(\la) g \|_{L^2} \lesssim \|g \|_{H}$ and $\A_0(\la) g$ is supported on a ball of radius $C R \la$.  Thus, the second to last line above satisfies  
\EQ{
\bigg| \bigg\langle &\A_0(\la) g \mid \frac{1}{r^2} \Big( f(Q_\la - Q_\mu) - f(Q_\la) + f(Q_\mu) \Big) \bigg \rangle \bigg| \\ & \lesssim \|g \|_H\left[ \left( \int_0^{C R \s} r^{-2} (\La Q_\s)^4 (\La Q)^2 \,  \frac{d r}{r} \right)^{\frac{1}{2}} +  \left( \int_0^{C R \s} r^{-2} (\La Q)^4 (\La Q_\s)^2 \,  \frac{d r}{r} \right)^{\frac{1}{2}} \right] \\
&\lesssim \|g \|_{H} \lesssim \la^{1/2} \ll 1.
} 

Thus, 
\EQ{ \label{eq:34terms} 
& -  \ang{ \A(\la) g \mid \frac{1}{r^2}\big(f({-}Q_\mu + Q_\lambda + g) - f({-}Q_\mu + Q_\lambda) -g\big)}  \\
&{-} \Big\langle \partial_r^2 g + \frac 1r \partial_r g - \frac{1}{r^2}\big(f(Q_\la- Q_\mu + g) - f(Q_\la) +f(Q_\mu)\big) \mid  \A_0(\lambda)g\Big\rangle  \\
& \ge  \frac{1}{\lambda}\int_0^{R\lambda}\Big((\partial_r g)^2 + \frac{1}{r^2}g^2\Big)r dr  \\
& \quad +\Big\langle \big(\A_0(\lambda) - \A(\lambda)\big) g \mid \frac{1}{r^2}\big(f({-}Q_\mu + Q_\lambda + g) - f({-}Q_\mu+Q_\lambda)- g\big)\Big\rangle + o(1).
}

The difference $\A_0(\lambda) - \A(\lambda)$ is given by the operator of multiplication by $\frac{1}{2\lambda} \Big(q''\big(\frac{r}{\lambda}\big) + \frac{\lambda}{r}q'\big(\frac{r}{\lambda}\big)\Big)$.
By \eqref{eq:approx-potential-wm} we have  
\EQ{ \label{eq:34terms-error} 
 \Big\langle \big(\A_0(\lambda) - \A(\lambda)\big) g\mid \frac{1}{r^2}\big(f({-}Q_\mu + Q_\lambda + g) - f({-}Q_\mu+Q_\lambda)- k^2g\big)\Big\rangle   \\ 
 =  \frac{1}{\lambda}\int_0^{\infty} \frac{1}{r^2}\big(f'(Q_\lambda)-1\big)g^2 dr  + O(c_0 \la)
}
where $c_0>0$ is as in Lemma~\ref{lem:op-A-wm}. 

The estimates \eqref{eq:4thtermb},~\eqref{eq:34terms} and~\eqref{eq:34terms-error} combine to yield
\EQ{
b'(t)  - \frac{8}{\mu}
\geq \frac{1}{\lambda}\int_0^{R\lambda}\Big((\partial_r g)^2 + \frac{1}{r^2}g^2\Big)r dr  + \frac{1}{\lambda}\int_0^{\infty} \frac{1}{r^2}\big(f'(Q_\lambda)-1\big)g^2 dr + o(1).
}
The orthogonality condition $\ang{ \ZZ_{\ula} \mid g} = 0$ implies 
the localized coercivity estimate, 
\EQ{
\frac{1}{\lambda}\int_0^{R\lambda}\Big((\partial_r g)^2 + \frac{1}{r^2}g^2\Big)r dr  + \frac{1}{\lambda}\int_0^{\infty} \frac{1}{r^2}\big(f'(Q_\lambda)-1\big)g^2 dr   \ge - \frac{c_1}{\la} \| g\|_H^2  
}
(see~\cite[Lemma 5.4, eq. (5.28)]{JJ-AJM} for the proof). The constant $c_1>0$ appearing above can be made small by choosing $R$ sufficiently large.  Since $\|  g \|_{H}^2 \lesssim \la$, we conclude that
\begin{align*}
b'(t) - \frac{8}{\mu} \geq -\frac{\delta}{2} \geq -\frac{\delta}{\mu}
\end{align*} 
as long as $L$ is sufficiently large, $M$ is sufficiently large depending on $L$ and $\la / \mu$ is sufficiently small depending on $L$ and $M$. 
\qed 

From Proposition~\ref{p:modp} and Proposition~\ref{p:modp2} we now show that, roughly, if the modulation parameters are approaching each other in scale, then the solution to \eqref{eq:wmk} is ejected from a small neighborhood of the set of two-bubbles. 

\begin{rem}\label{r:param}
We now fix the parameter $L$ and $M$ used in the definition of $\zeta(t)$ for the remainder of the section.  In particular, we fix $L = L_0$ and $M = M_0$, large enough so that the estimates in Proposition \ref{p:modp2} hold for 
\begin{align*}
\de = \frac{1}{2018},
\end{align*} 
whenever $\bfd(\vec \psi(t)) < \eta_1 = \eta_1(L_0,M_0) < \eta_0$. 
\end{rem}

\begin{prop}
	\label{prop:modulation}
	Let $C > 0$. Then for all $\eps_0>0$ sufficiently small, and for any $\eps>0$ sufficiently small relative to $\e_0$ the following holds.
	Let $\vec\psi(t): [T_0, T_+) \to \HH_0$ be a solution of \eqref{eq:wmk}. 
	Assume that $t_0 \in [T_0, T_+)$ is so that $\bfd(\vec \psi(t_0)) \leq \eps$
	and $\frac{d}{dt} (\zeta(t)/\mu(t))\vert_{t = t_0} \geq 0$.
	Then there exist $t_1$ and $t_2$, $T_0 \leq t_0 \leq t_1 \leq t_2 < T_+$,
	such the follows estimates hold: 
	\begin{align}
	\bfd(\vec \psi(t)) &\geq 2\eps,\qquad \text{for }t \in [t_1, t_2], \label{eq:d-t1-t2}\\
	\bfd(\vec \psi(t)) &\leq \frac 14\eps_0,\qquad \text{for }t \in [t_0, t_1], \label{eq:d-t0-t1}\\
	\bfd(\vec \psi(t_2)) &\geq 2\eps_0, \label{eq:d-at-t2} \\
	\int_{t_1}^{t_2} \|\partial_t\psi(t)\|_{L^2}^2d t &\geq C\int_{t_0}^{t_1}\sqrt{\bfd(\vec \psi(t))}d t
	\label{eq:err-absorb}
	\end{align}
	If we assume that $\frac{d}{dt} (\zeta(t)/\mu(t))\vert_{t = t_0} \le 0$, then analogous statements hold with times  $t_2  \le  t_1 \le t_0$.  
\end{prop} 

\begin{proof} 
	From \eqref{eq:lamud1}, ~\eqref{eq:bound-on-l} and Remark \ref{r:param}, it follows that if $\e_1 > 0$ is sufficiently small and 
	$\zeta(t)/\mu(t) \leq 4\eps_1$, then the estimates in Proposition \ref{p:modp2} hold with $\delta = 1/2018$
	in a neighborhood of $t_0$. In particular, we have 
	\begin{align}\label{eq:comparabile1}
	\frac{1}{4} \frac{\zeta(t)}{\mu(t)} \leq \frac{\la(t)}{\mu(t)} \left |\log \frac{\la(t)}{\mu(t)} \right | \leq \frac{\zeta(t)}{\mu(t)}. 
	\end{align}
	Let $t_2$ be the first time $t_2 \geq t_0$ such that $\zeta(t_2) / \mu(t_2) = 4\eps_1$. If there is no such time, we set  $t_2 = T_+$.  Define $$f(x) = x |\log x|,$$ which is smooth and increasing on $(0,100 \e_1)$ for $\e_1$ sufficiently small and satisfies $\lim_{x \rar 0^+} f(x) = 0$. Then \eqref{eq:comparabile1} becomes 
	\begin{align}\label{eq:comp}
	\frac{1}{4} \frac{\zeta(t)}{\mu(t)} \leq f (\la(t)/\mu(t)) \leq \frac{\zeta(t)}{\mu(t)}. 
	\end{align}
	Then if $t_2 < T_+$ we have $f(\la(t_2)/\mu(t_2)) \geq \e_1$ which by \eqref{eq:lamud1} implies \eqref{eq:d-at-t2} by taking $\e_0$ comparable to $f^{-1}(\e_1)$. 
	By the scaling symmetry of the equation, we can assume that $\mu(t_0) = 1$.
	%We let $M = 1$ in the 
	Let $t_3 \leq t_2$ be the last time such that $\mu(t) \in [\frac 12, 2]$ for all $t \in [t_0, t_3]$. If there is no such final time we set $t_3= t_2$. We will see by a bootstrapping argument that we can always take $t_3 = t_2$ and that $t_2 < T_+$. 
	
	By Remark \ref{r:param} and by taking $\e_1$ small enough, we have by \eqref{eq:b'lb} 
	\begin{equation}
	\label{eq:b'lb2}
	b'(t) \geq 1.
	\end{equation}
	We also obtain from \eqref{eq:kala'}
	\begin{equation}
	\zeta'(t) \geq b(t) - \zeta(t)^\frac 12.
	\end{equation}
	Consider $\xi(t) := b(t) + \zeta(t)^\frac 12$.
	Using the two inequalities above we obtain
	\begin{equation}
	\label{eq:psi'}
	\begin{aligned}
	\xi'(t) \geq 1 + \frac 12 \zeta(t)^{-\frac 12}
	\Big( b(t) - \zeta(t)^\frac 12\Big) 
	= \frac{1}{2} \zeta(t)^{-\frac 12} (b(t) + \zeta(t)^{\frac 12})
	= \frac{1}{2} \zeta(t)^{-\frac 12} \xi(t).
	\end{aligned}
	\end{equation}
	By \eqref{eq:b-bound} and the fact that $\mu(t) \in (1/2,2)$, we conclude that 
	\begin{equation}
	\label{eq:psi-bound}
	\xi(t) \leq 10 \zeta(t)^\frac 12.
	\end{equation}
	
	Let $\xi_1(t) := b(t) + \frac{1}{2}\zeta(t)^\frac 12 = \frac 12 b(t) + \frac 12 \xi(t) =  \xi(t) - \frac{1}{2} \zeta(t)^{\frac 12}$.
	Since $b'(t) \geq 0$, we have
	\begin{equation}
	\label{eq:psi1'}
	\xi_1'(t) \geq \frac 12 \xi'(t) \geq \frac 14 \zeta(t)^{-\frac 12} \xi(t) \geq \frac 14 \zeta(t)^{-\frac 12} \xi_1(t).
	\end{equation}
	
	Since $\mu(t_0) = 1$, we have $0 \leq \frac{d}{dt} (\lambda(t)/\mu(t))\vert_{t = t_0} = \zeta'(t_0) - \zeta(t_0)\mu'(t_0)$, so \eqref{eq:mu'} and~\eqref{eq:bound-on-l} imply that $\zeta'(t_0) \geq -\frac{1}{8}\zeta(t_0)^\frac 12$ as long as $\eps$ is taken small enough. This fact and \eqref{eq:kala'} gives $b(t_0) \geq -\frac{1}{4}\zeta(t_0)^\frac 12$,
	so $\xi_1(t_0) > 0$ and \eqref{eq:psi1'} yields $\xi_1(t) > 0$ for all $t \in [t_0, t_3]$.
	Thus
	\begin{equation}
	\label{eq:psi-lbound}
	\xi(t) \geq \frac{1}{2}\zeta(t)^\frac 12, \qquad \text{for }t \in [t_0, t_3].
	\end{equation}
	This lower bound along with $\xi'(t) \geq \frac{1}{2} \zeta(t)^{-\frac 12} \xi(t)$ imply
	\begin{align}
	\xi'(t) \geq \frac{1}{4}. \label{eq:psi'2}
	\end{align}
	By \eqref{eq:psi-bound} we see that  $\zeta(t)$ and thus $\la(t)$  is far from $0$ on $[t_0, t_3]$.
	
	The bounds \eqref{eq:psi-bound}, \eqref{eq:comp} and \eqref{eq:lamud1} imply that there exists a constant $\alpha_0$
	such that $\xi(t) \geq 40 [f(\al_0 \e)]^{1/2}$ forces $\bfd(\vec \psi(t)) \geq 2\eps$.
	Let $t_1 \in [t_0, t_3]$ be the last time such that $\xi(t_1) = 40  [f(\al_0 \e)]^{1/2}$
	(set $t_1 = t_3$ if no such time exists).
	Then by \eqref{eq:psi-lbound} and~\eqref{eq:comp}  we have	
	\EQ{
		[f(\la(t)/\mu(t))]^{1/2}  \le \zeta(t)^{\frac{1}{2}} \leq 80 [f(\al_0 \e)]^{1/2} \quad\text{for }t\in[t_0, t_1],
	}
	which yields \eqref{eq:d-t0-t1} if $\eps$ is small enough.
	
	We now claim that $\mu(t) \in (1/2, 2)$ for all $t \in [t_0,t_2]$ and that $t_2 < T_+$. Recall that on $[t_0,t_3]$ we have $\xi'(t) > 0$ as well as 
	\begin{align*}
	\zeta(t)^{\frac 12} \leq 2 \xi(t) \leq 20 \zeta(t)^{\frac 12}. \quad 
		\xi'(t) \geq \frac{1}{2} \zeta^{-\frac 12}(t) \xi(t), \quad 
		\zeta(t) \leq 8 \e_1. 
	\end{align*}  
	Thus, by \eqref{eq:mu'}
	\begin{align*}
	&\int_{t_0}^{t_3} |\mu'| dt \lesssim \int_{t_0}^{t_3} 
	\zeta(t)^{\frac 12} dt \lesssim 
	\int_{t_0}^{t_3} \xi(t) dt 
	\lesssim \int_{t_0}^{t_3} \zeta(t)^{\frac 12} \xi'(t) dt \\ 
	&\quad \lesssim \sqrt{\e_1} \int_{t_0}^{t_3} \xi'(t) dt 
	\lesssim \sqrt{\e_1} \xi(t_3) 
	\lesssim \e_1, 
	\end{align*}
	where the implied constant is absolute. Thus, 
	we get $\mu(t_3) \in [2/3, 3/2]$ if $\eps_1$ is small enough, which implies that $t_3 = t_2$.
	Now suppose that there is no $t_2 \geq t_0$ such that $\zeta(t_2)/\mu(t_2) = \eps_1$.
	Then, since $\zeta(t)$ (and hence $\la(t)$) is far from $0$, by \cite[Corollary A.4]{JJ-AJM}
	the solution is global and \eqref{eq:psi'2} implies that $\xi(t)$ is eventually $O(1)$.  
	Thus $\zeta(t)$ is eventually $O(1)$ which contradicts our definition of $t_2$.
	This implies that there exists $t_2 < T_+$ such that $\zeta(t_2)/\mu(t_2) = \eps_1$,
	which implies \eqref{eq:d-at-t2} by choosing $\eps_0$ comparable to $f^{-1}(\eps_1)$.
	
	By \eqref{eq:kala'} and~\eqref{eq:b-bound} we have $|\zeta'(t)| \lesssim |\zeta(t)|$.  Thus, 
	there exists an absolute constant $\al_1 > 0$
	such that $\zeta(t) \geq \frac 14 \eps_1$ for $t \in [t_2 - \al_1, t_2]$.  Since $\zeta(t) \lesssim f(\al_0 \e)$ on $[t_0,t_1]$, we must have $t_2 - t_1 \geq \al_1$ if $f(\al_0\e) \ll \e_1$.  
	Then \eqref{eq:b'lb2} yields
	\begin{equation}
	b(t) \geq b(t_1) + \al_1 \geq b(t_0) + \al_1, \qquad \text{for }t \in [t_1, t_2].
	\end{equation}
	Thus, if $\eps$ is small enough, we get
	\begin{align}
	b(t) \geq \al_1/2, \quad t \in [t_1, t_2]. \label{b_lower}
	\end{align}
	By Proposition \ref{p:modp}, the Cauchy-Schwarz inequality and the definition of $b(t)$ we have 
	\begin{align*}
	b(t) \lesssim |\log \la|^{1/2} \| \dot g \|_{L^2}.
	\end{align*}
	Since $\la(t) \leq \zeta(t) \leq 8 \e_1$ on $[t_0,t_2]$, we conclude that  
	there exists an absolute constant $\al_2 > 0$ such that on $[t_0,t_2]$
	\begin{align}
	|b(t)| \leq \al_2 | \log \e_1|^{1/2} \| \dot g \|_{L^2}. \label{b_upper}
	\end{align}
	Integrating, from $t_1$ to $t_2$ the lower bound \eqref{b_lower} and using \eqref{b_upper} we obtain 
	\begin{align*}
	\frac{\al_1^3}{4}
	\leq \int_{t_1}^{t_2} |b(t)|^2 dt \leq \al_2^2 |\log \eps_1| \int_{t_1}^{t_2} \| \dot g(t) \|_{L^2}^2 dt,
	\end{align*}
	which implies
	\begin{align}
	\frac{\al_1^3}{4 \al_2^2  |\log \eps_1|} \leq \int_{t_1}^{t_2} \| \p_t \psi (t) \|^2_{L^2} dt. \label{eq:err-lbound} 
	\end{align}
	Recall that on $[t_0,t_1]$, we have $\xi'(t) \geq 1/4$ and $|\xi(t)| + \zeta(t)^{1/2} \lesssim \sqrt{\e \al_0 |\log \al_0 \e|}$ where $\al_0$ is an absolute constant. Thus, 
	\begin{align}
	\int_{t_0}^{t_1} \sqrt{\bfd (\vec \psi(t))} dt 
	\lesssim \int_{t_0}^{t_1} \sqrt{\zeta(t)} dt 
	\lesssim \int_{t_0}^{t_1} \sqrt{\zeta(t)} \xi'(t) dt \lesssim 
	\sqrt{\e |\log \e|} \int_{t_0}^{t_1} \xi'(t) dt 
	\lesssim \eps |\log \eps|
	\end{align}
	where the implied constant is absolute.
	This estimate and \eqref{eq:err-lbound} imply \eqref{eq:err-absorb} after choosing $\e$ sufficiently small.
\end{proof}

\section{Dynamics of Non-Scattering Threshold Solutions} \label{s:dynamics}

In this section we prove the main result, Theorem \ref{t:main}. We will obtain it as a consequence of the following proposition. 

\begin{prop}\label{p:psi_t} 
Let $\psi(t) :(T_-, T_+) \to \HH_0$ be a corotational wave map with $\E(\vec \psi) = 2 \E(\vec Q)$ which does not scatter in forward time. Then
\EQ{ 
\lim_{t \to T_+} \bfd(\vec \psi(t)) = 0.\label{eq:psi_t}
}
\end{prop}

As a first step, we state a direct consequence of Theorem~\ref{t:cjk}.
\begin{prop}\label{p:cjk}
Let $\vec \psi(t) : (T_-, T_+) \to \HH_0$ be a corotational wave map with $\E(\vec\psi) = 2 \E( \vec Q)$
which does not scatter in forward time. Then
\EQ{ \label{eq:seqbub} 
\liminf_{t\to T_+} \bfd(\vec \psi(t)) = 0.
}
\end{prop}

For the remainder of this section, we will always denote by $\vec \psi(t)$ a solution to~\eqref{eq:wmk}, $\vec\psi(t):(T_-, T_+) \to \HH_0$,
such that $\E(\vec \psi) = 2 \E(\vec Q)$
and $\vec\psi(t)$ does not scatter in forward time.
  A rough sketch of our strategy to prove Proposition \ref{p:psi_t} is as follows.  By our preliminary step Proposition \ref{p:cjk}, we know that $\bfd(\vec \psi(t))$ tends to 0 along a sequence of times.  If Proposition \ref{p:psi_t} were false, then we split the maximal time interval of existence into a collection of \emph{bad} intervals where $\vec \psi(t)$ is close to the set of two-bubbles, and \emph{good} intervals where $\vec \psi(t)$ is far from them.  On the union of good intervals which we denote by $I$, we use Lemma \ref{l:d-size} and Lemma \ref{l:1profile} to show that the $\vec \psi(t)$ has the following \emph{compactness property}: there exists a continuous function $\nu(t) : I \rar (0,\infty)$ such that the trajectory 
\begin{align*}
\cl K = \{ \vec \psi(t)_{1/\nu(t)} \mid t \in I \}
\end{align*} 
is pre-compact in $\cl H_0$. Solutions with the compactness property do not radiate energy, and thus we expect that such solutions are given by rescalings of stationary solutions (harmonic maps).  If this intuition is correct, we arrive at a contradiction since the only degree-0 harmonic map is the constant map which has energy equal to $0 \neq 8 \pi$.  
 
To prove that a solution with the compactness property on the union of good intervals is stationary, we will use the virial identity.  Integrating \eqref{eq:vir} from $t =\tau_1$ to $t = \tau_2$ yields 
\EQ{ \label{eq:virtau}
\int_{\tau_1}^{\tau_2} \| \p_t \psi (t) \|_{L^2}^2 \,  d t  &\le  \abs{ \ang{ \p_t\psi \mid \chi_R r \p_r \psi}(\tau_1)}  + \abs{ \ang{ \p_t\psi \mid   \chi_R r \p_r \psi}(\tau_2)} \\
 &\quad +  \int_{\tau_1}^{\tau_2} \abs{\Om_{R}( \vec \psi(t))} \, d t
 }
 where the error $\Om_{R}(\vec \psi(t))$ is given by ~\eqref{eq:OmRdef}.
By Lemma~\ref{l:error-estim}, we obtain 
 \EQ{ \label{eq:virR} 
 \begin{aligned}
 \int_{\tau_1}^{\tau_2} \| \p_t \psi (t) \|_{L^2}^2 \,  d t  &\leq C_0\Big(  R \sqrt{\bfd(\vec\psi(\tau_1))} +  R \sqrt{\bfd(\vec\psi(\tau_2))}\Big) \\ &+  \int_{\tau_1}^{\tau_2} \abs{\Om_{R}( \vec \psi(t))} \, d t.
 \end{aligned}
 }
We will then show that by choosing the parameters $R, \tau_1, \tau_2$ appropriately and using Proposition \ref{p:modp}, we can absorb the error term involving $\Om_{R}(\vec \psi(t))$ from the right hand side into the left hand side.  The resulting averaged smallness of $\| \p_t \psi(t) \|_{L^2}$ and the compactness property allow us to conclude that $\vec \psi(t) = \vec 0$, our desired contradiction. 
  
\subsection{Splitting time into the good and the bad}
Before we are able to split the time interval of existence into good and bad intervals, we establish the following initial splitting of the time axis. 
\begin{lem}\label{cl:pre-split}
Suppose that~\eqref{eq:psi_t} fails. Then for any $\eps_0 > 0$ sufficiently small there exist
sequences $p_n$, $q_n$ such that
\EQ{
T_- < p_0 < q_0 < p_1 < q_1 < \dots < p_{n-1} < q_{n-1} < p_n < q_n < \dots 
}
with the property that for all $n \in \N$: 
\begin{gather}
\forall t \in [p_n, q_n]: \bfd(\vec\psi(t)) \leq \eps_0, \label{eq:d-small-all-pq} \\
\forall t \in [q_n, p_{n+1}]: \bfd(\vec\psi(t)) \geq \frac 12\eps_0, \label{eq:d-large-all-pq} \\
\lim_{n \to \infty}p_n = \lim_{n\to\infty}q_n = T_+. \label{eq:pq-lim}
\end{gather}
\end{lem}
\begin{proof}
Suppose that \eqref{eq:psi_t} fails. Let $\eps_0 > 0$ so that
\EQ{\label{eq:eps-limsup}
0 < \eps_0 < \min(\limsup_{t\to T_+}\bfd(\vec\psi(t)), \eta_1)
}
Then there exists $T_0 \in (T_-, T_+)$ such that $\bfd(\vec\psi(T_0)) > \eps_0$.
Define
\EQ{
p_0 := \sup\Big\{t: \bfd(\vec\psi(\tau)) \geq \frac 12 \eps_0, \forall \tau\in[T_0, t]\Big\}.
}
By Proposition \ref{p:cjk} we have $p_0 < T_+$ and $\bfd(\vec\psi(p_0)) = \frac 12 \eps_0$.
For $n \geq 1$ we define inductively:
\begin{align}
q_{n-1} := \sup\Big\{t: \bfd(\vec\psi(\tau)) \leq \eps_0, \forall \tau\in[p_{n-1}, t]\Big\}, \\
p_n := \sup\Big\{t: \bfd(\vec\psi(\tau)) \geq \frac 12 \eps_0, \forall \tau\in[q_{n-1}, t]\Big\}.
\end{align}
By \eqref{eq:eps-limsup}, Proposition~\ref{p:cjk} and induction
we have for all $n \in \N$
\begin{gather}
p_{n-1} < q_{n-1} < T_+, \\
q_{n-1} < p_n < T_+, \\
\bfd(\vec\psi(p_n)) = \frac 12\eps_0, \label{eq:d-val-at-p} \\
\bfd(\vec\psi(q_n)) = \eps_0 \label{eq:d-val-at-q}.
\end{gather}
The estimates \eqref{eq:d-large-all-pq} and \eqref{eq:d-small-all-pq} are immediate consequences of our choice
of $p_n$ and $q_n$. Suppose now that \eqref{eq:pq-lim} does not hold. Since $p_n$ and $q_n$ are increasing sequences, we have 
\EQ{
\lim_{n \to \infty}p_n = \lim_{n\to\infty}q_n = T_1 < T_+.
}
By continuity of the flow, $\bfd(\vec\psi(t))$ has a limit as $t \to T_1$ which contradicts \eqref{eq:d-val-at-p} and \eqref{eq:d-val-at-q}.
\end{proof}
\begin{lem}\label{cl:split}
Let $\eps > 0$. There exist $\zeta_0, \eps' > 0$ so that the following holds.
Assume that $\bfd(\vec\psi(t)) < \eta_1$, with $\eta_1$ as in the remarks immediately preceding Proposition \ref{p:psi_t}. Let $\lambda(t)$, $\mu(t)$ be the modulation parameters
given by Lemma~\ref{l:modeq}. Let $\zeta(t)$ be defined as in~\eqref{eq:zetadef}. Then
\begin{align}
\frac{\zeta(t)}{\mu(t)} \geq \zeta_0 &\Rightarrow \bfd(\vec\psi(t)) > \eps', \label{eq:d-geq-epsp} \\
\frac{\zeta(t)}{\mu(t)} \leq \zeta_0 &\Rightarrow \bfd(\vec\psi(t)) < \eps \label{eq:d-leq-eps}.
\end{align}
\end{lem}

\begin{proof}
By Lemma~\ref{l:modeq}, $$\bfd(\vec\psi(t)) \leq (C^2+1)\frac{\la(t)}{\mu(t)} \le 2(C^2+1)\frac{\zeta(t)}{\mu(t)},$$
so we get \eqref{eq:d-leq-eps} with any $\zeta_0 < \eps/2(C^2+1)$.

To prove \eqref{eq:d-geq-epsp}, we first note that by \eqref{eq:bound-on-l}, we have 
\begin{align*}
\frac{\la(t)}{\mu(t)} \left | 
\log \frac{\la(t)}{\mu(t)}
\right | \geq \zeta_0/3. 
\end{align*}
Since the function $f(x) = x |\log x|$ is increasing on $(0,C \eta_1)$ for $\eta_1$ small, we conclude that 
\begin{align*}
\frac{\la(t)}{\mu(t)} \geq f^{-1}(\zeta_0/3). 
\end{align*}
Thus, by \eqref{eq:gdotgd} we obtain \eqref{eq:d-geq-epsp} for any $\e' < \frac{1}{C} f^{-1}(\zeta_0/3)$. 
\end{proof}

We now split the time axis into a collection of good and bad intervals. 

\begin{lem}\label{l:time_split}
Suppose that the conclusion ~\eqref{eq:psi_t} fails. Let $\eps_0>0$ be small enough so that the conclusions of Claim~\ref{cl:pre-split}
and Proposition~\ref{prop:modulation} hold, and let $C_0$ denote the constant from Lemma \ref{l:error-estim}.  Then there exist $\eps, \eps' > 0$ with $\eps' < \eps$  and $\eps< \frac{1}{10} \eps_0$ as in Proposition~\ref{prop:modulation}, 
and sequences of times $(a_m), (b_m), (c_m)$ such that 
\EQ{
T_- < a_1 < c_1 < b_1 < \dots < a_m < c_m < b_m < a_{m+1} < \dots 
}
and the following holds for all $m = 2, 3, 4, \ldots$:
\begin{gather}
\forall t \in [b_{m}, a_{m+1}]: \bfd(\vec\psi(t)) \geq \eps', \label{eq:d-large-all} \\
\exists t \in [b_{m}, a_{m+1}]: \bfd(\vec\psi(t)) \geq 2\eps, \label{eq:d-large-some} \\
\bfd(\vec\psi(a_m)) = \bfd(\vec\psi(b_{m})) = \eps, \label{eq:d-small-some} \\
C_0 \int_{a_{m+1}}^{c_{m+1}}\sqrt{\bfd(\vec\psi(t))}\,d t \leq \frac{1}{10}\int_{b_{m}}^{a_{m+1}}\|\partial_t\psi(t)\|_{L^2}^2\,d t \label{eq:err-mod-contr-1} \\
C_0 \int_{c_{m}}^{b_{m}}\sqrt{\bfd(\vec\psi(t))}\,d t \leq \frac{1}{10}\int_{b_{m}}^{a_{m+1}}\|\partial_t\psi(t)\|_{L^2}^2\,d t \label{eq:err-mod-contr-2}
\end{gather}
and
\EQ{
\liminf_{m \to\infty} \bfd(\vec\psi(c_m)) = 0. \label{eq:d-conv-0}
}
\end{lem}
\begin{proof}
Let $\eps, \eps_0 > 0$ be small, $\e < \frac{1}{10} \e_0$, so that Claim~\ref{cl:pre-split}
and Proposition~\ref{prop:modulation} hold
with the constant $C = 10 C_0$ in Proposition~\ref{prop:modulation}.
Let $\zeta_0$ and $\eps'$ be as in Lemma~\ref{cl:split}. We first construct the sequence of times $(c_m)$. By Proposition  \ref{p:cjk} and our initial splitting, there exists a sequence $0 \leq n_1 < n_2 < ...$ of indices such that
\EQ{\label{eq:inf-leq-eps}
\inf_{t\in[p_{n_m}, q_{n_m}]} \bfd(\vec\psi(t)) \leq \eps'.
}
Since $\bfd(\vec \psi(t)) < \e_0$ on $[p_n,q_n]$, and $\e_0$ is small, the modulation parameters $\lambda(t)$, $\mu(t)$ and $\zeta(t)$ 
are well defined on $[p_{n_m}, q_{n_m}]$.
Let $c_m \in [p_{n_m}, q_{n_m}]$ be such that
\EQ{
\zeta(c_m)/\mu(c_m) = \inf_{t \in [p_{n_m}, q_{n_m}]}\zeta(t)/\mu(t).
}
By Lemma~\ref{cl:split} and \eqref{eq:inf-leq-eps} we conclude that $\zeta(c_m) / \mu(c_m) < \zeta_0$.
By Lemma~\ref{cl:split} it also follows
that $\bfd(\vec\psi(c_m)) < \eps < \frac{1}{10}\eps_0$.
Hence $c_m \in (p_{n_m}, q_{n_m})$ and
\EQ{
\frac{d}{dt} \Big|_{t = c_m}\Big(\frac{\zeta(t)}{\mu(t)}\Big) = 0.
}

By Proposition~\ref{prop:modulation} with $t_0$ = $c_m$, there exists $t_1,t_2$ with $t_2 < t_1 < t_0 = c_m$ such that the following holds:
	\begin{align}
	\bfd(\vec \psi(t)) &\geq 2\eps,\qquad \text{for }t \in [t_2, t_1], \label{eq:d-t1-t2_sp}\\
	\bfd(\vec \psi(t)) &\leq \frac 14\eps_0,\qquad \text{for }t \in [t_1, t_0], \label{eq:d-t0-t1_sp}\\
	\bfd(\vec \psi(t_2)) &\geq 2\eps_0, \label{eq:d-at-t2_sp} \\
	\int_{t_2}^{t_1} \|\partial_t\psi(t)\|_{L^2}^2d t &\geq 10 C_0\int_{t_1}^{t_0}\sqrt{\bfd(\vec \psi(t))}d t
	\label{eq:err-absorb_sp}
	\end{align}
We denote $\al_m := t_2$.  Since $\bfd(\vec \psi(t)) \geq \e_0/2$ on $[q_{n_m-1}, p_{n_m}]$, \eqref{eq:d-t0-t1_sp} implies that $t_1 \in (p_{n_m}, c_m]$. Define 
\EQ{
a_m := \sup\{t \geq t_1: \bfd(\vec\psi(t)) \geq \eps,\ \forall \tau\in [t_1, t]\}.
}
Note that the supremum is well defined since \eqref{eq:d-t1-t2_sp} implies that $\bfd(\vec\psi(t_1)) > \eps$.
Since $\bfd(\vec\psi(c_m)) < \eps$, we have $a_m \in (p_{n_m}, c_m)$
and $\bfd(\vec\psi(a_m)) = \eps$.

The bound \eqref{eq:d-t1-t2_sp} implies that
$\bfd(\vec\psi(t)) \geq \eps$ for $t \in [\al_m, t_1]$.
Since $\bfd(\vec\psi(t)) \geq \eps$ for $t \in [t_1, a_m]$ by our definition of $a_m$, we conclude that
\EQ{\label{eq:d-sigm-am}
\bfd(\vec\psi(t)) \geq \eps,\ \forall t \in [\al_m, a_m].
}
The bounds \eqref{eq:d-at-t2_sp} and \eqref{eq:d-small-all-pq} imply that $t_2 < p_{n_m}$. 
Thus,
by \eqref{eq:d-sigm-am} we have 
\EQ{\label{eq:d-pnm-am}
\bfd(\vec\psi(t)) \geq \eps,\ \forall t \in [p_{n_m}, a_m].
}
Moreover, since $t_1 > p_{n_m}$ we have 
\begin{align}\label{eq:al_m}
\bfd(\vec \psi(t)) \geq 2\e, \quad \forall t \in [\al_m, p_{n_m}]. 
\end{align} 
Finally, \eqref{eq:err-absorb_sp} implies that
\EQ{\label{eq:err-absorb-sig}
\int_{\al_m}^{a_m} \|\partial_t\psi(t)\|_{L^2}^2d t \geq 10 C_0\int_{a_m}^{c_m}\sqrt{\bfd(\vec \psi(t))}d t.
}

We now use Proposition~\ref{prop:modulation} with $t_0$ = $c_m$ in the forward time direction and
obtain times $t_1, t_2$ (different from the previous) with $c_m < t_1 < t_2$.  Arguing as before, we conclude that $t_1 \in (c_m, q_{n_m})$ and define 
\EQ{
b_m := \inf\{t \leq t_1: \bfd(\vec\psi(t)) \geq \eps,\ \forall \tau\in[t, t_1]\}.
}
As in the construction of $a_m$, we conclude that $b_m \in (c_m, q_{n_m})$ and $\bfd(\vec\psi(b_m)) = \eps$. We denote $\beta_m := t_2$. By nearly the same arguments used to establish \eqref{eq:d-sigm-am} and \eqref{eq:d-pnm-am}, we conclude that 
\begin{gather}
\label{eq:d-bm-taum}
\bfd(\vec\psi(t)) \geq \eps,\ \forall t \in [b_m, \beta_m], \\
\label{eq:d-bm-qnm}
\bfd(\vec\psi(t)) \geq \eps,\ \forall t \in [b_m, q_{n_m}], \\
\label{eq:err-absorb-tau}
\int_{b_m}^{\beta_m} \|\partial_t\psi(t)\|_{L^2}^2d t \geq 10 C_0\int_{c_m}^{b_m}\sqrt{\bfd(\vec \psi(t))}d t.
\end{gather}

We now prove \eqref{eq:err-mod-contr-1}. The argument for \eqref{eq:err-mod-contr-2} is completely analogous and will be omitted. 
 By \eqref{eq:err-absorb-sig}, it suffices to prove that $b_{m-1} < \al_m$. If not, then by \eqref{eq:al_m}, it follows that $\e = \bfd(\psi(b_{m-1})) \geq 2 \e$, a contradiction.

To prove \eqref{eq:d-large-all}, let $t \in \R$ such that $\bfd(\vec\psi(t)) < \eps'$.
By our initial splitting, $t \in [p_{n_m}, q_{n_m}]$ for some $m$. Then
\eqref{eq:d-pnm-am} and \eqref{eq:d-bm-qnm} imply that $t\in[a_m, b_m]$.

Finally, we prove the convergence \eqref{eq:d-conv-0}.  By Proposition \ref{p:cjk} and \eqref{eq:d-large-all}, it follows that 
\begin{align*}
\liminf_{m \rar \infty} \frac{\zeta(c_m)}{\mu(c_m)} = 0. 
\end{align*}
Since the function $f(x) = x |\log x|$ is increasing for small $x$ and $\zeta(t)/\mu(t) \sim f(\la(t)/\mu(t))$, it follows that $\liminf_m (\la(c_m)/\mu(c_m)) = 0$. The claim \eqref{eq:d-conv-0} then follows from \eqref{eq:lamud1}. 
\end{proof}

\begin{rem}\label{rem:eps-small}
It follows from the proof that $\eps$ can be taken as small as we wish.
\end{rem}

\subsection{Compactness on good intervals}

For the remainder of the proof of Proposition~\ref{p:psi_t},
we fix $\eps, \eps' > 0$ and the partition of the time axis given by Lemma 
\ref{l:time_split}. The intervals $[b_{m-1},a_m]$ on which $\bfd(\vec \psi(t)) \geq \e'$ are what we referred to earlier as the good intervals and are denoted by
\EQ{
I_m := [b_{m-1}, a_m], \qquad I := \bigcup_{m\geq 1} I_m.
}
We then have 
\begin{equation}
\forall t\in I: \bfd(\vec\psi(t)) \geq \eps'. \label{eq:d-large-I}
\end{equation}

We now show that $\vec\psi(t)$ has the following compactness property on $I$.

\begin{lem} \label{l:Icompact} 
There exists a function $\nu \in C(I;(0,\infty))$ such that the modulated trajectory 
\EQ{
\calK:=  \{  \vec \psi(t)_{1/\nu(t)} \mid t \in I \} \subset \HH_0 
}
is pre-compact in $\HH_0$. 
\end{lem} 

\begin{proof}%[Proof of Lemma~\ref{l:Icompact}]
We will first show pre-compactness along an arbitrary sequence of times. In particular, we claim that if $t_n \in I$ is a sequence of times, then there exists a subsequence, which we continue to denote by $t_n$, and a sequence of scales $\nu_n \in (0,\infty)$ so that $\vec \psi(t_n)_{1/\nu_n}$ converges strongly in $\HH_0$. Indeed, by Lemma~\ref{l:d-size} and \eqref{eq:d-large-I}, we conclude that
\EQ{ \label{eq:IepsH} 
\| \vec \psi(t) \|_{\HH_0} \le C(\eps') \quad \forall t \in I.
}
The claim now follows immediately from Lemma~\ref{l:1profile}.  

We now transfer the above sequential pre-compactness to full pre-compactness using continuity of the flow.
For $t \in I$, we define $\nu(t)$ to be the unique positive number so that 
\begin{equation}
\begin{gathered}
\int_0^\infty  \mathrm e^{-r}\left(( \p_t \psi_{1/\nu(t)}(t, r))^2+( \p_r \psi_{1/\nu(t)}(t, r))^2 +   \frac{(\psi_{1/\nu(t)}(t, r))^2}{r^2} \right) rdr \\ = \frac 12 \| \vec \psi(t) \|_{\HH_0}^2
\end{gathered}
\end{equation}
By the change of variables $\rho = \nu(t) r$, we see that $\nu(t)$ is well defined.  Moreover, since the flow $t \mapsto \vec \psi(t)$ is continuous in $\HH_0$, the function $\nu(t)$ is continuous.

Assume, towards a contradiction, that $\{ \vec \psi(t)_{1/\nu(t)} \mid t \in I \}$ is not pre-compact in $\HH_0$.
Then there exists a sequence of times $t_n \in I$ such that $\vec \psi(t_n)_{1/\nu(t_n)}$ has no convergent subsequence.
By our sequential pre-compactness claim, there exist a subsequence (still denoted $t_n$)
and scales $\nu_n$ such that $\vec \psi(t_n)_{1/\nu_n}$ converges in $\HH_0$
to some $\vec\fy = (\fy_0, \fy_1) \neq (0,0)$ (see Lemma \ref{l:1profile}).

We claim that there exist $C > 0$ such that  $1/C \leq \nu_n / \nu(t_n) \leq C$. Suppose not. Then after passing to a subsequence, either $\nu_n / \nu(t_n) \rar \infty$ or $\nu_n / \nu(t_n) \rar 0$.  By a change of variables we have
\begin{align*}
\int_0^\infty &  \mathrm e^{-r}\left(( \p_t \psi_{1/\nu(t_n)}(t_n, r))^2+( \p_r \psi_{1/\nu(t_n)}(t_n, r))^2 +   \frac{(\psi_{1/\nu(t_n)}(t_n, r))^2}{r^2} \right) rdr \\&=
\int_0^\infty  \mathrm e^{- \nu_n r / \nu(t_n)}\left(( \p_t \psi_{1/\nu_n}(t_n, r))^2+( \p_r \psi_{1/\nu_n}(t_n, r))^2 +   \frac{(\psi_{1/\nu_n}(t_n, r))^2}{r^2} \right) rdr 
\end{align*}  
Since $\|\vec\psi(t_n)\|_{\HH_0}$ converges to $\|\vec\fy\|_{\HH_0}$, the choice of $\nu(t)$ implies the left hand side above converges to $\frac{1}{2} \| \vec \varphi \|_{\HH_0}^2 > 0$.  Since $\vec \psi(t_n)_{1/\nu_n} \rar \vec \varphi$ in $\HH_0$, if either $\nu_n / \nu(t_n) \rar \infty$ or $\nu_n / \nu(t_n) \rar 0$ then the right hand side converges to either 0 or $\| \vec \varphi \|_{\HH_0}^2$, a contradiction. This proves the claim. 

Since $1/C \leq \nu_n / \nu(t_n) \leq C$, the sequence $\nu_n / \nu(t_n)$ has a sub sequential limit in $(0,\infty)$.
Thus $\vec \psi(t_n)_{1/\nu(t_n)}$ has a convergent subsequence in $\HH_0$.  This contradicts our initial assumption that  $\vec \psi(t_n)_{1/\nu(t_n)}$ has no convergent subsequence and finishes the proof. 
\end{proof}

For $m \in \N$, we denote the length of a good interval by 
\EQ{
\nu_m := |I_m| = a_m - b_{m-1}.
}
\begin{lem}\label{l:interval-bd}
There exists $C_2 > 0$ such that for all $m \geq 1$ and all $t \in I_m$ we have 
\EQ{
\frac{1}{C_2} \nu(t) \leq \nu_m \leq C_2 \nu(t).
}
\end{lem}

\begin{proof}[Proof of Lemma~\ref{l:interval-bd}]
The proof is by contradiction.   Suppose that there exists a sequence of integers $m_\ell$ and times $t_\ell \in I_{m_\ell}$ such that
\EQ{ \label{eq:int-bd-part-1}
\lim_{\ell\to\infty}\frac{\nu_{m_\ell}}{\nu(t_\ell)} = 0.
}
Let $\vec\psi_\ell(s)$ be the solution of \eqref{eq:wmk} with initial data
$
\vec\psi_\ell(0) = \vec\psi(t_\ell)_{1/\nu(t_\ell)}.
$
By Lemma \ref{l:Icompact} and after extraction of a subsequence, there exists $\vec \varphi_0 \neq (0,0)$ such that $\vec\psi_\ell(0) \rar \vec\fy_0$ in $\HH_0$. 
Let $\vec\fy(s)$ be the solution of \eqref{eq:wmk} with initial data
$\vec\fy(0) = \vec\fy_0$ which is defined on some interval $[-s_0,s_0]$. By the well-posedness theory for \eqref{eq:wmk}, the flow $\vec\psi_\ell(s)$ exists for $s \in [{-}s_0, s_0]$ for all sufficiently large $\ell$, and $$\lim_{\ell \rar \infty} \| \vec\psi_\ell(s) - \vec\fy(s) \|_{L^\infty_t([-s_0,s_0]; \HH_0)} = 0.$$

Let $t_\ell' \in I_{m_\ell}$ be any sequence of times.  We define $s_\ell = \frac{t_\ell' - t_\ell}{\nu(t_\ell)}$.
By \eqref{eq:int-bd-part-1} we have that $\lim_{\ell} s_\ell = 0$.
Thus, $s_\ell \in [{-}s_0, s_0]$ for all $\ell$ sufficiently large, and we conclude that 
\EQ{
\lim_{\ell\to\infty}\|\vec\psi_\ell(s_\ell) - \vec\fy(s_\ell)\|_{\HH_0} = 0.
}
By continuity of the flow $\lim_{\ell\to \infty}\|\vec\fy(s_\ell) - \vec\fy_0\|_{\HH_0} = 0$,
which by the triangle inequality implies 
\EQ{
\lim_{\ell\to\infty}\|\vec\psi_\ell(s_\ell) - \vec\fy_0\|_{\HH_0} = 0.
}
In particular, $\lim_{\ell\to\infty} \bfd(\vec\psi_\ell(s_\ell)) = \bfd(\vec\fy_0)$.
By the time translation and scaling symmetry of \eqref{eq:wmk}, we have $\vec\psi_\ell(s_\ell) = \vec\psi(t_\ell')_{1/\nu(t_\ell)}$.   Thus, $\bfd(\vec\psi(t_\ell')) = \bfd(\vec\psi_\ell(s_\ell))$
and we obtain
\EQ{
\lim_{\ell\to\infty} \bfd(\vec\psi(t_\ell')) = \bfd(\vec\fy_0),
}
for any sequence $t_\ell' \in I_{m_\ell}$. If we choose $t_{\ell}' = a_{m_\ell}$ then we conclude that $\bfd(\vec \varphi_0) = \e$.  But by Lemma \ref{l:time_split}, there exists a sequence of times $t_{\ell}'$ such that $\bfd(\vec \psi(t'_\ell)) \geq 2\e$ so then $\bfd(\vec \varphi_0) \geq 2\e$.  This is a contradiction, and we obtain the lower bound of the lemma. 

Suppose now that there exist a sequence of integers $m_\ell$ and times $t_\ell \in I_{m_\ell}$ such that
\EQ{
\lim_{\ell\to\infty}\frac{\nu(t_\ell)}{\nu_{m_\ell}} = 0.
}
After extracting a subsequence, either $t_\ell \leq \frac{a_{m_\ell} + b_{m_\ell-1}}{2}$ for all $\ell$ or $t_\ell \geq \frac{a_{m_\ell} + b_{m_\ell-1}}{2}$ for all $\ell$.  In the former case we conclude that 
\EQ{\label{eq:int-bd-part-2}
\lim_{\ell\to\infty}\frac{\nu(t_\ell)}{a_{m_\ell} - t_\ell} = 0.
}
We will consider this situation only; the other case is treated similarly. 

As before we denote by $\vec\psi_\ell(s)$ the solution of \eqref{eq:wmk} with initial data
$\vec\psi_\ell(0) = \vec\psi(t_\ell)_{1/\nu(t_\ell)}$,
and assume (after extraction if necessary) $\vec\psi_\ell(0) \to \vec\fy_0 \in \HH_0$.
Let $\vec\fy(s): (-T_-(\vec \fy_0), T_+(\vec \fy_0)) \to \HH_0$ be the solution of \eqref{eq:wmk} with initial data
$\vec\fy(0) = \vec\fy_0$.
By Lemma~\ref{l:1profile}, the solution $\vec \fy(s)$ does not scatter in forward or backward time and has threshold energy 
\EQ{
 \E( \vec \fy) = \E( \vec \psi) = 2 \E(\vec Q).
} 
Then by Proposition~\ref{p:cjk} there exists $\sigma \in[0, T_+(\vec\fy_0))$ such that
$ \bfd(\vec\fy(\sigma)) \leq \frac 12 \eps'$.
By the well-posedness theory for \eqref{eq:wmk}, $\vec\psi_\ell(s)$ is defined for $s \in [0, \sigma]$ for all $\ell$ sufficiently large and
\begin{align*} 
\vec\psi_\ell(\sigma) \to \vec\phi(\sigma) \mbox{ in } \HH_0. 
\end{align*}
Thus, $\bfd(\vec\psi_\ell(\sigma)) \to \bfd(\vec\phi(\sigma)) \leq \frac 12 \eps'$.

Define $t_\ell' := t_\ell + \nu(t_\ell)\sigma$. Then by the time translation and scaling symmetries of \eqref{eq:wmk} we have $\vec\psi(t_\ell') = \vec\psi_\ell(\sigma)_{\nu(t_\ell)}$, so 
\EQ{
\lim_{\ell \to \infty} \bfd(\vec\psi(t_\ell')) = \lim_{\ell\to\infty} \bfd(\vec\psi_\ell(\sigma)) \leq \frac 12 \eps'.
}
Our assumption \eqref{eq:int-bd-part-2} implies that for all $\ell$ sufficiently large, we have
$t_\ell \leq t_\ell' \leq a_{m_\ell}$.  Thus, by \eqref{eq:d-large-all} we conclude that 
$\bfd(\vec\psi(t_\ell')) \geq \eps'$, a contradiction.  Thus the upper bound of the lemma holds, and the proof is complete. 
\end{proof}

An immediate corollary of Lemma \ref{l:interval-bd} is the following. 

\begin{cor}\label{cor:k1}
	The modulated trajectory
	\EQ{
		\calK_1:=  \bigcup_{m \geq 1}\{  \vec \psi(t)_{1/\nu_m} \mid t \in I_m \} \label{eq:K1-comp}
	}
	is pre-compact in $\HH_0$. 
\end{cor}

Before concluding the proof of Proposition \ref{p:psi_t}, we record the following standard consequence of compactness of the trajectory. 

\begin{lem}\label{l:err-on-I}
	Given any $\de > 0$, there exists $R_0 > 0$ such that if $R_1 \geq R_0$, then for all $m = 2, 3, \ldots$ we have
	\EQ{
		\int_{I_m}|\Omega_{\nu_m R_1}(\vec\psi(t))|d t \leq \de \nu_m.
	}
\end{lem}
\begin{proof}
	By a change of variables, it suffices to show that
	\EQ{
		\left | \Omega_{R_1}(\vec\psi(t)_{1/\nu_m}) \right| \leq \de, 
		\quad \forall t \in I,
	}
for all $R_1$ sufficiently large.  By \eqref{eq:OmRest}
\begin{align*}
\left | \Omega_{R_1}(\vec\psi(t)_{1/\nu_m}) \right | 
\lesssim \cl E^\infty_{R_1}(\vec\psi(t)_{1/\nu_m}). 
\end{align*}
By the pre-compactness of the trajectory $\KK_1$, the result follows. 
\end{proof}

\subsection{Proof of Proposition \ref{p:psi_t}}. 
By \eqref{eq:d-conv-0}, there exists a sequence of times $c_{m_j}$ such that $\bfd(\vec \psi(c_{m_j})) \rar 0$ as $j \rar \infty$.  Let $\de > 0$, and let $R_0$ be as in Lemma \ref{l:err-on-I}.  For $m_j < m_k$, we define
\begin{align*}
R := R_0 \max_{m_j < m < m_k} \nu_m. 
\end{align*} 
By the virial identity, Lemma \ref{l:vir}, we have 
\begin{align*}
%\int \displaylimits_{I\cap (c_{m_j},c_{m_k})} \| \p_t \psi(t)\|^2 dt \leq 
\int_{c_{m_j}}^{c_{m_k}} \| \p_t \psi(t) \|^2 dt &\leq 
 \left | \ang{ \p_t \psi \mid \chi_R \, r \p_r \psi}(t) \Big |_{t = c_{m_j}}^{t = c_{m_k}} \right | + \sum_{m = m_j}^{m_k-1} \int_{b_{m}}^{a_{m+1}} | \Om_{R}(\vec \psi(t)) | \, dt \\
&+ \sum_{m=m_j}^{m_k-1}\int_{c_m}^{b_m}| \Om_{R}(\vec \psi(t)) | d t 
+ \sum_{m=m_j}^{m_k-1} \int_{a_{m+1}}^{c_{m+1}} | \Om_{R}(\vec \psi(t)) | d t.
\end{align*}
Replacing the left-hand side above with an integral over only the good intervals and using Lemma \ref{l:error-estim} to bound the right-hand side, we obtain
\begin{align}
\begin{split}\label{eq:virial_est}
\int \displaylimits_{I\cap (c_{m_j},c_{m_k})} \| \p_t \psi(t)\|^2 dt
 &\leq 
C_0 R \left [ \sqrt{\bfd(\vec \psi(c_{m_j}))} + \sqrt{\bfd(\vec \psi(c_{m_k}))} \right ] + \sum_{m = m_j}^{m_k-1} \int_{b_{m}}^{a_{m+1}} | \Om_{R}(\vec \psi(t)) | \, dt \\
&+ C_0 \sum_{m=m_j}^{m_k-1}\int_{c_m}^{b_m}\sqrt{\bfd(\vec\psi(t))}d t 
+ C_0 \sum_{m=m_j}^{m_k-1}  \int_{a_{m+1}}^{c_{m+1}}\sqrt{\bfd(\vec\psi(t))}d t.
\end{split}
\end{align}
The estimate \eqref{eq:err-mod-contr-1} implies that 
\begin{align*}
 C_0 \sum_{m=m_j}^{m_k-1}\int_{c_m}^{b_m}\sqrt{\bfd(\vec\psi(t))}d t 
 \leq \frac{1}{10} \sum_{m = m_j}^{m_k-1} \int_{b_{m}}^{a_{m+1}} \| \p_t \psi(t)\|^2 dt = \frac{1}{10} \int \displaylimits_{I\cap (c_{m_j},c_{m_k})} \| \p_t \psi(t)\|^2 dt,
\end{align*}
and the estimate \eqref{eq:err-mod-contr-2} implies that 
\begin{align*}
 C_0 \sum_{m=m_j}^{m_k-1}  \int_{a_{m+1}}^{c_{m+1}} \sqrt{\bfd(\vec\psi(t))}d t 
 \leq \frac{1}{10} \sum_{m = m_j}^{m_k-1} \int_{b_{m}}^{a_{m+1}} \| \p_t \psi(t)\|^2 dt = \frac{1}{10} \int \displaylimits_{I\cap (c_{m_j},c_{m_k})} \| \p_t \psi(t)\|^2 dt.
\end{align*}
By our choice of $R$ we have 
\begin{align*}
 \sum_{m = m_j}^{m_k-1} \int_{b_{m}}^{a_{m+1}} | \Om_{R}(\vec \psi(t)) | \, dt
 \leq \de  \sum_{m = m_j}^{m_k-1} \nu_m = \de |I \cap (c_{m_j},c_{m_k})|, 
\end{align*}
as well as $R \leq R_0  |I \cap (c_{m_j},c_{m_k})|$. The previous three estimates and \eqref{eq:virial_est} imply that 
\begin{align}
\fint \displaylimits_{I\cap (c_{m_j},c_{m_k})} \| \p_t \psi(t)\|^2 dt 
\leq \frac{5C_0 R_0}{3} \left [ \sqrt{\bfd(\vec \psi(c_{m_j}))} + \sqrt{\bfd(\vec \psi(c_{m_k}))} \right ] + \frac{5 \de}{3}. 
\end{align}
Since $\bfd (\vec \psi(c_{m_j})) \rar 0$ as $j \rar \infty$ and $\de$ is arbitrary, we conclude that 
\begin{align}\label{eq:nearly}
\limsup_{j \rar \infty} \limsup_{k \rar \infty} 
\fint \displaylimits_{I\cap (c_{m_j},c_{m_k})} \| \p_t \psi(t)\|^2_{L^2} dt = 0. 
\end{align}

We claim that
there exists a sequence of good intervals $I_{m_\ell} = [b_{m_\ell-1}, a_{m_\ell}]$ such that 
\begin{align}
\lim_{\ell \rar \infty} \fint \displaylimits_{I_{m_\ell}} \| \p_t \psi(t) \|^2_{L^2} dt = 0. \label{average_psit}
\end{align}
If not, then there exists $\de_0 > 0$ such that for all $m = 2, 3, \ldots$, we have 
$\int_{I_m} \| \p_t \psi(t) \|^2_{L^2} dt \geq \de_0 \nu_m.$ Summing this lower bound implies that for every $c_{m_j} < c_{m_k}$ we have 
\begin{align*}
\int  \displaylimits_{I\cap (c_{m_j},c_{m_k})} \| \p_t \psi(t) \|_{L^2}^2 dt \geq \de_0 |I \cap (c_{m_j},c_{m_k})|, 
\end{align*}
which contradicts \eqref{eq:nearly}. This proves our claim.  

We now conclude the proof of Proposition \ref{p:psi_t}. We 
denote the midpoint of our sequence of good intervals $[b_{m_\ell-1}, a_{m_\ell }]$ by $t_{m_\ell} := \frac 12(b_{m_\ell-1} + a_{m_\ell})$.
Note that for any $0<s_1 \le 1/2$ we have
\EQ{
b_{m_{\ell}-1} \le  t_{m_\ell} - \nu_{m_\ell} s_1  \leq t_{m_\ell} + \nu_{m_\ell} s_1  \le   a_{m_\ell}. \label{eq:tm-dist}
}
We define a sequence of solutions of \eqref{eq:wmk} via
\EQ{
\vec \psi_{m_\ell}(s) := \vec\psi(t_{m_\ell} + \nu_{m_\ell} s)_{1/\nu_{m_\ell}}\qquad \text{for }s \in [-s_1, s_1].
}
Then by a change of variables, \eqref{eq:tm-dist} and \eqref{average_psit} we have 
\begin{align}\label{eq:dt_zero}
\lim_{\ell \rar \infty} \int_{-s_1}^{s_1} \|\partial_s \psi_{m_\ell}(s)\|_{L^2}^2d s = 0.
\end{align}
By Corollary \ref{cor:k1} and extraction of a subsequence if necessary, $\vec \psi_{m_\ell}(0) \to \vec\fy_0$ in $\HH_0$.
Let $\vec\fy(s)$ be the solution of \eqref{eq:wmk}
with initial data $\vec\fy(0) = \vec\fy_0$. For $s_1 >  0$  sufficiently small, $\vec  \varphi(t)$ is defined on $[-s_1,s_1]$, and by the well-posedness theory for \eqref{eq:wmk} we have 
\begin{align*}
\lim_{\ell \rar \infty} \sup_{s \in [-s_1,s_1]}
\| \vec\psi_{m_\ell}(s) - \vec\fy(s) \|_{\HH_0} = 0. 
\end{align*}
By \eqref{eq:dt_zero} we conclude that 
\EQ{
\int_{-s_1}^{s_1}\|\partial_s \fy(s)\|_{L^2}^2d s = 0,
}
so $\vec \fy(s)$ is a harmonic map. The only degree-0 harmonic map is the constant map $\vec \fy = (0,0)$.  This contradicts the fact that $\cl E(\vec \varphi) = \cl E(\vec \psi) = 8\pi \neq 0$.  The proof of Proposition \ref{p:psi_t} is complete.  

\qed 

\subsection{Proof of Theorem~\ref{t:main}}
We first use Proposition \ref{p:psi_t} to prove $\vec \psi(t)$ converges to a pure two-bubble or anti two-bubble as $t \rar T_+$.  Let $\e >  0$ be sufficiently small. By Proposition \ref{p:psi_t} there exists a $T_0 \in (T_-, T_+)$ such that 
\begin{align}
\bfd(\vec \psi(t)) < \e, \quad \forall t \geq T_0. 
\end{align}
 We further assume that $\eps < \alpha_0$,
where $\alpha_0$ is the constant from Lemma~\ref{l:dpm}. Towards a contradiction, assume that $\vec \psi(t)$ alternates between being close to a pure two-bubble and anti two-bubble, i.e. that there exist $t_1, t_2 \geq T_0$, $t_1 < t_2$ such that
$\bfd_+(\vec\psi(t_1)) \leq \eps$ and $\bfd_-(\vec\psi(t_2)) \leq \eps$.
By Lemma~\ref{l:dpm} we have $\bfd_+(\vec\psi(t_2)) \geq \alpha_0$
and $\bfd_-(\vec \psi(t_1)) \geq \alpha_0$.  By continuity there exists $t_0 \in (t_1,t_2)$ such that $\bfd_+(\vec\psi(t_0)) = \bfd_-(\vec \psi(t_0))$.  But then again by Lemma \ref{l:dpm}, we conclude that $\bfd_+(\vec \psi(t_0)) = 
\bfd_-(\vec \psi(t_0)) > \al_0 > \e$.  This contradicts our definition of $T_0$ which proves the desired convergence.  Without loss of generality, we assume that $\bfd_+(\vec \psi(t)) \rar 0$ as $t \rar T_+$. 

We now prove finite time blow-up and asymptotics of the scales. By taking $T_0$ larger if necessary, we may assume that 
\begin{align*}
\bfd_+(\vec \psi(t)) < \e, \quad \forall t \geq T_0. 
\end{align*}
We note that as long as $\e >  0$ is sufficiently small, 
the modulation parameters $\lambda(t)$ and $\mu(t)$
are well-defined on $[T_0, T_+)$, and by Lemma \ref{l:modeq}
\begin{align}
\vec \psi(t) = \vec Q_{\la(t)} + \vec Q_{\mu(t)} + o_{\HH_0}(1), \quad \mbox{ as } t \rar T_+. 
\end{align}
Let $\e_0 > 0$ and choose $\e$ smaller if necessary so that the conclusions of Proposition \ref{prop:modulation} hold. 
Let $\zeta(t)$ be as in~\eqref{eq:zetadef} with $L$ and $M$ chosen as in Remark \ref{r:param} so that $\zeta(t) \sim \la(t) |\log \la(t)/\mu(t) |$.
By rescaling if necessary, we can assume that $\mu(T_0) = 1$.

Since $\bfd_+(\vec \psi(t)) \rar 0$ as $t \rar T_+$, there exists a sequence of times $\tau_n \to T_+$ such that
\EQ{
\frac{d}{dt} \Big|_{t = \tau_n}\Big(\frac{\zeta(t)}{\mu(t)}\Big) \leq 0.
}
Then there exist times $t_1 \leq t_0 =: \tau_n$ and $t_2 \leq t_1$ satisfying the conclusions of Proposition \ref{prop:modulation}.
By our choice of $T_0$ and \eqref{eq:d-t1-t2} we have $t_1 \leq T_0$ for every $t_0 = \tau_n$.  From the proof of Proposition \ref{prop:modulation} we recall that $\mu(t) \in [1/2,2]$ on $[T_0,\tau_n]$, and the function $$\xi(t) = -b(t) + \zeta(t)^{\frac 12}$$ satisfies for all $t \in [T_0,\tau_n]$
	\begin{align}\label{eq:fin_est}
	\zeta(t)^{\frac 12} \leq 2 \xi(t) \leq 20 \zeta(t)^{\frac 12}. \quad 
	\xi'(t) \leq -\frac{1}{2} \zeta^{-\frac 12}(t) \xi(t). 
	\end{align}  
Since $\tau_n \rar T_+$, these same bounds hold on $[T_0,T_+)$.  From \eqref{eq:lamud1}, \eqref{eq:b-bound}, \eqref{eq:fin_est} and the fact that $\bfd_+(\vec \psi(t)) \rar 0$ as $t \rar T_+$ we can conclude  
\begin{align*}
\zeta(t) \rar 0 \mbox{ and } \xi(t) \rar 0 \mbox{ as } t \rar T_+. 
\end{align*} 
From \eqref{eq:fin_est} we see that $\xi(t)$ is positive on $[T_0, T_+)$ and satisfies $\xi'(t) \leq -1/4$.  Since $\xi(t) \rar 0$ as $t \rar T_+$, we conclude that $T_+ < \infty$ which proves finite time blow-up. 

We now turn to the asymptotics of the scales.  The estimates \eqref{eq:fin_est} and \eqref{eq:mu'} imply that 
	\begin{align*}
	&\int_{T_0}^{T_+} |\mu'| dt \lesssim \int_{T_0}^{T_+}
	\zeta(t)^{\frac 12} dt \lesssim 
	\int_{T_0}^{T_+} \xi(t) dt 
	\lesssim \int_{T_0}^{T_+} \zeta(t)^{\frac 12} (-\xi'(t)) dt 
	\lesssim \int_{T_0}^{T_+} (-\xi'(t)) dt 
	\lesssim 1
	\end{align*}
Thus, $\mu(t)$ converges to some $\mu_0 \in [1/2,2]$. For the decay of $\la(t)$, we first recall that by \eqref{eq:fin_est} we have $\xi'(t) \lesssim -1$.  By Lemma \ref{p:modp}, we see that 
\begin{align*}
|\xi'(t)| \lesssim |b'(t)| + \zeta^{-1/2} |\zeta'(t)| \lesssim 1.
\end{align*}
Thus, there exists $C > 0$ such that 
\begin{align*}
-C \leq \xi'(t) \leq -\frac{1}{C}, \quad \forall t \in [T_0, T_+),
\end{align*}
which implies 
\begin{align}
\frac{1}{C} (T_+ - t) \leq \xi(t) \leq C(T_+ - t), \quad \forall t \in [T_0, T_+).
\end{align}
Since $\xi(t) \sim \zeta(t)^{\frac{1}{2}} \sim [\la(t) |\log \la(t)|]^{\frac 12}$ on $[T_0, T_+)$, we conclude that 
\begin{align*}
\la(t) |\log \la(t)| \sim (T_+ - t)^2, \quad \mbox{ as } t \rar T_+
\end{align*}
as desired. 

Finally, we show that $\vec \psi$ scatters backward in time. Suppose not. 
Then $-\infty < T_- < T_+ < \infty$, and $\int_{T_-}^{T_+} \sqrt{\bfd(\vec \psi(t))} \, dt < \infty$ by what we have shown up to this point. The virial identity \eqref{eq:vir}, \eqref{eq:virial-end} and the fact that $\bfd(\vec \psi(t)) \rar 0$ as $t \rar T_{\pm}$ imply that 
\begin{align*}
\int_{T_-}^{T_+} \| \p_t \psi(t) \|_{L^2}^2 \, dt 
\leq \int_{T_-}^{T_+} |\Om_R(\psi(t))| \, dt, \quad \forall R > 0. 
\end{align*}
For all $t \in (T_-,T_+)$, we have 
$|\Om_R(\vec \psi(t))| \leq C_0 \sqrt{\bfd(\vec \psi(t)} \in L^1(T_-,T_+)$ and
$\lim_{R \rar \infty} \Om_R(\vec \psi(t)) = 0$.  Thus, by the dominated convergence theorem
\begin{align*}
\int_{T_-}^{T_+} \| \p_t \psi(t) \|_{L^2}^2 \, dt  = 0. 
\end{align*}
We conclude that $\vec \psi$ is a degree-0 harmonic map, i.e $\vec \psi = (0,0)$.  This contradicts $\cl E(\vec \psi) = 8\pi$ and finishes the proof. 

\qed

\section{Construction of a Minimal Blow-up Solution}

\subsection{Proof of Theorem \ref{t:main2}} Let $T > 0$ be small (to be determined later).  We define a function $\ell(t) : [0,T) \rar [0,\infty)$ implicitly by the relation 
\begin{align}\label{eq:ell_def}
\ell(t) |\log \ell(t)| = 2t^2, \quad t \in (0,T),
\end{align} 
with $\ell(0) = 0$. By elementary calculus it is easy to see that 
$\ell \in C^\infty(0,T),$ $\ell$ is increasing on $[0,T)$ and
\begin{align}\label{eq:ell_deriv}
\ell'(t) |\log \ell(t)| = 4t \left [ 1 + O(|\log \ell(t)|^{-1}) \right ].
\end{align}
In particular, this implies that 
\begin{align}
\frac{\ell(t)}{\ell'(t)} &= \frac{t}{2}  \left [ 1 + O(|\log \ell(t)|^{-1}) \right ], \label{eq:ellrelation1} \\
\frac{\ell(t)}{(\ell'(t))^2|\log \ell(t)|} &= \frac{1}{8} + O(|\log \ell(t)|^{-1}) \label{eq:ellrelation2}. 
\end{align}

Let $t_n$ be a sequence in $(0,T)$ which is monotonically decreasing to 0.  We define a sequence of initial data at time $t = t_n$ via 
\begin{align}
\psi_{0,n} &:=  Q_{\ell(t_n)} -  Q, \\
\psi_{1,n} &:= -\ell'(t_n) \La Q_{\underline{\ell(t_n)}} \chi_{\sqrt{R_n \ell(t_n)}},
\end{align} 
where $\chi$ is now a sharp cutoff, $\chi(r) = 1$ for $0 \leq r \leq 1$ and $\chi(r) = 0$ for $r > 1$, and $R_n > 0$ is chosen so that 
\begin{align}
\cl E(\psi_{0,n}, \psi_{1,n}) = 2\cl E(\vec Q).  
\end{align}
We first show that $R_n$ exists and that $R_n + R_n^{-1}$ is bounded. 

\begin{lem}\label{l:R_nlem}
For $T > 0$ sufficiently small, for all $n$ there exists $R_n > 0$ such that the pair of initial data $(\psi_{0,n}, \psi_{1,n})$ defined above satisfies 
$\cl E(\psi_{0,n}, \psi_{1,n}) = 2 \cl E(Q)$. Moreover, there exists $R > 0$ such that  
\begin{align*}
\frac{1}{R} \leq R_n \leq R. 
\end{align*}
\end{lem}

\begin{proof}
We expand the nonlinear energy and obtain (see Section 3 of \cite{JL})
\begin{align*}
2 \cl E(Q) &= \cl E(\psi_{0,n}, \psi_{1,n}) \\
&= 2 \cl E(Q) + \int_0^\infty \psi_{1,n}^2 r dr 
- 4 \int_0^\infty \La Q_{\ell(t_n)} (\La Q)^3 \frac{dr}{r} +
2 \int_0^\infty (\La Q_{\ell(t_n)})^2 (\La Q)^2 r \frac{dr}{r},
\end{align*}
so that 
\begin{align}\label{eq:psi1_n}
\int_0^\infty \psi_{1,n}^2 r dr = 
 4 \int_0^\infty \La Q_{\ell(t_n)} (\La Q)^3 \frac{dr}{r} -
2 \int_0^\infty (\La Q_{\ell(t_n)})^2 (\La Q)^2 \frac{dr}{r}. 
\end{align}
By a change of variables, the left side of \eqref{eq:psi1_n} is readily computed to be
\begin{align}
\begin{split}\label{eq:s53}
\int_0^\infty \psi_{1,n}^2 r dr &= (\ell'(t_n))^2 \int_0^{\sqrt{R_n/\la_n}} |\La Q|^2 r dr \\
&= 2 (\ell'(t_n))^2 \left [ 
\log \Bigl (1 + \frac{R_n}{\ell(t_n)} \Bigr ) + \frac{1}{1 + R_n/\ell(t_n)} 
- 1
\right ].  
\end{split}
\end{align}
For the right side of \eqref{eq:psi1_n}, we first consider the expression
\begin{align}
 4 \int_0^\infty \La Q_{\s} (\La Q)^3 \frac{dr}{r} &=
 64 \s \int_0^\infty \frac{r^3}{(\s^2 + r^2)(1 + r^2)^3} dr \\
 &= 64 \s \int_0^\s \frac{r^3}{(\s^2 + r^2)(1 + r^2)^3} dr dr
 + 64 \s \int_\s^\infty \frac{r^3}{(\s^2 + r^2)(1 + r^2)^3} dr    
\end{align}
where for brevity we have set $\s = \ell(t_n)$. Now 
\begin{align*}
\int_0^\s \frac{r^3}{(\s^2 + r^2)(1 + r^2)^3} dr \lesssim \int_0^\s r dr \lesssim \s^2.  
\end{align*}
Since $\frac{1}{\s^2 + r^2} = \frac{1}{r^2} + \frac{\s^2}{(\s^2 + r^2)r^2}$, we have 
\begin{align*}
\int_\s^\infty \frac{r^3}{(\s^2 + r^2)(1 + r^2)^3} dr
&= \int_\s^\infty \frac{r}{(1 + r^2)^3} dr 
+ \s^2 \int_\s^\infty \frac{r}{(1+r^2)^3(\s^2 + r^2)} \\
&= \frac{1}{4} + O(\s^2 |\log \s|). 
\end{align*}
We conclude that 
\begin{align}\label{eq:s51}
 4 \int_0^\infty \La Q_{\ell(t_n)} (\La Q)^3 \frac{dr}{r} = 16 \ell(t_n) \Bigl [1 + 
 O(\ell(t_n)^2 |\log \ell(t_n)|)\Bigr ].  
\end{align}
By a similar argument we also obtain 
\begin{align}\label{eq:s52}
 \int_0^\infty (\La Q_{\ell(t_n)})^2 (\La Q)^2 \frac{dr}{r} \lesssim 
 \ell(t_n)^2 |\log \ell(t_n)|. 
\end{align}
Combining \eqref{eq:psi1_n}, \eqref{eq:s53}, \eqref{eq:s51} and \eqref{eq:s52} we obtain 
\begin{align}
\log \Bigl (1 + \frac{R_n}{\ell(t_n)} \Bigr ) + \frac{1}{1 + R_n/\ell(t_n)}
- 1
= \frac{8 \ell(t_n)}{(\ell'(t_n))^2} \Bigl [ 1 + O(\ell(t_n)|\log \ell(t_n)|) \Bigr ].
\end{align} 
Thus by \eqref{eq:ellrelation2} 
\begin{align}\label{elllast} 
\log \Bigl (1 + \frac{R_n}{\ell(t_n)} \Bigr ) + \frac{1}{1 + R_n/\ell(t_n)}
- 1
= |\log \ell(t_n)| \Bigl [ 1 + O(|\log \ell(t_n)|^{-1}) \Bigr ]. 
\end{align}
The function $f(x) = \log (1 + x) + \frac{1}{1+x} - 1$ is continuous, is equal to 0 when $x = 0$ and tends to $\infty$ as $x \rar \infty$.  Thus, by the intermediate value theorem and as long as $T$ is sufficiently small, there exist $R_n$ satisfying \eqref{elllast} for all $n$.  From \eqref{elllast} we see that $R_n / \ell(t_n) \rar \infty$ as $n \rar 0$. Rearranging the previous expression yields 
\begin{align*}
\log R_n = 1 - \log \Bigl (1 + \frac{\ell(t_n)}{R_n} \Bigr ) - \frac{1}{1 + R_n/\ell(t_n)} + O(1).  
\end{align*}
Since $R_n / \ell(t_n) \rar \infty$, the right side of the previous expression is bounded.  This concludes the proof of the lemma.  
\end{proof}

Let $\vec \psi_n(t)$ denote the solution to \eqref{eq:wmk} with initial data 
$\vec \psi_n(t_n) = (\psi_{0,n}, \psi_{1,n})$.  We remark that the previous computations yield
\begin{align} \label{eq:psi_1n_L2}
\| \psi_{1,n} \|_{L^2}^2 
= 16 \ell(t_n) \Bigl [
1 + O(\ell(t_n)^2 |\log \ell(t_n)| )
\Bigr ]
\end{align}
Therefore, as long as $T > 0$ is small, for all $t$ in a neighborhood of $t_n$ the modulation parameters $\la_n(t)$ and $\mu_n(t)$ are well defined for $\vec \psi_n(t)$ and 
\begin{align*}
\la_n(t_n) = \ell(t_n), \quad \mu_n(t_n) = 1.  
\end{align*}
If we denote $g_n(t) := \psi_n(t) - (Q_{\la_n(t)} - Q_{\mu_n(t)})$ and $\dot g_n(t) = \p_t \psi_n(t)$, then 
\begin{align*}
g_n(t_n) = 0, \quad \dot g_n(t_n) = -\ell'(t_n) \La Q_{\underline{\ell(t_n)}} \chi_{\sqrt{R_n \ell(t_n)}}.  
\end{align*}
Let $\zeta_n(t)$ and $b_n(t)$ be defined as in \eqref{eq:zetadef}, \eqref{eq:bdef} for each $\vec \psi_n$, i.e. 
\EQ{ 
	\zeta_n(t) &:= 2 \la_n(t) |\log (\la_n(t)/\mu_n(t))| - \langle \chi_{M \sqrt{\la_n(t) \mu_n(t)}}\Lambda Q_{\uln{\lambda_n(t)}} \mid g_n(t)\rangle, \\
	b_n(t)&:= - \ang{ \chi_{M \sqrt{\la_n(t)\mu_n(t)}} \La Q_{\underline{\la_n(t)}}  \mid \dot g_n(t)}  - \ang{ \dot g_n(t) \mid \A_0( \la_n(t) ) g_n(t)}.
}
\begin{cor}\label{l:zeta'}
As long $M > 0$ is sufficiently large we have 
\begin{align}
b_n(t_n) = 8 t_n \left [ 1 + O ( |\log \ell(t_n)|^{-1}) \right ], 
\end{align}
\end{cor}

\begin{proof}
%We compute
%\begin{align*}
%\frac{d}{dt} \frac{\zeta_n(t)}{\mu_n(t)} \Big |_{t = t_n} = 
%\zeta'(t_n) - \mu_n'(t_n) \zeta_n(t_n). 
%\end{align*}
%By \eqref{eq:mu'}
%\begin{align}\label{eq:s55}
%|\mu_n'(t_n)\zeta_n(t_n)| \lesssim \zeta_n(t_n)^{3/2} \lesssim t_n^{3/2} 
%\end{align}
%For $\delta > 0$ small
%\begin{align*}
%|\zeta_n'(t_n) - b_n(t_n)| \leq \de \zeta_n(t_n)^{1/2} 
%\end{align*}
%by \eqref{eq:kala'}.  
Let $M^2$ be larger than $R$ given by Lemma \ref{l:R_nlem}.  Then by \eqref{eq:psi_1n_L2} and \eqref{eq:ellrelation1} we have 
\begin{align*}
b_n(t_n) &=  -\ang{ \chi_{M \sqrt{\ell_n(t_n)}} \La Q_{\underline{\ell(t_n)}} \mid \dot g_n(t_n)} \\
&= \frac{1}{\ell'(t_n)} \| \psi_{1,n}(t_n) \|_{L^2}^2 \\
&= \frac{16 \ell(t_n)}{\ell'(t_n)} \Bigl [ 1 + O(\ell(t_n)^2 |\log \ell(t_n)|) \Bigr ] \\
&= 2 \ell'(t_n) |\log \ell(t_n)| \Bigl [ 1 + O( |\log \ell(t_n)|^{-1}) \Bigr ] \\ 
&= 8 t_n \Bigl [ 1 + O( |\log \ell(t_n)|^{-1}) \Bigr ].
\end{align*}
\end{proof}

Let $L = L_0 > 0$, $M = M_0 > 0$ and $\eta_1 > 0$ be chosen so that the conclusions of Proposition \ref{p:modp2} hold with $\de = \frac{1}{2018}$ and so that the conclusion of Corollary \ref{l:zeta'} holds.  
Let 
\begin{align*}
T_n' = \sup \Bigl \{ t \in [t_n,T] \mid \vec \psi_n(s) \mbox{ exists, } \bfd_+(\vec \psi_n(s) < \eta_1, \mbox{ and } \mu_n(s) \in (1/2, 2) \quad \forall s\in[t_n,t]  \Bigr \}.
\end{align*}
We will show that $T_n' = T$ as long as $T$ is sufficiently small.  

Let $t \in [t_n,T'_n]$.  By \eqref{eq:kala'}, \eqref{eq:b-bound} and our assumption on $\mu_n(t)$ 
\begin{align*}
\zeta_n(t) &= \zeta_n(t_n) + \int_{t_n}^t \zeta_n'(s) ds \\
&\leq \zeta_n(t_n) + \int_{t_n}^t [|b_n(s)| +  \zeta_n(s)^{1/2}] ds  \\
&\leq \zeta_n(t_n) + 6 \int_{t_n}^t \zeta_n(s)^{1/2} ds. 
\end{align*}
Thus, 
\begin{align*}
\zeta_n(t) \leq 2 \zeta_n(t_n) + 36 (t - t_n)^2 
\end{align*}
Since $\zeta_n(t_n) = 2 \ell(t_n) |\log \ell(t_n)| = 4 t_n^2$, we conclude that 
\begin{align}\label{eq:zetabound}
\zeta_n(t) \leq 148 t^2.   
\end{align}
Then by \eqref{eq:bound-on-l} 
\begin{align}\label{eq:lambound}
\la_n(t) |\log \la_n(t)| \leq 75 t^2. 
\end{align}
We now consider $\mu_n(t)$.  By the fundamental theorem of calculus, \eqref{eq:mu'}, \eqref{eq:bound-on-l} and \eqref{eq:zetabound} there exists an absolute constant $\beta > 0$ such that 
\begin{align}\label{eq:mubound}
|\mu_n(t) - 1| \leq \beta t^2.   
\end{align}
By \eqref{eq:gH}, \eqref{eq:lambound} and our assumption on $\mu_n$ there exists a constant $\al > 0$ such that 
\begin{align}\label{s56}
\| \vec \psi_n(t) - (\vec Q_{\la_n(t)} - Q_{\mu_n(t)}) \|^2_{\HH_0} \leq \al t^2. 
\end{align}
In summary, we have shown that 
\begin{align}
\la_n(t) |\log \la_n(t)| &\leq 75 t^2, \\
|\mu_n(t) - 1 | &\leq \beta t^2, \\
\bfd_+(\vec \psi_n(t)) &\leq (\al + 150) t^2. 
\end{align}
By a continuity argument, it follows that $T_n' = T$ provided that  $\vec \psi_n(t)$ is defined on $[t_n,T]$.  We now prove this fact. 

Let $t \in [t_n,T'_n]$.  By Corollary \ref{l:zeta'} and \eqref{eq:b'lb} we have
\begin{align}\label{eq:b_nlower}
b_n(t) \geq \frac{1}{2} \Bigl (8 - \frac{1}{2018} \Bigr ) (t - t_n) + 8 t_n 
\Bigl[1 + O(|\log \ell(t_n)|^{-1}) \Bigr ] 
\geq 3 (t - t_n) + 5 t_n \geq 3 t. 
\end{align}
By \eqref{eq:kala'}, \eqref{eq:b_nlower} and \eqref{eq:zetabound} we have
\begin{align*}
\zeta_n'(t) \geq b_n(t) - \frac{2}{2018} \zeta_n^{1/2}(t) \geq 3 t - \frac{2 \sqrt{148}}{2018} t \geq 2 t.
\end{align*}
By the fundamental theorem of calculus we conclude that 
\begin{align}
\zeta_n(t) \geq \zeta_n(t_n) + t^2 - t_n^2 
= 4 t_n^2 + t^2 - t_n^2 \geq t^2.  
\end{align}
By \eqref{eq:bound-on-l}, the previous implies that 
\begin{align}\label{lamlower}
\la_n(t) |\log \la_n(t)| \geq \frac{1}{3} t^2.
\end{align}
The estimates \eqref{lamlower}, \eqref{eq:lambound} and \eqref{s56} imply
\begin{align}\label{intermediate}
\inf_{\substack{\mu \in [1/2,2] \\
		\la |\log \la| \in [t^2/3,75t^2]}}  \| \vec \psi_n(t) - (Q_\la - Q_\mu) \|_{\HH_0}^2 \leq\al t^2
\end{align}
on $[t_n,T_n']$.  By Corollary A.4 of \cite{JJ-AJM} we conclude that the interval of existence of $\vec \psi_n$ strictly includes $[t_n,T_n']$ as long as long as $T$ is small. Thus, we have proved that $T_n' = T$. 

The bound \eqref{intermediate} also implies that we may pass to a weak limit and obtain our desired blow-up solution.  Indeed, for any $T_0 < T$ 
\begin{align}
\inf_{\substack{\mu \in [1/2,2] \\
		\la |\log \la| \in [T_0^2/3,75T^2]}}  \| \vec \psi_n(t) - (Q_\la - Q_\mu) \|_{\HH_0}^2 \leq\al T^2, \quad \forall t \in [T_0,T], \forall n. 
\end{align}
By Corollary A.6 of \cite{JJ-AJM} we can conclude, after shrinking $T$ and extracting subsequences if necessary, there exists a solution $\vec \psi_c(t)$ defined on $(0,T]$ such that $\vec \psi_n(t) \rightharpoonup_n \vec \psi_c(t)$ for all $t \in (0,T]$.  By weak convergence and \eqref{intermediate} 
\begin{align*}
\inf_{\substack{\mu \in [1/2,2] \\
		\la |\log \la| \in [t^2/3,75t^2]}}  \| \vec \psi(t) - (Q_\la - Q_\mu) \|_{\HH_0}^2 \leq \al t^2
\end{align*}
Thus, $\vec \psi_c$ is the desired solution with blow-up time $T_- = 0$. 
\qed

\bibliographystyle{plain}
\bibliography{researchbib}
\bigskip

\centerline{\scshape Casey Rodriguez}
\smallskip
{\footnotesize
 \centerline{Department of Mathematics, Massachusetts Institute of Technology}
\centerline{77 Massachusetts Ave, 2-246B, Cambridge, MA 02139, U.S.A.}
\centerline{\email{caseyrod@mit.edu}}
}

\end{document}